\documentclass[a4paper,11pt]{amsart}
\usepackage[english]{babel}
\usepackage{amsmath,amscd,amssymb,amsthm,amsxtra,enumerate}
\usepackage{mathtools}
\usepackage{color,graphicx}
\usepackage{epsfig,epstopdf}
\usepackage[margin=1.25in]{geometry}
\usepackage{enumerate}
\usepackage{hyperref}

\newcommand{\norm}[1]{\left\Vert#1\right\Vert}

\newcommand{\R}{\mathbb{R}}

\newcommand{\lU}{\tilde{U}}

\newcommand{\StarQQPlus}{(Q^+_1)^*}

\newcommand\reallywidehat[1]{\arraycolsep=0pt\relax%
\begin{array}{c}
\stretchto{
  \scaleto{
    \scalerel*[\widthof{\ensuremath{#1}}]{\kern-.5pt\bigwedge\kern-.5pt}
    {\rule[-\textheight/2]{1ex}{\textheight}} %WIDTH-LIMITED BIG WEDGE
  }{\textheight} % 
}{0.5ex}\\           % THIS SQUEEZES THE WEDGE TO 0.5ex HEIGHT
#1\\                 % THIS STACKS THE WEDGE ATOP THE ARGUMENT
\rule{-1ex}{0ex}
\end{array}
}

\DeclareMathOperator{\diam}{diam}
\DeclareMathOperator{\dist}{dist}
\DeclareMathOperator{\dvie}{div}
\DeclareMathOperator{\Dom}{Dom}

\DeclareMathOperator*{\osc}{osc}

\DeclareMathOperator{\dive}{div}

\newtheorem{thm}{Theorem}[section]
\newtheorem{prop}[thm]{Proposition}
\newtheorem{cor}[thm]{Corollary}
\newtheorem{lem}[thm]{Lemma}

\theoremstyle{definition}

\newtheorem{rem}[thm]{Remark}

\numberwithin{equation}{section}

\allowdisplaybreaks

\author[A. Biswas]{Animesh Biswas}

\address{Department of Mathematics\\
Iowa State University\\
396 Carver Hall, Ames\\
IA 50011, United States of America}
\email{animbisw@iastate.edu, stinga@iastate.edu}

\author[P. R. Stinga]{Pablo Ra\'ul Stinga}

\thanks{The authors were partially supported by
Simons Foundation grant 580911 and by grant MTM2015-66157-C2-1-P
(MINECO/FEDER) from Government of Spain}



\begin{document}

%%%%%%%%%%%%%%%%%%%%%%%%%%%%%%%%%%%%%%%%%%%%%%%%%%%%%%
\title[Regularity for master equations]{Regularity estimates for nonlocal space-time master equations in bounded domains}
%%%%%%%%%%%%%%%%%%%%%%%%%%%%%%%%%%%%%%%%%%%%%%%%%%%%%%

%%%%%%%%%%%%%%%%%%%%%%%%%%%%%%%%%%%%%%%%%%%%%%%%%%%%%%
\begin{abstract}
We obtain sharp parabolic interior and global Schauder estimates for solutions to nonlocal space-time
master equations $(\partial_t +L)^su = f$ in $\R \times \Omega$, where $L$ is an elliptic operator in divergence form,
subject to homogeneous Dirichlet and Neumann boundary conditions.
In particular, we establish the precise behavior of solutions near the boundary.
Along the way, we prove a characterization of the correct intermediate parabolic H\"older spaces in the spirit of Sergio Campanato.
\end{abstract}
%%%%%%%%%%%%%%%%%%%%%%%%%%%%%%%%%%%%%%%%%%%%%%%%%%%%%%

\maketitle

%%%%%%%%%%%%%%%%%%%%%%%%%%%%%%%%%%%%%%%%%%%%%%%%%%%%%%
\section{Introduction}\label{sec:sec_1}
%%%%%%%%%%%%%%%%%%%%%%%%%%%%%%%%%%%%%%%%%%%%%%%%%%%%%%

We study interior and global Schauder estimates for solutions $u$ to nonlocal space-time master equations 
\begin{equation}\label{eq:Hsuf}
H^su\equiv(\partial_t-\dive(A(x)\nabla_x))^su=f\qquad\hbox{in}~\R\times\Omega
\end{equation}
where $f=f(t,x):\R\times\Omega\to\R$ is a given datum and $\Omega\subset\R^n$, $n\geq1$, is a bounded Lipschitz domain.
Here $H^s=(\partial_t-\dive(A(x)\nabla_x))^s$ is the fractional power of order $0<s<1$ of the parabolic operator
$H=\partial_t-\dive(A(x)\nabla_x)$. The coefficients in \eqref{eq:Hsuf}
are symmetric $A(x) = (A^{ij}(x)) = (A^{ji}(x))$ for $i, j = 1,\ldots,n$, bounded and measurable, and satisfy the uniform ellipticity condition
 $\Lambda_1 |\xi|^2 \leq A(x) \xi \cdot \xi \leq \Lambda_2 |\xi|^2$, for all $\xi \in \R^n$ and almost every $x \in \Omega$, for some ellipticity constants $0 < \Lambda_1 \leq \Lambda_2$.  The problem is subject to either homogeneous Dirichlet or Neumann boundary conditions, that is,
$$u=0\qquad\hbox{or}\qquad\partial_Au\equiv A(x)\nabla_xu\cdot\nu=0\qquad\hbox{on}~\R\times\partial\Omega$$
where $\nu$ is the exterior unit normal to $\partial\Omega$.

Master equations as in \eqref{eq:Hsuf} arise in several different physical applications
such as the phenomenon of osmosis in semipermeable membranes, in diffusion models for biological invasions,
in financial mathematics, in the Signorini problem of elasticity in heterogeneous materials and also in probability,
among others, see, for instance, \cite{Allen-Caffarelli-Vasseur, ACM, Danielli, BRR, Caffarelli-Silvestre-Master, DL, Stinga-Torrea-SIAM}
and references therein. All these phenomena are governed by a master equation given in generalized form as
 \begin{equation}\label{eq:gen_master}
     \int_{\R^n}\int^\infty_0 (u(t-\tau,z)-u(t,x))K(t,x,\tau,z) \,d \tau  \, dz= f(t,x)
 \end{equation}
 for $t \in \R$ and $x \in \R^n$, for some kernel $K$.
 
In terms of regularity, Caffarelli and Silvestre proved H\"older estimates of viscosity solutions to \eqref{eq:gen_master} with bounded right hand side, see \cite{Caffarelli-Silvestre-Master}. They assumed conditions on the kernel $K$ that ensure that \eqref{eq:gen_master} is an equation of fractional
order $s$ in time and $2s$ in space. On the other hand, in \cite{Stinga-Torrea-SIAM}, Stinga and Torrea studied the problem $(\partial_t - \Delta)^s u = f$, for $0<s<1$, which is the most basic form of a master equation. The systematic study of weak solutions to master equations as in \eqref{eq:Hsuf} was initiated in
 \cite{B-DLC-S}, where a precise definition of the fractional power operator $H^s$ is given. In particular, in \cite{B-DLC-S}, the pointwise formula
 and weak formulation for $H^su$ are obtained, see \eqref{eq:pointwise_formula} below.
In addition, it is shown that nonnegative solutions to $H^su=0$ satisfy interior and boundary parabolic Harnack inequalities and H\"older estimates.
 
We continue the development of the regularity theory for \eqref{eq:Hsuf}. We obtain interior and boundary parabolic Schauder estimates
for the solution $u$ to \eqref{eq:Hsuf} in the cases when $f$ is H\"older continuous, see Theorems \ref{thm:interiorholder}, \ref{thm:1.3_f_nonzero}, \ref{thm:1.3_f_zero} and \ref{thm:boundaryNeumannHolder}, and also when $f$ is just $L^p$ integrable, for $p$ large depending on $s$ and $n$, see Theorems \ref{thm:interiorLp} and \ref{thm:1.5}.
For these results, the coefficients $A(x)$ are assumed to be at least continuous. In particular, to establish the boundary behavior of solutions we need to perform a
precise asymptotic analysis of half space solutions. Furthermore, in order to apply our method, which 
is based on energy estimates and compactness arguments and is nonlinear in nature,
we need to prove a characterization of parabolic H\"older spaces in the spirit of Campanato.

In the following we present our main results. From now on, we fix $T_1<0<T_2$ and we call $I=(T_1,T_2)$.
We refer the reader to Section \ref{sec:Campanato_statement} for the definition of parabolic H\"older spaces.
In the first two statements, we present the interior regularity when $f$ is parabolically H\"older continuous in $I \times \Omega$
and when $f$ is in $L^p (I \times \Omega)$, respectively, under precise continuity assumptions on $A(x)$.
Interior regularity in both cases does not depend on the prescribed boundary conditions nor on the regularity of the boundary.

\begin{thm}[Interior regularity for $f$ H\"older]\label{thm:interiorholder}
Let $0<\alpha<1$ and suppose that $f\in C^{\alpha/2,\alpha}_{t,x}(I\times\Omega)$.
Let $u\in\Dom(H^s)$ be a weak solution to \eqref{eq:Hsuf} such that $u=0$ or $\partial_Au=0$ on $\R\times\partial\Omega$.
\begin{enumerate}[$(i)$]
\item Assume that $0<\alpha+2s<1$ and that $A(x)$ is continuous in $\Omega$. Then
$$u\in C^{(\alpha+2s)/2,\alpha+2s}_{t,x,\mathrm{loc}}(I\times\Omega)$$
and for any open subset $K\subset\subset I\times\Omega$
we have the estimate
$$\|u\|_{C^{(\alpha+2s)/2,\alpha+2s}_{t,x}(K)}\leq C\big(\|u\|_{\Dom(H^s)}+\|f\|_{C^{\alpha/2,\alpha}_{t,x}(I\times\Omega)}\big).$$
\item Assume that $1<\alpha+2s<2$ and that $A(x)\in C^{0,\alpha+2s-1}(\Omega)$. Then
$$u\in C^{(\alpha+2s)/2,1+(\alpha+2s-1)}_{t,x,\mathrm{loc}}(I\times\Omega)$$
and for any open subset $K\subset\subset I\times\Omega$
we have the estimate
$$\|u\|_{C^{(\alpha+2s)/2,1+(\alpha+2s-1)}_{t,x}(K)}\leq C\big(\|u\|_{\Dom(H^s)}+\|f\|_{C^{\alpha/2,\alpha}_{t,x}(I\times\Omega)}\big).$$
\end{enumerate}
The constants $C>0$ above depend only on $s,\alpha,K,I\times\Omega$ and the modulus of continuity of $A(x)$.
\end{thm}

\begin{thm}[Interior regularity for $f$ in $L^p$]\label{thm:interiorLp}
Suppose that $f\in L^p(I\times\Omega)$ for some $2\leq p<\infty$.
Let $u\in\Dom(H^s)$ be a weak solution to \eqref{eq:Hsuf} such that $u=0$ or $\partial_Au=0$ on $\R\times\partial\Omega$.
\begin{enumerate}[$(i)$]
\item Assume that $(n+2)/(2s)<p<(n+2)/(2s-1)_+$ and that $A(x)$ is continuous in $\Omega$. Then
$$u\in C^{\alpha/2,\alpha}_{t,x,\mathrm{loc}}(I\times\Omega)$$
where $\alpha=2s-(n+2)/p\in(0,1)$ and, for any open subset $K\subset\subset I\times\Omega$,
we have the estimate
$$\|u\|_{C^{\alpha/2,\alpha}_{t,x}(K)}\leq C\big(\|u\|_{\Dom(H^s)}+\|f\|_{L^p(I\times\Omega)}\big).$$
\item Assume that $s>1/2$, $p>(n+2)/(2s-1)$ and that $A(x)\in C^{0,\alpha}(\Omega)$ for $\alpha=2s-(n+2)/p-1\in(0,1)$. Then
$$u\in C^{(1+\alpha)/2,1+\alpha}_{t,x,\mathrm{loc}}(I\times\Omega)$$
and for any open subset $K\subset\subset I\times\Omega$
we have the estimate
$$\|u\|_{C^{(1+\alpha)/2,1+\alpha}_{t,x}(K)}\leq C\big(\|u\|_{\Dom(H^s)}+\|f\|_{L^p(I\times\Omega)}\big).$$
\end{enumerate}
The constants $C>0$ above depend only on $s,p,K,I\times\Omega$ and the modulus of continuity of $A(x)$.
\end{thm}

Next we state our results on global regularity. The first one, Theorem \ref{thm:1.3_f_nonzero}, deals with solutions satisfying the Dirichlet boundary condition
$u=0$ on $\R\times\partial\Omega$ when $f$ is H\"older continuous in $\overline{I\times\Omega}$ and, in addition,
is allowed to be nonzero on the boundary $I\times \partial \Omega$.  The fact that $f$ is nonzero on the boundary will affect the global regularity of the solution. Instead, when $f$ is identically zero on the boundary, we get better global regularity which is consistent with the interior estimates of Theorem \ref{thm:interiorholder},
see Theorem \ref{thm:1.3_f_zero}.
This is in high contrast with the local case of parabolic equations, namely, when $s=1$, see \cite{Lieberman}.
Such feature had already been observed in the case of fractional elliptic equations in divergence form in \cite{Stinga-Caffa}.
Our statements are also precise in terms of the sharp regularity of the coefficients and the boundary $\partial\Omega$.

\begin{thm}[Global regularity for Dirichlet and $f$ H\"older]\label{thm:1.3_f_nonzero}
Let $0<\alpha<1$ and suppose that $f\in C^{\alpha/2,\alpha}_{t,x}(\overline{I\times\Omega})$.
Let $u\in\Dom(H^s)$ be a weak solution to \eqref{eq:Hsuf} such that $u=0$ on $\R\times\partial\Omega$.
\begin{enumerate}[$(i)$]
\item Assume that $0<\alpha+2s<1$, $\partial\Omega$ is $C^{1,\alpha}$ and that $A(x) \in C^{0,\alpha}(\overline{\Omega})$. Then
$$u(t,x)\sim\dist(x,\partial\Omega)^{2s}+v(t,x)\qquad\hbox{for all}~t\in I$$
where
$$v\in C^{(\alpha+2s)/2,\alpha+2s}_{t,x}(\overline{I\times\Omega})$$
and we have the estimate
$$\|v\|_{C^{(\alpha+2s)/2,\alpha+2s}_{t,x}(\overline{I\times\Omega})}
\leq C\big(1+\|u\|_{\Dom(H^s)}+\|f\|_{C^{\alpha/2,\alpha}_{t,x}(\overline{I\times\Omega})}\big).$$
\item Assume that $s=1/2$, $\partial \Omega$ is $C^{1,\alpha+\varepsilon}$ and that $A(x)\in C^{0,\alpha+\varepsilon}(\overline{\Omega})$,
for some $\varepsilon>0$ such that $0<\alpha+\varepsilon<1$. Then
$$u(t,x)\sim\dist(x,\partial\Omega)|\log\dist(x,\partial\Omega)|+v(t,x)\qquad\hbox{for all}~t\in I$$
where
$$v\in C^{(1+\alpha)/2,1+\alpha}_{t,x}(\overline{I\times\Omega})$$
and we have the estimate
$$\|v\|_{C^{(1+\alpha)/2,1+\alpha}_{t,x}(\overline{I\times\Omega})}
\leq C\big(1+\|u\|_{\Dom(H^s)}+\|f\|_{C^{\alpha/2,\alpha}_{t,x}(\overline{I\times\Omega})}\big).$$
\item Assume that $s>1/2$, $1<\alpha+2s<2$, $\partial \Omega$ is $C^{1,\alpha+2s-1}$ and that $A(x)\in C^{0,\alpha+2s-1}(\overline{\Omega})$. Then
$$u(t,x)\sim\dist(x,\partial\Omega)+v(t,x)\qquad\hbox{for all}~t\in I$$
where
$$v\in C^{(\alpha+2s)/2,1+(\alpha+2s-1)}_{t,x}(\overline{I\times\Omega})$$
and we have the estimate
$$\|v\|_{C^{(\alpha+2s)/2,1+(\alpha+2s-1)}_{t,x}(\overline{I\times\Omega})}
\leq C\big(1+\|u\|_{\Dom(H^s)}+\|f\|_{C^{\alpha/2,\alpha}_{t,x}(\overline{I\times\Omega})}\big).$$
\end{enumerate}
The constants $C>0$ above depend only on $n,s,\alpha$ and the modulus of continuity of $\partial\Omega$ and $A(x)$.
\end{thm}

\begin{thm}[Global regularity for Dirichlet and $f$ H\"older, $f\equiv0$ on the boundary]\label{thm:1.3_f_zero}
Let $0<\alpha<1$ and suppose that $f\in C^{\alpha/2,\alpha}_{t,x}(\overline{I\times\Omega})$ is such that $f=0$ on $I\times\partial\Omega$.
Let $u\in\Dom(H^s)$ be a weak solution to \eqref{eq:Hsuf} such that $u=0$ on $\R\times\partial\Omega$.
\begin{enumerate}[$(i)$]
\item Assume that $0<\alpha+2s<1$, $\partial\Omega$ is $C^1$ and that $A(x)$ is continuous in $\overline{\Omega}$. Then
$$u\in C^{(\alpha+2s)/2,\alpha+2s}_{t,x}(\overline{I \times \Omega})$$
and we have the estimate
$$\|u\|_{C^{(\alpha+2s)/2,\alpha+2s}_{t,x}(\overline{I \times \Omega})}
\leq C\big(\|u\|_{\Dom(H^s)}+\|f\|_{C^{\alpha/2,\alpha}_{t,x}(\overline{I \times \Omega})}\big).$$
\item Assume that $1<\alpha+2s<2$, $\partial \Omega$ is $C^{1,\alpha+2s-1}$ and that $A(x)\in C^{0,\alpha+2s-1}(\overline{\Omega})$. Then
$$u\in C^{(\alpha+2s)/2,1+(\alpha+2s-1)}_{t,x}(\overline{I \times \Omega})$$
and we have the estimate
$$\|u\|_{C^{(\alpha+2s)/2,1+(\alpha+2s-1)}_{t,x}(\overline{I \times \Omega})}
\leq C\big(\|u\|_{\Dom(H^s)}+\|f\|_{C^{\alpha/2,\alpha}_{t,x}(\overline{I \times \Omega})}\big).$$
\end{enumerate}
The constants $C>0$ above depend only on $n,s,\alpha$ and the modulus of continuity of $\partial\Omega$ and $A(x)$.
\end{thm}
 
In the following we turn our attention to global regularity results for the case of the Neumann boundary condition $\partial_Au=0$ on $\R\times\partial\Omega$,
when $f$ is H\"older continuous. In contrast with the case of Dirichlet boundary condition, here the global estimates 
do not depend on the values of $f$ on the boundary and, therefore,
are consistent with the interior regularity obtained in Theorem \ref{thm:interiorholder}.

\begin{thm}[Global regularity for Neumann and $f$ H\"older]\label{thm:boundaryNeumannHolder}
Let $0<\alpha<1$ and suppose that $f\in C^{\alpha/2,\alpha}_{t,x}(\overline{I\times\Omega})$.
Let $u\in\Dom(H^s)$ be a weak solution to \eqref{eq:Hsuf} such that $\partial_Au=0$ on $\R\times\partial\Omega$.
\begin{enumerate}[$(i)$]
\item Assume that $0<\alpha+2s<1$, $\partial\Omega\in C^1$ and that $A(x)$ is continuous in $\overline{\Omega}$. Then
$$u\in C^{(\alpha+2s)/2,\alpha+2s}_{t,x}(\overline{I\times\Omega})$$
and we have the estimate
$$\|u\|_{C^{(\alpha+2s)/2,\alpha+2s}_{t,x}(\overline{I\times\Omega})}\leq
C\big(\|u\|_{\Dom(H^s)}+\|f\|_{C^{\alpha/2,\alpha}_{t,x}(\overline{I\times\Omega})}\big).$$
\item Assume that $1<\alpha+2s<2$, $\partial\Omega\in C^{1,\alpha+2s-1}$ and that $A(x)\in C^{0,\alpha+2s-1}(\overline{\Omega})$. Then
$$u\in C^{(\alpha+2s)/2,1+(\alpha+2s-1)}_{t,x}(\overline{I\times\Omega})$$
and we have the estimate
$$\|u\|_{C^{(\alpha+2s)/2,1+(\alpha+2s-1)}_{t,x}(\overline{I\times\Omega})}\leq
C\big(\|u\|_{\Dom(H^s)}+\|f\|_{C^{\alpha/2,\alpha}_{t,x}(\overline{I\times\Omega})}\big).$$
\end{enumerate}
The constants $C>0$ above depend only on $n,s,\alpha$ and the modulus of continuity of $\partial\Omega$ and $A(x)$.
\end{thm}

Finally, we state our global Schauder estimates for the case of $L^p$ right hand side,
which are in accordance with the interior estimates of Theorem \ref{thm:interiorLp}.

\begin{thm}[Global regularity for $f$ in $L^p$]\label{thm:1.5}
Suppose that $f\in L^p(I\times\Omega)$ for some $2\leq p<\infty$.
Let $u\in\Dom(H^s)$ be a weak solution to \eqref{eq:Hsuf} such that $u=0$ or $\partial_Au=0$ on $\R\times\partial\Omega$.
\begin{enumerate}[$(i)$]
\item Assume that $(n+2)/(2s)<p<(n+2)/(2s-1)_+$, $\partial\Omega\in C^1$ and that $A(x)$ is continuous in $\overline{\Omega}$. Then
$$u\in C^{\alpha/2,\alpha}_{t,x}(\overline{I\times\Omega})$$
where $\alpha=2s-(n+2)/p\in(0,1)$ and we have the estimate
$$\|u\|_{C^{\alpha/2,\alpha}_{t,x}(\overline{I\times\Omega})}\leq C\big(\|u\|_{\Dom(H^s)}+\|f\|_{L^p(I\times\Omega)}\big).$$
\item Assume that $s>1/2$, $p>(n+2)/(2s-1)$ and that $A(x)\in C^{0,\alpha}(\overline{\Omega})$ for $\alpha=2s-(n+2)/p-1\in(0,1)$. Then
$$u\in C^{(1+\alpha)/2,1+\alpha}_{t,x}(\overline{I\times\Omega})$$
and we have the estimate
$$\|u\|_{C^{(1+\alpha)/2,1+\alpha}_{t,x}(\overline{I\times\Omega})}\leq C\big(\|u\|_{\Dom(H^s)}+\|f\|_{L^p(I\times\Omega)}\big).$$
\end{enumerate}
The constants $C>0$ above depend only on $n,s,p$ and the modulus of continuity of $\partial\Omega$ and $A(x)$.
\end{thm}

The main technique to prove our Schauder estimates is to use the parabolic extension problem, which turns the nonlocal equation \eqref{eq:Hsuf}
into a local degenerate parabolic problem with Neumann boundary condition.
This result for $H^s$ was proved in \cite{B-DLC-S}. Such an extension problem is in the spirit of the famous Caffarelli--Silvestre
extension problem for the fractional Laplacian \cite{Caffa-Silv}. As the extension \eqref{th_exten} localizes the equation,
we can prove energy estimates with appropriate test functions and then apply compactness arguments in the local parabolic equation.
Indeed, we first prove a counterpart of the parabolic Caccioppoli inequality in Lemma \ref{lem:cacc_mod}. 
For this the Steklov averages are an essential tool. Second, the compactness provided by the Aubin--Lions lemma \cite{Aubin-Lions}, together with the energy estimate,
give us that there is a solution $W$ to a degenerate heat equation \eqref{eq:harmonic} that is `close' to our solution $U$ in $L^2$, see Corollary \ref{cor:approx}.
This approximation is applied at any scale to finally transfer the regularity from $W$ to $U$.

For the last step above, we need to use an appropriate characterization of parabolic H\"older spaces in terms of approximations of solutions by linear polynomials.
The definition of the space $C^{\delta/2,\delta}_{t,x}$ is clear in the case when $0<\delta<1$, namely, when there are no derivatives in time and space.
It is also clear how to define the space $C^{1+\delta/2,2+\delta}_{t,x}$, that is, when we have one derivative in time and two derivatives in space.
But it is not immediate how to define the appropriate intermediate H\"older space $C^{(1+\delta)/2,1+\delta}_{t,x}$, that is,
the one that corresponds to one derivative in space. In \cite{Krylov}, N.~V.~Krylov used interpolation results to suggest a definition.
Indeed, in Remark 8.8.7 he claims that ``\textit{with respect to the parabolic metric, one derivative in $t$ is worth two derivatives in $x$. This suggests that $C^{(1+\delta)/2, 1+\delta}(\R^{d+1})$ should be defined as the space of all functions with finite norm 
$\|u\|_0 + \|u_x\|_{\delta/2,\delta} + \sup_{s\neq t,x} \frac{|u(t,x)-u(s,x)|}{|t-s|^{(1+\delta)/2}}$.}''
Stinga and Torrea showed that this definition for the intermediate H\"older space $C^{(1+\delta)/2,1+\delta}_{t,x}(\R^{n+1})$ is correct in terms of the Poisson semigroup
generated by the heat operator, see \cite[Theorem~7.2]{Stinga-Torrea-SIAM}. They used such a semigroup characterization to prove Schauder estimates for 
solutions to $(\partial_t - \Delta)^{\pm s}u = f$. In turn, here we show in Theorem \ref{thm:Campanato}$(2)$
that Krylov's definition of intermediate parabolic H\"older space is also the correct one for bounded
domains in terms of approximations by linear polynomials that depend only on space.
This is a Campanato-type characterization that, up to the best of our knowledge, has not been proved in the literature.
Notice that in the case of no derivative in time and space, or one derivative in time and two derivatives in space,
such characterizations are very well known, see \cite{Lieberman, Schlag}.

There are some intricate issues in the proof of global regularity, in particular, in Theorem \ref{thm:1.3_f_nonzero}.
We have already pointed out that regularity is improved when $f$ is zero on the boundary.
This fact is better explained by computing particular one dimensional pointwise solutions to $(\partial_t - D^+_{xx})^su=f$ in $\R \times \R_+$, given $u(t,0) = 0$ in $\R$, when $f$ is nonzero on the boundary. Here $D_{xx}^+$ is the Dirichlet Laplacian in the positive half line $\R_+$.
On one hand, this one dimensional particular solution has the same regularity as the difference $u-v$ in Theorem \ref{thm:1.3_f_nonzero}.
On the other hand, due to this solution, for $s \leq 1/2$, we need a little bit more regularity on the boundary $\partial \Omega$ and on the coefficients $A(x)$ to get $C^{(\alpha+2s)/2,1+(\alpha+2s-1)}_{t,x}(\overline{I \times \Omega})$ regularity for $v$.
Therefore, in this paper we also need to prove sharp estimates on the behavior of half space solutions,
both for Dirichlet and Neumann boundary conditions.

This paper is organized as follows. We give the definition of intermediate parabolic H\"older spaces and state the Campanato-type characterization in Section \ref{sec:Campanato_statement}. Then, in Section \ref{sec:Hs_def},
we go through a brief review of the definition and some properties of the operator $H^s$, including fundamental solutions,
and where we also define the weak solution to the extension problem for $H^s$. In Section \ref{sec:caccio} we prove a parabolic Caccioppoli inequality using Steklov averages for the extension problem. The proofs of interior and global Schauder estimates for the solution of \eqref{eq:Hsuf} are given in Sections \ref{sec:regularity_interior} and \ref{sec:regularity_global} respectively. In between, we present the boundary regularity for the fractional heat equation, see Section \ref{sec:regularity_bdd_heat}. In 
addition, we give a detailed study of the behavior of particular one dimensional pointwise solutions to $(\partial_t - D^+_{xx})^su=f$,
in $\R \times \R_+$. Finally, in Section \ref{section:proofofCampanato} we provide the proof of Theorem \ref{thm:Campanato}.

%%%%%%%%%%%%%%%%%%%%%%%%%%%%%%%%%%%%%%%%%%%%%%%%%%%%%%
\section{Notation and parabolic H\"older spaces}\label{sec:Campanato_statement}
%%%%%%%%%%%%%%%%%%%%%%%%%%%%%%%%%%%%%%%%%%%%%%%%%%%%%%

\noindent\textbf{Notation.} Throughout this paper we will use the following notation.
For $(t,x) \in\R\times\R^n$ and  $r>0$, we define
\begin{align*}
B_r(x) &= \{z=(z_1,z_2,\ldots,z_n)\in \R^n : |x-z|<r\}\subset\R^n \\
Q_r(t,x) &=\{(\tau,z)\in \R\times\R^n: |t-\tau|<r^2,~|x-z|<r \} \\
&=(t-r^2,t+r^2)\times B_r(x)\subset\R\times\R^{n}\\
B_r(x)^* &= \{(z,y)\in\R^n\times(0,\infty):z\in B_r(x),~0<y<r\} \\
&=B_r(x)\times(0,r)\subset\R^{n+1}_+\\
Q_r(t,x)^* &= \{(\tau,z,y)\in\R\times\R^n\times(0,\infty):|t-\tau|<r^2,~z\in B_r(x),~0<y<r\} \\
&=Q_r(t,x) \times(0,r)\subset\R\times\R^{n+1}_+.
\end{align*}
We write $B_r$, $Q_r$, etc, when $(t,x)=(0,0)$. If we let
$$B_r^+=B_r\cap\{x_n>0\}\subset\R^n_+$$
then we can also define $Q^+_r$, $(B_r^+)^*$ and $(Q^+_r)^*$ analogously. The fractional power $s\in(0,1)$
and we will always denote
$$a=1-2s\in(-1,1).$$
Finally, $x\in\Omega\subset\R^n$, $y>0$, $X=(x,y)\in\Omega\times(0,\infty)$ and $\dive$ and $\nabla$ denote
the divergence and gradient with respect to the variable $X$, respectively.

\medskip

\noindent\textbf{Parabolic H\"older spaces.}
Let $\Omega\subset\R^n$ be a bounded Lipschitz domain with Lipschitz constant $M>0$, and let $I\subset\R$ be a bounded interval. Fix any $0<\beta\leq1$.

The classical parabolic H\"older space $C^{\beta/2,\beta}_{t,x}(\overline{I\times\Omega})$
is the set of continuous functions $u=u(t,x):\overline{I\times\Omega}\to\R$ such that
$$\|u\|_{C^{\beta/2,\beta}_{t,x}(\overline{I\times\Omega})}=\|u\|_{L^\infty(I\times\Omega)}+[u]_{C^{\beta/2,\beta}_{t,x}(I\times\Omega)}<\infty$$
where
$$[u]_{C^{\beta/2,\beta}_{t,x}(I\times\Omega)}=\sup_{t,\tau\in I,\,x,z\in\Omega}
\frac{|u(t,x)-u(\tau,z)|}{\max(|t-\tau|^{1/2},|x-z|)^\beta}.$$
It is also customary to define the space $C^{(2+\beta)/2,2+\beta}_{t,x}(\overline{I\times\Omega})=C^{1+\beta/2,2+\beta}_{t,x}(\overline{I\times\Omega})$
by requiring that $u_t,D^2u\in C^{\beta/2,\beta}_{t,x}(\overline{I\times\Omega})$. For these two definitions see \cite[Chapter~8]{Krylov}.

We define the space $C^{(1+\beta)/2,1+\beta}_{t,x}(\overline{I\times\Omega})$, as the set of continuous functions
$u=u(t,x):\overline{I\times\Omega}\to\R$ such that
\begin{itemize}
\item $u$ is $(1+\beta)/2$-H\"older continuous in $t$ uniformly in $x$, that is,
$$[u]_{L^\infty_x(\Omega;C^{(1+\beta)/2}_t(I))}=\sup_{x\in\Omega}[u(\cdot,x)]_{C^{(1+\beta)/2}_t(I)}=\sup_{x\in\Omega}\sup_{t,\tau\in I}
\frac{|u(t,x)-u(\tau,x)|}{|t-\tau|^{(1+\beta)/2}}<\infty.$$
\item $\nabla_x u\in C(\overline{I\times\Omega})$ and$$[\nabla_x u]_{C^{\beta/2,\beta}_{t,x}(I\times\Omega)}=
\sup_{t,\tau\in I,\,x,z\in\Omega}\frac{|\nabla_x u(t,x)-\nabla_x u(\tau,z)|}{\max(|t-\tau|^{1/2},|x-z|)^\beta}<\infty.$$
\end{itemize}
The norm in $C^{(1+\beta)/2,1+\beta}_{t,x}(\overline{I\times\Omega})$ is given by
\begin{align*}
\|u\|_{C^{(1+\beta)/2,1+\beta}_{t,x}(\overline{I\times\Omega})}
&= \|u\|_{L^\infty(I\times\Omega)}+\|\nabla_x u\|_{L^\infty(I\times\Omega)} \\
&\quad+[u]_{L^\infty_x(\Omega;C^{(1+\beta)/2}_t(I))}+[\nabla_x u]_{C^{\beta/2,\beta}_{t,x}(I\times\Omega)}.
\end{align*}

For a point $(t,x)\in\R^{n+1}$ and $r>0$ recall that $Q_r(t,x)=(t-r^2,t+r^2)\times B_r(x)$.
Notice that $|Q_r(t,x)|=C_nr^{n+2}$, for some universal constant $C_n>0$. For the rest of this section we let
$$r_0=\min\{|I|^{1/2},\diam(\Omega)\}>0.$$
Observe that there exists a constant $C>0$ depending on $n$ and $M$ such that
for any $(t,x)\in\overline{I\times\Omega}$ and $0<r\leq r_0$ we have (see, for instance, \cite[eq.~(1.1)]{Campanato})
$$|Q_r(t,x)\cap(I\times\Omega)|=|(t-r^2,t+r^2)\cap I||B_r(x)\cap\Omega|\geq C_nr^{n+2}.$$
Let $\mathcal{P}_1$ be the set of polynomials of degree $1$ in $x$, that is,
$$\mathcal{P}_1=\big\{P(z)=A_0+A_1\cdot z:A_0\in\R,~A_1\in\R^n\big\}.$$ 

\begin{thm}[Campanato-type characterizations]\label{thm:Campanato}
Let $0<\beta\leq1$. Suppose that $u=u(t,x)\in L^2(I\times\Omega)$. Then:
\begin{enumerate}[$(1)$]
\item $u\in C^{\beta/2,\beta}_{t,x}(\overline{I\times\Omega})$ if and only if there is a constant $C>0$
such that
$$\inf_{c\in\R}\frac{1}{|Q_r(t,x)\cap(I\times\Omega)|}\int_{Q_r(t,x)\cap(I\times\Omega)}|u(\tau,z)-c|^2\,d\tau\,dz\leq Cr^{2\beta}$$
for all $(t,x)\in\overline{I\times\Omega}$ and $0<r\leq r_0$ small. In this case, if we denote by $C_\ast>0$
the least constant for which the inequality above holds, then $\|u\|^2_{L^2(I\times\Omega)}+C_\ast$ is equivalent
to $\|u\|_{C^{\beta/2,\beta}_{t,x}(\overline{I\times\Omega})}^2$.
\item $u\in C^{(1+\beta)/2,1+\beta}_{t,x}(\overline{I\times\Omega})$ if and only if there is a constant $C>0$
such that
\begin{equation}\label{eq:condition2}
\inf_{P\in \mathcal{P}_1}\frac{1}{|Q_r(t,x)\cap(I\times\Omega)|}\int_{Q_r(t,x)\cap(I\times\Omega)}|u(\tau,z)-P(z)|^2\,d\tau\,dz\leq Cr^{2(1+\beta)}
\end{equation}
for all $(t,x)\in\overline{I\times\Omega}$ and $0<r\leq r_0$ small. In this case, 
if we denote by $C_{\ast\ast}>0$ the least constant for which the inequality above holds,
then $\|u\|^2_{L^2(I\times\Omega)}+C_{\ast\ast}$ is equivalent
to $\|u\|_{C^{(1+\beta)/2,1+\beta}_{t,x}(\overline{I\times\Omega})}^2$.
\end{enumerate}
\end{thm}

We postpone the proof of Theorem \ref{thm:Campanato} until Section \ref{section:proofofCampanato}.

%%%%%%%%%%%%%%%%%%%%%%%%%%%%%%%%%%%%%%%%%%%%%%%%%%%%%%
\section{Existence of weak solutions, fundamental solution and extension problem}\label{sec:Hs_def}
%%%%%%%%%%%%%%%%%%%%%%%%%%%%%%%%%%%%%%%%%%%%%%%%%%%%%%

In this section we present the precise definition of $H^su(t,x)=(\partial_t+L)^su(t,x)$.
Let $\Omega\subset\R^n$ be a bounded Lipschitz domain and
$$Lu=-\dvie(A(x) \nabla_x u )\qquad\hbox{in}~\Omega$$
where $A(x)=(A^{ij}(x))$ is as in the introduction. Let $f\in L^2(\Omega)$. For $u \in L^2(\Omega)$, $Lu = f$ in $\Omega$ in the weak sense means that
$ \nabla_x u \in L^2(\Omega)$ and 
$$\int_{\Omega} A(x) \nabla_x u \nabla_x v\,dx  = \int_{\Omega} fv\,dx,$$
for every $v\in C^\infty_c(\Omega)$.
It is well known that, under homogeneous Dirichlet boundary condition $u=0$ on $\partial\Omega$,
$L$ has a countable family of nonnegative
eigenvalues and eigenfunctions $(\lambda_k,\phi_k)_{k=0}^\infty$
such that the set $\{\phi_k\}_{k=0}^\infty$ forms an orthonormal basis for $L^2(\Omega)$.
In the case of homogeneous Neumann boundary condition $\partial_Au=0$ on $\partial\Omega$,
a similar statement is true but the first eigenvalue $\lambda_0=0$ and we will still denote
the corresponding eigenfunctions as $\phi_k$. In this situation we will assume that all the functions
involved have zero spatial mean. In particular, $L \phi_k = \lambda_k \phi_k$, for all $k\geq0$ in the weak sense.  
Therefore, if we define
$$H^1_{L}(\Omega)\equiv \Dom(L)=\Big\{u\in L^2(\Omega):\sum_{k=0}^\infty\lambda_k|u_k|^2<\infty\Big\}$$
where $\displaystyle u_k=\int_{\Omega}u\phi_k\,dx$, then, for any $u,v\in H^1_L(\Omega)$,
$$\int_{\Omega} A(x) \nabla_x u \nabla_x v\,dx=\sum_{k=0}^\infty\lambda_k u_kv_k.$$
Thus, if $L$ is endowed with homogeneous Dirichlet boundary condition, then $H^1_L(\Omega)=H^1_0(\Omega)$,
while if $L$ is endowed with homogeneous Neumann boundary condtion, then $H^1_L(\Omega)=H^1(\Omega)$.
With this, any function $u(t,x)\in L^2(\R\times\Omega)$ can be written as
$$u(t,x)=\frac{1}{(2\pi)^{1/2}}\int_{\R}\sum_{k=0}^\infty\widehat{u_k}(\rho)\phi_k(x)e^{it\rho}\,d\rho$$
where, for almost every $t\in\R$,
$$u_k(t)=\int_{\Omega}u(t,x)\phi_k(x)\,dx$$
and $\widehat{u_k}(\rho)$ is the Fourier transform of $u_k(t)$ with respect to the variable $t\in\R$:
$$\widehat{u_k}(\rho)=\frac{1}{(2\pi)^{1/2}}\int_\R u_k(t)e^{-i\rho t}\,dt.$$
The domain of the fractional operator $H^s\equiv(\partial_t+L)^s$, $0\leq s\leq1$, is defined as
$$\Dom(H^s)=\bigg\{u\in L^2(\R\times\Omega):
\|u\|_{\Dom(H^s)}^2:=\int_{\R}\sum_{k=0}^\infty|i\rho+\lambda_k|^s|\widehat{u_k}(\rho)|^2\,d\rho<\infty\bigg\}.$$
This is a complex Hilbert space with norm $\|\cdot\|_{\Dom(H^s)}$,
whose dual is denoted by $\Dom(H^s)^\ast$.
For $u\in\Dom(H^s)$ we define
$H^su\in\Dom(H^s)^\ast$ as acting on any $v\in\Dom(H^s)$ by
$$\langle H^su,v\rangle\equiv\int_{\R}\sum_{k=0}^\infty(i\rho+\lambda_k)^s\widehat{u_k}(\rho)\overline{\widehat{v_k}(\rho)}\,d\rho$$
where $\overline{\widehat{v_k}(\rho)}$ denotes the complex conjugate of $\widehat{v_k}(\rho)$. 

As the family of eigenfunctions $\{\phi_k\}_{k\geq0}$ is an orthonormal basis of $L^2(\Omega)$, we can write
the semigroup $\{e^{-\tau L}\}_{\tau\geq0}$ generated by $L$ as
$$\langle e^{-\tau L}\varphi,\psi\rangle_{L^2(\Omega)} =\sum_{k=0}^\infty e^{-\tau \lambda_k}\varphi_k\psi_k
=\int_\Omega\int_\Omega W_\tau (x,z)\varphi(z){{ \psi(x) }}\,dz\, dx$$ for any $\varphi,\psi\in L^2(\Omega)$, where $\displaystyle\varphi_k=\int_{\Omega}\varphi\phi_k\,dx$ and $\displaystyle\psi_k=\int_{\Omega}\psi\phi_k\,dx$.
The heat kernel for $L$ is symmetric and nonnegative:
 $$W_\tau(x,z)=W_\tau(z,x)\geq0\qquad x,z\in\Omega,~\tau>0$$
see \cite{Davies}. We define, for any $u\in L^2(\R\times\Omega)$,
$$e^{-\tau H}u(t,x)=e^{-\tau L}(e^{-\tau \partial_t} u)(t,x)=e^{-\tau L}(u(t-\tau,\cdot))(x)$$
in the sense that, for any $v\in L^2(\R\times\Omega)$, 
\begin{align*}
\langle e^{-\tau H}u,v\rangle_{L^2(\R\times\Omega)} &= \int_\R\sum_{k=0}^\infty e^{-\tau(i\rho+\lambda_k)}\widehat{u_k}(\rho)\overline{\widehat{v_k}(\rho)}\,d\rho \\
&=\int_\R\sum_{k=0}^\infty e^{-\tau\lambda_k}u_k(t-\tau)v_k(t)\,dt\\
&=\int_\R\iint_\Omega W_\tau (x,z) u(t-\tau,z){{v(t,x)}} \,dz\, dx\,dt.
\end{align*}

\begin{lem}[{See \cite{B-DLC-S}}]
		Let $0<s<1$. If $u\in\Dom(H^s)$ then 
	$$H^su =  \frac{1}{\Gamma(-s)} \int^{\infty}_0 \left( {e^{-\tau H} u }- {u} \right) \frac{d \tau} {\tau^{1+s}}$$
		in the sense that, for any $v\in\Dom(H^s)$,
$$ \langle H^s u , v \rangle =  \frac{1}{\Gamma(-s)} \int^{\infty}_0 \left(\langle {e^{-\tau H}u} , v \rangle_{L^2(\R\times\Omega)} - \langle {u} , {v} \rangle_{L^2{(\R\times\Omega)}} \right) \frac{d \tau}{\tau^{1+s}}.$$
	\end{lem}

\begin{thm}[{See \cite{B-DLC-S}}]
If $u,v\in \Dom(H^s)\cap C^\infty_c(\R\times\Omega)$ then
\begin{equation}\label{eq:pointwise_formula}
\begin{aligned}
&\langle H^su,v\rangle = \langle(\partial_t+L)^su,v\rangle \\ 
&=\int_0^\infty\int_{\R}\int_\Omega\int_{\Omega} K_s(\tau,x,z)(u(t-\tau,x)-u(t-\tau,z))({{v(t,x)-v(t,z)}})\,dz\,dx\,dt\,d\tau \\ 
&\quad +\int_0^\infty\bigg[\int_\R\int_\Omega\frac{\big(1-e^{-\tau L}1(x)\big)}{|\Gamma(-s)|\tau^{1+s}}u(t,x){{v(t,x)}}\,dx\,dt\\ 
&\quad\qquad\qquad-\int_\R\int_\Omega e^{-\tau L}1(x)\frac{(u(t-\tau,x)-u(t,x))}{|\Gamma(-s)|\tau^{1+s}}{{v(t,x)}}\,dx\,dt\bigg]\,d\tau
\end{aligned}
\end{equation}
where
$$K_s(\tau,x,z)=\frac{W_\tau(x,z)}{2|\Gamma(-s)|\tau^{1+s}}$$
and
$$e^{-\tau L}1(x)=\int_\Omega W_\tau(x,z)\,dz.$$
\end{thm}

\begin{rem}[Fundamental solution]\label{rem:fund1}
Given $f\in L^2(\R\times\Omega)$, the solution $u\in\Dom(H^s)$ to $H^su=f$ is given by
$$u(t,x)=H^{-s}f(t,x)=\frac{1}{(2\pi)^{1/2}}\int_{\R}\sum_{k=1}^\infty(i\rho+\lambda_k)^{-s}\widehat{f}_k(\rho)\varphi_k(x)e^{i\rho t}\,d\rho.$$
Using the Gamma function identity
$$(i\rho+\lambda_k)^{-s}=\frac{1}{\Gamma(s)}\int_0^\infty e^{-\tau(i\rho+\lambda_k)}\,\frac{d\tau}{\tau^{1-s}}$$
and the heat kernel $W_\tau(x,z)$ for $L$, we readily find that
\begin{align*}
u(t,x) &= H^{-s}f(t,x)=\frac{1}{\Gamma(s)}\int_0^\infty e^{-\tau L}f(t-\tau,x)\,\frac{d\tau}{\tau^{1-s}} \\
&= \int_{-\infty}^\infty\int_\Omega K_{-s}(\tau,x,z)f(t-\tau,z)\,dz\,d\tau
\end{align*}
where
$$K_{-s}(\tau,x,z)=\chi_{\tau>0}\frac{W_\tau(x,z)}{\Gamma(s)\tau^{1-s}}=\frac{\chi_{\tau>0}}{\Gamma(s)\tau^{1-s}}\sum_{k=0}^\infty e^{-\tau \lambda_k} \phi_k(x)\phi_k(z)$$
is the fundamental solution for $H^s$. We can estimate this kernel by applying known estimates for the heat kernel for $L$.
\begin{enumerate}[(a)]
\item If the coefficients $A(x)$ are bounded and measurable then, by \cite{Davies}, we find that
$$K_{-s}(\tau,x,z)\leq\frac{C}{\tau^{n/2+1-s}}e^{-|x-z|^2/(c\tau)}\qquad x,z\in\Omega,~\tau>0$$
for some constants $C,c>0$.
\item If the coefficients $A(x)$ are bounded and measurable in $\Omega=\R^n$ then, by Aronson's estimates \cite{Aronson},
$$\frac{C_1}{\tau^{n/2+1-s}}e^{-|x-z|^2/(c_1\tau)} \leq K_{-s}(\tau,x,z)\leq\frac{C_2}{\tau^{n/2+1-s}}e^{-|x-z|^2/(c_2\tau)}\qquad x,z\in\R^n,~\tau>0$$
for some constants $C_1,c_1,C_2,c_2>0$.
\item If the coefficients $A(x)$ are H\"older continuous with exponent $\alpha \in (0,1)$ and $L$ is endowed with homogeneous Dirichlet
boundary conditions then, from \cite[Theorem 2.2]{Riahi},
there exist positive constants $c, c_1, c_2$ and $\eta \leq 1 \leq \nu$ depending only on $n, \alpha, \Omega$ and ellipticity, with $c$ depending also on $s$, such that
\begin{multline*}
     c^{-1}\tau^{s-1} \min \bigg( 1, \frac{\phi_0(x) \phi_0(z)}{\max(1,\tau^\eta)}\bigg)e^{-\lambda_0 \tau} \frac{e^{-c_1 |x-z|^2/(\tau)}}{\max(1, \tau^{n/2})} 
    \leq K_{-s}(\tau,x,z) \\
    \leq  c\tau^{s-1} \min \bigg( 1, \frac{\phi_0(x) \phi_0(z)}{\max(1, \tau^\nu)}\bigg)e^{-\lambda_0 \tau} \frac{e^{-c_2 |x-z|^2/(\tau)}}{\max(1, \tau^{n/2})}
\end{multline*}
for all $x, z \in \Omega$, $t>0$.
\item Under the hypotheses of (c), if in addition we assume that $\Omega$ is a $C^{1, \gamma}$ domain for some $0 < \gamma <1$, then the estimate above is true for $\eta = \nu = 1$ and the constant $c$ depending also on $\gamma$. In particular, the estimate holds when $(\partial_t + L)^s = (\partial_t - \Delta_D)^s$, the fractional power of
the heat operator with Dirichlet Laplacian in a $C^{1, \gamma}$ domain. 
\item For the case of Neumann boundary conditions, if $\Omega$ is an inner uniform domain
then two-sided Gaussian estimates for the Neumann heat kernel hold and we obtain
$$\frac{C_1}{\tau^{n/2+1-s}}e^{-d(x,z)^2/(c_1\tau)} \leq K_{-s}(\tau,x,z)\leq\frac{C_2}{\tau^{n/2+1-s}}e^{-d(x,z)^2/(c_2\tau)}\qquad x,z\in\Omega,~\tau>0$$
where $d(x,z)$ denotes the geodesic distance between $x$ and $z$ in $\Omega$.
In particular, if $\Omega$ is bounded and convex, or if it is the region above the graph of a globally
Lipschitz function, then the geodesic distance $d(x,z)$ can be replaced
by the Euclidean distance $|x-z|$. For details about inner uniform domains, see \cite{Coste}.
\end{enumerate}
\end{rem}

In \cite{B-DLC-S} it was also proved that if $U$ solves
$$\begin{cases}
\partial_t U=y^{-a}\dvie(y^{a}B(x)\nabla U)&\hbox{in}~\R\times\Omega\times(0,\infty)\\
U(t,x,0)=u(t,x)&\hbox{on}~\R\times\Omega \\
U(t,x,y)=0\quad\hbox{or}\quad\partial_AU(t,x,y)=0&\hbox{on}~\R\times\partial\Omega\times(0,\infty)
\end{cases}$$
where $a=1-2s\in(-1,1)$, then, for some explicit constant $c_s>0$,
$$-\lim_{y\to0^+}y^{a}U_y(t,x,y)=c_sH^su$$
where
$$B(x)=
\begin{bmatrix}
A(x) & 0 \\
0 & 1
\end{bmatrix}$$
is also uniformly elliptic. To state this claim precisely, we need some notation.
Let us denote $D=\{(x,y):x\in\Omega,~y>0\}\subset\R^{n+1}$.
The $A_2(\R^N)$-class of Muckenhoupt weights is the set of all a.e.~positive functions $\omega\in L^1_{\mathrm{loc}}(\R^N)$,
$N\geq1$, for which there exists a constant $C_\omega>0$ such that
$$\bigg(\frac{1}{|B|}\int_B\omega\bigg)\bigg(\frac{1}{|B|}\int_B\omega^{-1}\bigg)\leq C_\omega$$
for every ball $B\subset\R^N$, see \cite{Duo}. It is straightforward to check that the weight $\omega(x,y)=|y|^{a}$ belongs
to the class $A_2(\R^{n+1})$.
Define $H^1_{L,a}(D)$ as the set of functions $w=w(x,y)\in L^2(D,y^{a}dxdy)$ such that 
\begin{align*}
[w]_{H^1_{L,y}(D)}^2 &:= \int^{\infty}_0\int_\Omega y^{a}A(x)\nabla_x w\nabla_x w
\,dx\,dy + \int^{\infty}_0 \int_{\Omega} y^{a}
|\partial_yw|^2\,dx\,dy\\
&=\int^{\infty}_0y^{a}\sum_{k=0}^\infty\lambda_k|w_k(y)|^2\,dy + \int^{\infty}_0 \int_{\Omega} y^{a}
|\partial_yw|^2\,dx\,dy< \infty,
\end{align*}
where $\displaystyle w_k(y)=\int_{\Omega}w(x,y)\phi_k(x)\,dx$, under the norm
$$\|w\|_{H^1_{L,a}(D)}^2=\|w\|_{L^2(D,y^{a}dxdy)}^2+[w]_{H^1_{L,a}(D)}^2.$$
Recall that $\{e^{-\tau H}\}_{\tau\geq0}$ denotes the semigroup generated by $H=\partial_t-\dive(A(x)\nabla_x)$.

\begin{thm}[Extension problem]\label{th_exten}
Let $u\in\Dom(H^s)$. For $(t,x)\in\R\times\Omega$ and $y>0$ we define
\begin{equation}\label{solution}
\begin{aligned}
U(t,x,y) &= \frac{y^{2s}}{4^s \Gamma(s)} \int^{\infty}_0 e^{-y^2/(4\tau)} e^{-\tau H} u(t,x)\,\frac{d \tau}{\tau^{1+s}} \\
&= \frac{1}{\Gamma(s)}\int_0^\infty e^{-r}e^{-\frac{y^2}{4r}H}u(t,x)\,\frac{dr}{r^{1-s}}\\
&= \frac{1}{\Gamma(s)}\int_0^\infty e^{-y^2/(4r)}e^{-rH}(H^su)(t,x)\,\frac{dr}{r^{1-s}}.
\end{aligned}
\end{equation}
Then $U$ belongs to $L^2(\R;H^1_{L,a}(D))\cap C^\infty((0,\infty);L^2(\R\times\Omega))\cap C([0,\infty);L^2(\R \times \Omega))$ and
is a weak solution to the parabolic extension problem
$$\begin{cases}
\partial_t U=y^{-a}\dvie(y^{a}B(x)\nabla U)&\hbox{for}~(t,x,y)\in\R\times\Omega\times(0,\infty)\\
-y^{a}\partial_yU\Big|_{y=0^+}=\frac{\Gamma(1-s)}{4^{s-1/2}\Gamma(s)}H^su&\hbox{for}~(t,x)\in\R\times\Omega\\
U(t,x,0)=u(t,x)&\hbox{for}~(t,x)\in\R\times\Omega
\end{cases}$$
with the boundary condition $U=0$ or $\partial_AU=0$ on $\R\times\partial\Omega\times(0,\infty)$,
depending whether $L$ is endowed with homogeneous Dirichlet or Neumann boundary conditions, respectively.
Namely, for any $V(t,x,y) \in C^{\infty}_c (\R\times\Omega\times[0, \infty) )$, in case of Dirichlet; or for any
$V(t,x,y) \in C^{\infty} (\R\times\Omega\times[0, \infty) )$ with compact support in $t$ and $y$, in case of Neumann,
\begin{align*}
\int_\R\int_\Omega U\partial_t V\,dx\,dt &= \int_\R\int_\Omega A(x)\nabla_xU\nabla_xV\,dx \,dt \\
&\quad-\int_\R\int_\Omega\Big(\tfrac{a}{y}\partial_y+\partial_{yy}\Big)UV \,dx\,dt
\end{align*}
$\lim_{y\to0^+}U(t,x,y)=u(t,x)$ in $L^2(\R\times\Omega)$ and
\begin{align*}
\int_0^\infty\int_\R\int_\Omega y^{a}U\partial_tV\,dx \,dt \,dy 
&=\int_0^\infty\int_\R\int_\Omega y^{a}B(x)\nabla U\nabla V\,dx\,dt \,dy \\
&\quad-\frac{\Gamma(1-s)}{4^{s-1/2}\Gamma(s)}\langle H^su, V(t,x,0) \rangle.
\end{align*}
By density, these identities hold for test functions $V$ in $L^2(\R;H^1_{L,a}(D))$.
In addition, we have the estimate
\begin{equation}\label{eq:L2vsHs}
\|U\|_{L^2(\R;H^1_{L,a}(D))}\leq C\|u\|_{\Dom(H^s)}
\end{equation}
where $C>0$ depends only on $s$.
\end{thm}

\begin{rem}[Fundamental solution using extension problem]
We can also get an estimate for the fundamental solution of $H^s$ by the extension method.
Let $K_{-s}(\tau,x,z)$ be the fundamental solution of $H^s$ with pole at $\tau=0$ and $z=x$.
For fixed $x$, let $U^x=U^x(\tau,z,y)$ be the solution to the following extension problem
$$\begin{cases}
y^a \partial_\tau U^x - \dvie(y^a B(z) \nabla U^x)= 0&\hbox{in}~\R \times \Omega \times (0,\infty) \\ 
-\lim_{y \rightarrow 0} y^a(U^x)_y(\tau,z,y)=c_s \delta_{(0,x,0)}&\hbox{on}~\R\times\Omega
\end{cases}$$
with the appropriate boundary condition on $\R\times\partial\Omega\times[0,\infty)$.
Here $\delta_{(0,x,0)}$ denotes the Dirac delta at $\tau=0$, $x\in\Omega$ and $y=0$. 
Then
$$K_{-s}(\tau,x,z)=U^x(\tau,z,0).$$
Let $\tilde{U}^x$ to be the even reflection of $U^x$ with respect to the variable $y$,
that is, $\tilde{U}^x(\tau,z,y) = U^x(\tau,z,|y|)$. Then, exactly as in \cite[Lemma 4.2]{B-DLC-S}, we find that $\tilde{U}^x$ solves
$$|y|^a\partial_\tau \tilde{U}^x - \dvie(|y|^a B(z) \nabla \tilde{U}^x) = c_s \delta_{(0,x,0)}\qquad\hbox{in}~\R \times \Omega \times(-\infty,\infty)$$
with the corresponding boundary conditions.
Clearly, for $\mathcal{U}^x(\tau,z,y)  = \chi_{\tau\geq0}\tilde{U}^x(\tau,z,y)$, we have
$$\begin{cases}
|y|^a\partial_\tau \mathcal{U}^x - \dvie(|y|^a B(z) \nabla \mathcal{U}^x)= 0&\hbox{in}~(0,\infty)\times\Omega\times(-\infty,\infty)\\ 
\lim_{\tau \to 0} \mathcal{U}^x(\tau,z,y) =c_s \delta_{(0,x,0)} 
\end{cases}$$
Then $\mathcal{U}^x$ is the heat kernel associated with the elliptic operator $\dvie(|y|^a B(x) \nabla)$ with pole at $(\tau,z,y)=(0,x,0)$.
Thus, from known heat kernel estimates for degenerate parabolic operators, we can derive bounds for the fundamental solution $K_{-s}(\tau,x,z)$.

Suppose that $\Omega=\R^n$, denote $X=(x,x_{n+1}),Z=(z,y)\in\R^{n+1}$ and let $W_{\tau}(X,Z)$ be the heat kernel for
$\dvie(|y|^a B(x) \nabla)$ with pole at $\tau=0$ and $Z=X$. From \cite{Uribe-Rios}, we have the Gaussian estimate
$$|W_\tau(X,Z)| \leq  \frac{C}{\sqrt{w_{\tau}(X)}\sqrt{w_{\tau}(Z)}} e^{-c|X-Z|^2/\tau}$$ 
where $w(Z)=|y|^a$ is an $A_2$ Muckenhoupt weight, $w_\tau(Z)$ is the $w$-volume of the 
ball centered at $Z$ with radius $\sqrt{\tau}$ in the usual metric in $\R^{n+1}$ and $C,c>0$ depend on $s$, $n$ and ellipticity.
It is easy to check that $w_\tau((z,0))\sim\tau^{n/2+1+s}$. Therefore, the fundamental solution for $H^s$ in $\Omega=\R^n$ verifies
$$K_{-s}(\tau,x,z)=W_\tau((x,0),(z,0))\leq\frac{C}{\tau^{n/2+1+s}}e^{-c|x-z|^2/\tau}\qquad\tau>0$$
for $C,c>0$ depending only on $s$, $n$ and ellipticity. Compare this estimate with those in Remark \ref{rem:fund1}.
\end{rem}

\begin{prop}\label{prop:t_1t_2}
Let $U$ be as in \eqref{solution} and assume that $f=H^su\in L^2(\R\times\Omega)$. Then $U_t\in L^2(\R;(H^1_{L,a}(D))^\ast)$ and,
in particular, $U\in C(\R;L^2(D,y^adX))$.
Furthermore, for every $\phi \in H^1([-1,1];L^2(B^*_1,y^adX)) \cap L^2([-1,1];H^1_{L,a}(B^*_1))$
such that $\phi = 0$ on $\partial Q^*_1 \backslash (Q_1\times \{ 0 \})$ and a.e. $t_1,t_2\in[-1,1]$, we have
\begin{multline*}
\int_{B_1^*}y^a\big[U\phi\big]^{t=t_2}_{t=t_1}\,dX-\int_{t_1}^{t_2}\int_{B_1^*}y^aU \partial_t\phi\,dXdt + \int_{t_1}^{t_2}\int_{B^*_1} y^a B(x) \nabla U \nabla \phi \, dXdt \\
=\frac{\Gamma(1-s)}{4^{s-1/2}\Gamma(s)}\int_{t_1}^{t_2}\int_{B_1}f(t,x) \phi(t,x,0) \, dx \, dt.
\end{multline*}
\end{prop}

\begin{proof}
We claim that
\begin{equation}\label{UtinH1ast}
U_t=y^{-a}\dvie(y^{a}B(x)\nabla U)\in(H^1_{L,a}(D))^\ast
\end{equation}
in the weak sense, namely, that for any $\psi(t)\in C^\infty_c(\R)$ and any $\phi(x,y)\in H^1_{L,a}(D)$,
\begin{multline*}
\int_0^\infty y^a\int_\Omega\bigg(\int_\R U\psi_t\,dt\bigg)\,\phi\,dx\,dy \\
=\int_\R\int_0^\infty\int_\Omega y^aB(x)\nabla U\nabla\phi\,dX\psi\,dt+c_s\int_\R\int_{\Omega}f(t,x) \phi(x,0) \, dx\,\psi \, dt.
\end{multline*}
Indeed, notice that, by Theorem \ref{th_exten}, $U\in L^2(\R;H^1_{L,a}(D))$, so that
$$U_t(\psi)=-\int_\R U\psi_t\,dt\in H^1_{L,a}(D).$$
Therefore,
$$[U_t(\psi)](\phi)=-\int_0^\infty y^a\int_\Omega\bigg(\int_\R U\psi_t\,dt\bigg)\,\phi\,dx\,dy$$
is well defined. On the other hand, for a.e. $t\in\R$,
$$[y^{-a}\dvie(y^{a}B(x)\nabla U)](\phi)=-\int_0^\infty\int_\Omega y^{a}B(x)\nabla U\nabla \phi\,dx\,dy
+c_s\int_\Omega f(t,x)\phi(x,0)\,dx$$
is a well defined bounded linear functional on $H^1_{L,a}(D)$, because $U\in L^2(\R;H^1_{L,a}(D))$, $f\in L^2(\R\times\Omega)$
and the trace inequality $\|\phi\|_{L^2(\Omega)}\leq C_s\|\phi\|_{H^1_{L,a}(D)}$ holds true.
On the other hand, from Theorem \ref{th_exten}, we see that
$$-[U_t(\psi)](\phi)=\int_\R\bigg[\int_0^\infty\int_\Omega y^{a}B(x)\nabla U\nabla \phi\,dx\,dy
-c_s\int_\Omega f(t,x)\phi(x,0)\,dx\bigg]\,\psi\,dt.$$
Thus, \eqref{UtinH1ast} holds.

Moreover, it is clear that
\begin{multline*}
\bigg\|\int_0^\infty\int_\Omega y^{a}B(x)\nabla U\nabla \phi\,dx\,dy
-c_s\int_\Omega f(t,x)\phi(x,0)\,dx\bigg\|_{L^2(\R)} \\
\leq C\big(\|U\|_{L^2(\R,H^1_{L,a}(D))}+\|f\|_{L^2(\R\times\Omega)}\big)\|\phi\|_{H^1_{L,a}(D)}.
\end{multline*}
This gives that $U_t\in L^2(\R;(H^1_{L,a}(D))^\ast)$ and \eqref{UtinH1ast} holds a.e., namely,
\begin{equation}\label{ae}
\int_0^\infty y^a\int_\Omega U_t\phi\,dx\,dy=\int_0^\infty\int_\Omega y^{a}B(x)\nabla U\nabla \phi\,dx\,dy
-c_s\int_\Omega f(t,x)\phi(x,0)\,dx
\end{equation}
for a.e. $t\in\R$.

For the second claim, notice that $U\in H^1([-1,1];(H^1_{L,a}(D))^\ast)$. Then, for any $\psi\in C^\infty_c(-1,1)$
and a.e. $t_1,t_2\in(-1,1)$, by using a standard mollifier argument, we have
$$\big[U\psi\big]^{t=t_2}_{t=t_1}=(U\psi)(t_2)-(U\psi)(t_1)=\int_{t_1}^{t_2}U_t\psi\,dt+\int_{t_1}^{t_2}U\psi_t\,dt.$$
Whence, multiplying by $\psi$ and integrating from $t_1$ to $t_2$ in \eqref{ae}, we find that
\begin{multline*}
\int_0^\infty y^a\int_\Omega\big[U\psi\big]^{t=t_2}_{t=t_1} \phi\,dx\,dy-\int_0^\infty y^a\int_\Omega\int_{t_1}^{t_2}U\psi_t\phi\,dt\,dx\,dy \\
=\int_{t_1}^{t_2}\int_0^\infty\int_\Omega y^{a}B(x)\nabla U\nabla \phi\,dx\,dy\,\psi\,dt
-c_s\int_{t_1}^{t_2}\int_\Omega f(t,x)\phi(x,0)\,dx\,\psi\,dt.
\end{multline*}
The conclusion is true by approximation.
\end{proof}

%%%%%%%%%%%%%%%%%%%%%%%%%%%%%%%%%%%%%%%%%%%%%%%%%%%%%%
\section{Caccioppoli estimate and approximation}\label{sec:caccio}
%%%%%%%%%%%%%%%%%%%%%%%%%%%%%%%%%%%%%%%%%%%%%%%%%%%%%%

In view of Proposition \ref{prop:t_1t_2}, we define weak solutions to the extension problem in $Q^*_1$
in the following way. Consider the problem
\begin{equation}\label{eq:exten_1_mod}
\begin{cases}
y^a \partial_t U - \dvie(y^a B(x) \nabla U)=- \dvie(y^a F)&\hbox{in}~Q^*_1 \\
-y^a U_y\big|_{y=0}=f&\hbox{on}~Q_1.
\end{cases}
\end{equation}
Here $F=F(t,x)=(F_1,\ldots,F_n,F_{n+1})$ is an $\R^{n+1}$-valued vector field on $Q^*_1$
such that $F_{n+1} = 0$ and $|F|\in L^2(Q_1^*)$, and $f=f(t,x)\in L^2(Q_1)$.
We say that $U \in C([-1,1];L^2(B^*_1, y^a dX)) \cap L^2([-1,1];H^1_{L,a}(B^*_1))$ is a weak solution to \eqref{eq:exten_1_mod} if for every $-1<t_1<t_2<1$
\begin{equation}\label{eq:exten_2_mod}
\begin{aligned}
\int_{B^*_1} y^aU\phi\big|^{t=t_2}_{t=t_1} dX &-\int^{t_2}_{t_1} \int_{B^*_1}y^aU \partial_t\phi\,dXdt + \int^{t_2}_{t_1}\int_{B^*_1} y^a B(x) \nabla U \nabla \phi \, dXdt \\
&\qquad=\int^{t_2}_{t_1}\int_{B_1}f(t,x) \phi(t,x,0) \, dx \, dt +\int^{t_2}_{t_1}\int_{B^*_1} y^a F \nabla \phi \, dXdt
\end{aligned}
\end{equation}
holds for every $\phi \in H^1([-1,1];L^2(B^*_1,y^adX)) \cap L^2([-1,1];H^1_{L,a}(B^*_1))$
such that $\phi = 0$ on $\partial Q^*_1 \backslash (Q_1 \times \{ 0 \})$. Any such $\phi$ will be called a test function. 

\begin{lem}\label{lem:cacc_mod}
Suppose that $U$ is a weak solution to \eqref{eq:exten_1_mod} in the sense of \eqref{eq:exten_2_mod} with $F$ as described above.
Then, for any $\eta \in C^{\infty}_c(Q_1\times[0,1))$ and for any $-1<t_1<t_2<1$, 
\begin{align*}
&\sup_{t_1<t<t_2}\int_{B^*_1}y^aU^2 \eta^2 \,dX + \int^{t_2}_{t_1}\int_{B^*_1} y^a \eta^2 |\nabla U|^2\,dXdt \\
& \leq C\bigg[\int^{t_2}_{t_1}\int_{B^*_1} y^a \left(( |\partial_t (\eta^2)|+|\nabla \eta|^2)U^2 + |F|^2 \eta^2 \right)\,dXdt  \\
&\qquad+ \int^{t_2}_{t_1}\int_{B_1} (\eta(t,x,0))^2 |U(t,x,0)||f(t,x)|\,dxdt  \bigg] + \int_{B^*_1}y^aU^2(t_1,X) \eta^2(t_1,X) \,dX
\end{align*}
where $C>0$ depends only on ellipticity, $n$ and $s$.
\end{lem}

\begin{proof}
First we will define the Steklov averages of $U$ and state some of their properties (see, for example, \cite{lady}).
Let $-1<t<1$ and $h>0$ such that $t+h < 1$. We define
$$ U_h(t,x,y) = \frac{1}{h} \int^{t+h}_t U(\tau,x,y)\,d\tau\qquad\hbox{for}~t \in (-1, 1-h] $$
and $U_h(t,x,y)=0$ for $t > 1-h$, for all $(x,y)\in B_1^*$.
Since $U(\cdot,x,y)\in L^2([-1,1])$ for almost every $(x,y)\in B^*_1$, it follows that
$U_h$ is differentiable almost everywhere in $(-1,1)$, for almost every $(x,y)\in B_1^*$, and
$$\partial_t U_h(t,x,y) = \frac{U(t+h,x,y)-U(t,x,y)}{h}\in L^2([-1,1]).$$
Moreover, since $U \in C([-1,1];L^2(B^*_1;y^adX))$, we have that
$$\lim_{h\to0}U_h=U\qquad\hbox{in}~L^2(B^*_1;y^adX),~\hbox{for every}~t \in (-1,1-\delta)$$
for any $\delta \in (0,2)$.
Additionally, for any $\delta \in (0,2)$,
$$\lim_{h\to0}U_h=U\qquad\hbox{in}~L^2([-1,1-\delta];L^2(B^*_1;y^adX)).$$

Now we see that $U_h$ satisfies
\begin{equation}\label{eq:steklov_5}
\begin{aligned}
\int_{ B^*_1} \big[y^a (U_h)_t \varphi &+ y^a B(x) \nabla U_h\nabla\varphi \big]\,dX \\
 &= \int_{B_1} f_h(t,x) \varphi(x,0) \, dx+\int_{B^*_1} y^a F_h\nabla \varphi \, dX 
 \end{aligned}
 \end{equation}
 for almost every $-1<t < 1-h$ and for every $\varphi=\varphi(x,y)\in H^1(B^*_1)$ such that $\varphi = 0$
 on $\partial B^*_1 \setminus (B_1\times \{ 0 \})$, where $F_h, f_h$ are defined in the similar fashion. 
This follows by choosing $t_1 = t$ and $t_2 = t+h$ such that $[t_1, t_2] \subset [-1, 1]$ and 
$\phi=\varphi$ (which is independent of the time variable) in the weak formulation \eqref{eq:exten_2_mod}.

Next, fix a subinterval $[t_1,t_2] \subseteq [-1,1]$ such that $t_2 + h < 1$. In \eqref{eq:steklov_5} we take
$\varphi=\phi$, where $\phi=\phi(t,x,y)$ is a test function as in the definition of the weak formulation \eqref{eq:exten_2_mod}.
Then \eqref{eq:steklov_5} holds for almost every $t\in(-1,1-h)$ and, if we integrate in the $t$-variable over $[t_1, t_2]$
and use integration by parts in $t$, we finally get
\begin{equation}\label{eq:steklov_6}
\begin{aligned}
\int_{ B^*_1}[y^a U_h\phi]^{t_2}_{t_1} \,dX + &\int^{t_2}_{t_1}  \int_{ B^*_1} \left[-y^aU_h \phi_{t} + y^a B(x) \nabla U_h \nabla \phi \right] dXdt \\
&=  \int^{t_2}_{t_1} \int_{B_1} f_h(t,x) \phi(t,x,0) \, dx  dt  + \int^{t_2}_{t_1} \int_{B^*_1} y^a F_h\nabla \phi \, dX  dt.
\end{aligned}
\end{equation}
Observe that, from the earlier properties of Steklov average,
by taking $h \to 0$ in \eqref{eq:steklov_6} one arrives to \eqref{eq:exten_2_mod}.

For the proof of the Caccioppoli inequality, let $\phi = \eta^2 U_h$ in \eqref{eq:steklov_6}.
Since
\begin{align*}
\int^{t_2}_{t_1}\int_{B^*_1} y^a U_h \partial_t( \eta^2 U_h)\,dXdt &= \int^{t_2}_{t_1}\int_{B^*_1}  y^a U_h^2 \partial_t(\eta^2)\,dXdt + \frac{1}{2} \int^{t_2}_{t_1}\int_{B^*_1}  y^a \eta^2  \partial_t( U_h^2)\,dXdt \\
&=\frac{1}{2}\int^{t_2}_{t_1}\int_{B^*_1}  y^a U_h^2 \partial_t(\eta^2)\,dXdt + \frac{1}{2}\int_{ B^*_1}[y^a \eta^2 U^2_h]^{t_2}_{t_1} dX
\end{align*}
it follows that 
\begin{align*}
\frac{1}{2}\int_{ B^*_1}[y^a \eta^2 U^2_h]^{t_2}_{t_1}&\,dX+\int^{t_2}_{t_1}\int_{B^*_1} y^a B(x) \eta^2 \nabla U_h \nabla U_h\,dXdt \\ 
&= \frac{1}{2} \int^{t_2}_{t_1}\int_{B^*_1} y^a U_h^2 \partial_t(\eta^2)\,dXdt - 2\int^{t_2}_{t_1}\int_{B^*_1} y^a B(x) \eta U_h \nabla U_h \nabla \eta\,dXdt \\
&\quad+   \int^{t_2}_{t_1}\int_{B^*_1} y^a \eta^2 F_h \nabla U_h\,dXdt+ 2 \int^{t_2}_{t_1}\int_{B^*_1} y^a F_hU_h \eta \nabla \eta\,dXdt \\
&\quad + \int^{t_2}_{t_1}\int_{B_1} (\eta(t,x,0))^2U_h(t,x,0) f_h(t,x) \,dxdt.
\end{align*}
By the properties of Steklov averages, we can take the limit as $h\to0$ above to deduce that the same
identity holds for $U$, $F$ and $f$ in place of $U_h$, $F_h$ and $f_h$, respectively.
Then, by ellipticity and the Cauchy inequality with $\varepsilon >0$,
\begin{align*}\nonumber
\frac{1}{2}&\int_{ B^*_1}[y^a \eta^2 U^2]^{t_2}_{t_1}\,dX + \lambda \int^{t_2}_{t_1}\int_{B^*_1} y^a \eta^2 |\nabla U|^2\,dXdt \\ 
&\leq \frac{1}{2} \int^{t_2}_{t_1}\int_{B^*_1} y^a U^2 |\partial_t(\eta^2)|\,dXdt \\
&\quad+\frac{C}{\varepsilon}\int^{t_2}_{t_1}\int_{B^*_1} y^a U^2 |\nabla \eta|^2\,dXdt + \varepsilon \int^{t_2}_{t_1}\int_{B^*_1} y^a \eta^2 |\nabla U|^2\,dXdt \\ 
&\quad + \frac{C}{\varepsilon} \int^{t_2}_{t_1}\int_{B^*_1} y^a \eta^2 |F|^2\,dXdt + \varepsilon \int^{t_2}_{t_1}\int_{B^*_1} y^a \eta^2 |\nabla U|^2\,dXdt \\ 
&\quad+C\int^{t_2}_{t_1}\int_{B^*_1} y^a \eta^2 |F|^2\,dXdt +C\int^{t_2}_{t_1}\int_{B^*_1} y^a U^2 |\nabla\eta|^2\,dXdt \\
&\quad+ \int^{t_2}_{t_1}\int_{B^*_1} \eta^2 |U| |f|\,dxdt.
\end{align*}
The conclusion follows in a standard way by choosing $\varepsilon>0$ sufficiently small.
\end{proof}

Let us consider a test function $\phi \in H^1([-1,1];L^2(B^*_1,y^adX)) \cap L^2([-1,1];H^1_{L,a}(B^*_1))$
with $\phi = 0$ on $\partial Q^*_1 \backslash (Q_1 \times \{ 0 \})$.
Let $U \in C([-1,1];L^2(B^*_1, y^a dX)) \cap L^2([-1,1];H^1_{L,a}(B^*_1))$.
If $U$ is a weak solution to \eqref{eq:exten_1_mod} in the sense of \eqref{eq:exten_2_mod} then,
by letting $t_2\to1$ and $t_1\to-1$, we find that
\begin{equation}\label{eq:exten_2}
\begin{aligned}
-\int_{Q^*_1}y^aU \partial_t\phi\,dXdt &+ \int_{Q^*_1} y^a B(x) \nabla U \nabla \phi \, dXdt \\
&\qquad=\int_{Q_1}f(t,x) \phi(t,x,0) \, dx \, dt +\int_{Q^*_1} y^a F \nabla \phi \, dXdt.
\end{aligned}
\end{equation}
Conversely, if $U$ satisfies \eqref{eq:exten_2} for all such $\phi$ then, by
using arguments similar to Proposition \ref{prop:t_1t_2} we get that \eqref{eq:exten_2_mod} holds.
Therefore, when referring to weak solutions to \eqref{eq:exten_1_mod},
we will mean that \eqref{eq:exten_2_mod} or, equivalently, \eqref{eq:exten_2}, hold for the corresponding test functions.

\begin{cor}\label{cor:approx}
Let $U$ be a weak solution to \eqref{eq:exten_1_mod}. Suppose that
$$\int_{Q_1} U(t,x,0)^2 \, dxdt + \int_{Q^*_1} y^a U^2 \,dXdt\leq 1.$$
Then for every $\varepsilon >0$ there exists $\delta = \delta( \varepsilon)>0$ such that if
$$\int_{Q_1}f^2\,dxdt + \int_{Q^*_1} y^a |F|^2\,dXdt + \int_{B_1} |A(x) - I|^2\,dx \leq \delta^2$$
where $I$ denotes the identity matrix, then there exists a weak solution $W$ to 
\begin{equation}\label{eq:harmonic}
\begin{cases}
y^a \partial_tW - \dvie(y^a \nabla W) = 0& \hbox{in}~Q^*_{3/4} \\ 
-y^a W_y\big|_{y=0} = 0 &\hbox{on}~Q_{3/4} 
\end{cases}
\end{equation}
such that
$$\int_{Q^*_{3/4}}y^a |U-W|^2 \, dXdt< \varepsilon^2.$$
\end{cor}

\begin{proof}
We will prove this by contradiction. Let us assume that there exists $\varepsilon >0$,
coefficients $A_k(x)$, data $f_k$, vector fields $F^k$ and solutions $U_k$  in $Q^*_1$, $k\geq1$, such that
$$ \int_{Q_1} U^2_k \, dxdt + \int_{Q^*_1} y^a U^2_k \,dX dt\leq 1, $$
$$\int_{Q_1}f_k^2\,dxdt + \int_{Q^*_1} y^a |F^k|^2\,dXdt + \int_{B_1} |A_k(x) - I|^2\,dx < \frac{1}{k^2} $$
and such that, for any weak solution $W$ to \eqref{eq:harmonic},
\begin{equation}\label{eq:contradiction}
    \int_{Q^*_{3/4}}y^a|U_k - W|^2 \, dXdt  \geq \varepsilon^2.
\end{equation}

If in Lemma \ref{lem:cacc_mod} we choose $\eta$ such that $\eta\equiv1$ in $Q^*_{3/4}$, $0\leq\eta\leq1$ in $Q^*_1$,
and we let $t_1\to-1$ and $t_2\to1$, then we find that
$$\int_{Q^*_{3/4}} y^a |\nabla U_k |^2 \, dXdt\leq C$$
for all $k\geq1$. Let us define $T_1 = - 9/16$, $T_2 = 9/16$. The previous estimate says that the sequence $\{U_k\}_{k=1}^\infty$ is bounded in $L^2([T_1,T_2];H^1(B^*_{3/4},y^adX))$. By the Aubin--Lions Lemma, this space is compactly embedded in $L^2([T_1,T_2];L^2(B^*_{3/4},y^adX))$, so that
there exists a subsequence, again denoted by $\{U_k\}_{k=1}^\infty$, and a function $U_\infty$ such that 
\begin{align*}
U_k \to U_{\infty} \text{ strongly in } L^2([T_1,T_2];L^2(B^*_{3/4},y^adX))\\
U_k \to U_{\infty} \text{ weakly in } L^2([T_1,T_2];H^1(B^*_{3/4},y^adX)).
\end{align*}
We show next that $U_{\infty}$ is a solution to \eqref{eq:harmonic} and this will give a contradiction to \eqref{eq:contradiction}.
Indeed, for any $k\geq1$ and any test function $\phi$,
\begin{align*}
-\int_{Q^*_{3/4}}y^aU_k \partial_t\phi\,dXdt &+ \int_{Q^*_{3/4}} y^a B_k(x) \nabla U_k \nabla \phi \, dXdt \\
&\qquad=\int_{Q_{3/4}}f_k(t,x) \phi(t,x,0) \, dxdt +\int_{Q^*_{3/4}} y^a F^k \nabla \phi \, dXdt.
\end{align*}
By letting $k \to \infty$, the equation above reduces to 
$$-\int_{Q^*_{3/4}}y^aU_{\infty} \partial_t \phi \,dXdt + \int_{Q^*_{3/4}} y^a \nabla U_{\infty} \nabla \phi \, dX dt=0$$
as desired.
\end{proof}

Similarly as with \eqref{eq:exten_1_mod}--\eqref{eq:exten_2_mod}, we can define the notion of weak solutions to
\begin{equation}\label{eq:boundaryQ1}
\begin{cases}
y^a \partial_t U - \dvie(y^a B(x) \nabla U)=- \dvie(y^a F)&\hbox{in}~(Q^+_1)^* \\
-y^a U_y\big|_{y=0}=f&\hbox{on}~Q_1^+ \\
U=0\quad\hbox{or}\quad\partial_AU=0&\hbox{on}~Q_1^*\cap\{x_n=0\}
\end{cases}
\end{equation}
with test functions $\phi$ such that $\phi=0$ on $\partial (Q^+_1)^* \backslash (Q_1^+ \times \{ 0 \})$ (for Dirichlet boundary condition),
or $\phi=0$ on $\partial (Q^+_1)^* \backslash [(Q_1^+ \times \{ 0 \})\cup (Q_1^*\cap\{x_n=0\})]$ (for Neumann boundary condition).
Then, exactly as with Corollary \ref{cor:approx}, we can prove the following approximation result up to the boundary.

\begin{cor}\label{cor:boundaryapprox}
Let $U$ be a weak solution to \eqref{eq:boundaryQ1}. Suppose that
$$\int_{Q_1^+} U(t,x,0)^2 \, dxdt + \int_{(Q^+_1)^*} y^a U^2 \,dXdt\leq 1.$$
Then for every $\varepsilon >0$ there exists $\delta = \delta( \varepsilon)>0$ such that if
$$\int_{Q_1^+}f^2\,dxdt + \int_{(Q^+_1)^*} y^a |F|^2\,dXdt + \int_{B_1^+} |A(x) - I|^2\,dx \leq \delta^2$$
where $I$ denotes the identity matrix, then there exists a weak solution $W$ to 
$$\begin{cases}
y^a \partial_tW - \dvie(y^a \nabla W) = 0& \hbox{in}~(Q^+_{3/4})^* \\ 
-y^a W_y\big|_{y=0} = 0 &\hbox{on}~Q_{3/4}^+ \\
W=0\quad\hbox{or}\quad\partial_{x_n}W=0&\hbox{on}~Q_{3/4}^*\cap\{x_n=0\}
\end{cases}$$
such that
$$\int_{(Q^+_{3/4})^*}y^a |U-W|^2 \, dXdt< \varepsilon^2.$$
\end{cor}

Next, we present the regularity of $W$.

\begin{prop}\label{Prop:Harmonic}
Let $W$ be a weak solution to 
\begin{equation}\label{harmonic_2}
\begin{cases}
y^a \partial_tW - \dvie(y^a \nabla W) = 0& \hbox{in}~Q^*_{1} \\ 
-y^a W_y\big|_{y=0} = 0 &\hbox{on}~Q_{1}.
\end{cases}
\end{equation}
Then following estimates hold.
\begin{enumerate}[$(1)$]
    \item For every integer $k\geq0$, multi-index $\beta\in\mathbb{N}_0^n$ and each $Q_r(t_0,x_0)\subset Q_1,$ we have
    $$\sup_{Q_{r/2}(t_0,x_0) \times [0, r/2) } |\partial^k_t D^{\beta}_x W| \leq \frac{C(n,s)}{r^{k+|\beta|}} \osc_{Q_{r}(t_0,x_0) \times [0, r)} W.$$
    \item For each $Q_r(t_0,x_0) \subset Q_1$,
    $$\max_{Q_{r/2}(t_0,x_0) \times [0, r/2)} |W| \leq C(r,n,s)\|W\|_{L^2(Q_r(t_0,x_0)\times[0,r),y^adXdt)}.$$ 
    \item We have 
    $$\sup_{(t,x)\in Q_{1/2}} |W_y(t,x,y)| \leq C(n,s)\|W\|_{L^2(Q_1^*,y^adXdt)}y\qquad\hbox{for all}~0\leq y<1/2.$$
\end{enumerate}
\end{prop}

\begin{proof}
The proof of $(1)$ follows as in the proof of Corollary 1.13 of \cite{Stinga-Torrea-SIAM}.

To prove (2), we see from \cite{Stinga-Torrea-SIAM}
and \cite{B-DLC-S} that $\tilde{W}(t,x,y)=W(t,x,|y|)$ is a weak solution to
$|y|^a\partial_t\tilde{W}-\dvie(|y|^a \nabla\tilde{W})=0$ in $Q_1\times (-1,1)$. Then, by \cite{Chiarenza-Serapioni}, $\tilde{W}$ is locally bounded
and controlled by its $L^2$-norm.

To prove (3), we see that, since the coefficients of the equation in \eqref{harmonic_2} are smooth in $Q_1^*$,
we can differentiate through to get
$$y^a \partial_t W - y^a \left(\Delta_x W + \frac{a}{y} W_y + W_{yy} \right) = 0 \qquad\hbox{in}~Q_1^*.$$
It is easy to check that $V=y^aW_y$ is a weak solution to
$$\begin{cases}
y^{-a}\partial_tV - \dvie(y^{-a} \nabla V) = 0& \hbox{in}~Q^*_{1} \\ 
V\big|_{y=0} = 0 &\hbox{on}~Q_{1}
\end{cases}$$
(the test functions for this equation vanish on $\partial Q_1^*$). Let
$$\tilde{V}(t,x,y)=
\begin{cases}
V(t,x,y)&\hbox{for}~y>0\\
-V(t,x,-y)&\hbox{for}~y\leq0.
\end{cases}$$
Then $\tilde{V}$ is a weak solution to the degenerate parabolic equation
$$|y|^a\partial_t\tilde{V}-\dvie(|y|^a\nabla\tilde{V}) = 0\qquad\hbox{for}~(t,x,y)\in Q_1\times(-1,1).$$
Since $|y|^a$ is a Muckenhoupt $A_2$-weight, it follows that $\tilde{V}$ is locally H\"older continuous \cite{Chiarenza-Serapioni}.
Therefore, $y^aW_y\to0$ locally uniformly as $y\to0^+$.
Now, by substituting $z = \left(\frac{y}{1-a}\right)^{1-a}$ in the equation for $W$ above, we find that
$$\partial_t W - \left( \Delta_x W + z^{\alpha} W_{zz}\right) = 0$$
for $z>0$ small, where $\alpha = - \frac{2a}{1-a}$. Additionally,
$y^aW_y=W_z$, so that $W$ is differentiable with respect to $z$ up to the boundary $z=0$, with $W_z\big|_{z=0}=0$.
Next, for $z>0$ small, by $(1)$ and $(2)$,
$$|W_{zz}|\leq\frac{| -\partial_t W + \Delta_x W|}{|z^{\alpha}|}\leq \frac{C}{|z|^{\alpha}}\|W\|_{L^2(Q_1^*,y^adXdt)}$$
which in turn implies that, for $z_0>0$ small,
$$|W_z(t,x,z_0)| = \left| \int^{z_0}_0 W_{zz}(t,x,z)\,dz\right| \leq  C\|W\|_{L^2(Q_1^*,y^adXdt)} z_0^{1-\alpha}$$
for all $(t,x)\in Q_{1/2}$. After transforming back to $y$ we get the final result. 
\end{proof}

\begin{cor}\label{cor:bdd_reg}
Let $W$ be a weak solution to
$$\begin{cases}
y^a \partial_tW - \dvie(y^a \nabla W) = 0& \hbox{in}~(Q^+_1)^* \\ 
-y^a W_y\big|_{y=0} = 0 &\hbox{on}~Q_1^+ \\
W=0\quad\hbox{or}\quad\partial_{x_n}W=0&\hbox{on}~Q_{1}^*\cap\{x_n=0\}.
\end{cases}$$
Then Proposition \ref{Prop:Harmonic} holds for this $W$ if we replace the cubes $Q$
by half-cubes $Q^+$ in all the estimates there.
\end{cor}

\begin{proof}
This is an immediate consequence of Proposition \ref{Prop:Harmonic}. Indeed, 
the odd reflection of $W$ with respect to $x_n$ (for Dirichlet boundary condition) and
the even reflection of $W$ with respect to $x_n$ (for Neumann boundary condition)
are weak solutions to \eqref{harmonic_2}. 
\end{proof}

\begin{lem}[Trace inequality]\label{lem:trace_new}
There exists a constant $C>0$, depending only on $n$ and $s$, such that,
for any $U\in L^2((-1,1);H^1_{L,a}(B_1^*))$,
\begin{multline*}
    \int^{r^2}_{-r^2} r^{2-2s} \norm{U(t,\cdot,0)}^2_{L^2(B_r)} dt \\
    \leq C  \int^{r^2}_{-r^2} y^a\big(\norm{U(t,\cdot,\cdot)}^2_{L^2(B^*_r,y^a dX)} + r^2 \norm{\nabla U(t,\cdot,\cdot)}^2_{L^2(B^*_r,y^a dX)}\big) \,dt
\end{multline*}
for all $0<r<1$. The same is true if we replace $B_r$ by $B^+_r$. 
\end{lem}

\begin{proof}
The general estimate follows by scaling from the case $r=1$. From \cite{Nekvinda}, we have that, for a.e $t \in (-1,1)$, 
$\norm{U(t,\cdot,0)}^2_{L^2(B_1)} \leq C \norm{U(t,\cdot,\cdot)}^2_{H^1(B^*_1,y^a dX)}$. Then we just integrate in time.
\end{proof}

%%%%%%%%%%%%%%%%%%%%%%%%%%%%%%%%%%%%%%%%%%%%%%%%%%%%%%
\section{Interior Regularity}\label{sec:regularity_interior}
%%%%%%%%%%%%%%%%%%%%%%%%%%%%%%%%%%%%%%%%%%%%%%%%%%%%%%

In this Section we prove Theorems \ref{thm:interiorholder} and \ref{thm:interiorLp}.

We say that a function $f\in L^2(Q_1)$ is in $L^{\alpha/2, \alpha}(0,0)$, for $0 < \alpha \leq 1$, whenever 
$$[f]^2_{L^{\alpha/2,\alpha}(0,0)} = \sup_{0 < r \leq 1} \frac{1}{r^{n + 2 + 2\alpha}} \int_{Q_r} |f - f(0,0)|^2 \,dt \, dx < \infty $$
where $\displaystyle f(0,0)= \lim_{r \rightarrow 0} \frac{1}{|Q_r|} \int_{Q_r} f(t,x) \, dt \, dx $.  
In view of Theorem \ref{thm:Campanato}, we see that if $f$ satisfies this property uniformly in balls centered at points close to the origin
then $f$ is parabolically $\alpha$-H\"older continuous at the origin. Futhermore, Theorem \ref{thm:interiorholder}
will follow directly from the following statement after rescaling and translation, and by using estimate \eqref{eq:L2vsHs}.

\begin{thm}\label{thm:thm_1_holder}
Let $u\in\Dom(H^s)$ be as in Theorem \ref{thm:interiorholder}, with $f\in L^2(\R\times\Omega)$.
Suppose that $B_1\subset\Omega$ and that $f \in L^{\alpha/2,\alpha}(0,0)$, for some $0<\alpha <1$.
\begin{enumerate}[$(1)$]
\item Assume that $ 0 < \alpha + 2s <1$. There exist $0< \delta < 1$, depending only on $n$, ellipticity, $\alpha$ and $s$, and a constant $C_1 >0$ such that if 
$$\sup_{0 < r \leq 1} \frac{1}{r^n} \int_{B_r} |A(x) - A(0)|^2 \,dx < \delta ^2 $$
then there exists a constant $c$ such that 
$$ \frac{1}{r^{n+2}}\int_{Q_r} |u(t,x)-c|^2 \, dt \, dx \leq C_1 r^{2(\alpha + 2s)} $$
for all $r>0$ small. Moreover,
$$|c|+C_1^{1/2}\leq C_0(\|u\|_{\Dom(H^s)}+|f(0,0)|+[f]_{L^{\alpha/2,\alpha}(0,0)})$$
where $C_0>0$ depends on $A(x)$, $n$, $s$, $\alpha$ and ellipticity.
\item Assume that $1< \alpha + 2s<2$. There exists $0< \delta < 1,$ depending only on $n$, ellipticity, $\alpha$ and $s$, and a constant $C_1 >0$ such that if 
$$\sup_{0 < r \leq 1} \frac{1}{r^{n+2(\alpha + 2s-1)}} \int_{B_r} |A(x) - A(0)|^2 \,dx < \delta ^2 $$
then there exists a linear function $\ell(x) = \mathcal{A} + \mathcal{B} \cdot x$ such that 
$$ \frac{1}{r^{n+2}}\int_{Q_r} |u(t,x)-\ell(x)|^2 \, dt \, dx \leq C_1 r^{2(\alpha + 2s)} $$
for all $r>0$ small. 
Moreover,
$$|\mathcal{A}|+|\mathcal{B}|+C_1^{1/2}\leq C_0(\|u\|_{\Dom(H^s)}+|f(0,0)|+[f]_{L^{\alpha/2,\alpha}(0,0)})$$
where $C_0>0$ depends on $A(x)$, $n$, $s$, $\alpha$ and ellipticity.
\end{enumerate}
\end{thm}

We say that a function $f\in L^2(Q_1)$ is in $L^{-s+\alpha/2, -2s+\alpha}(0,0)$, for $0 < \alpha < 1$, whenever 
$$[f]^2_{L^{-s+\alpha/2, -2s+\alpha}(0,0)} = \sup_{0 < r \leq 1} \frac{1}{r^{n + 2 + 2(-2s+\alpha)}} \int_{Q_r} |f(t,x)|^2 \,dt \, dx < \infty $$
and that is in $L^{-s+(1+\alpha)/2, -2s+\alpha+1}(0,0)$ whenever 
$$[f]^2_{L^{-s+(1+\alpha)/2, -2s+\alpha+1}(0,0)} = \sup_{0 < r \leq 1} \frac{1}{r^{n + 2 + 2(-2s+\alpha+1)}} \int_{Q_r} |f(t,x)|^2 \,dt \, dx < \infty .$$
Then we have the following consequences 
\begin{itemize}
    \item If $f\in L^2(Q_1)$ is also in $L^p(Q_1)$, for $(n+2)/(2s) < p < (n+2)/(2s-1)^+$, then 
    $[f]_{L^{-s+\alpha/2, -2s+\alpha}(0,0)} \leq C_n\norm{f}_{L^p(Q_1)}$, for $\alpha = 2s - (n+2)/p$.
    \item If $s >1/2$ and $f \in L^p(Q_1)$ for $p > (n+2)/(2s-1)$, then 
    $  [f]_{L^{-s+(1+\alpha)/2, -2s+\alpha+1}(0,0)} \leq C_n\norm{f}_{L^p(Q_1)} $, for $\alpha = 2s- (n+2)/p - 1$.
\end{itemize}
In view of these observations, Theorem \ref{thm:interiorLp} will follow immediately from the next result.

\begin{thm}\label{thm:thm_2_holder}
Let $u\in\Dom(H^s)$ be as in Theorem \ref{thm:interiorLp}, with $f\in L^2(\R\times\Omega)$.
Suppose that $B_1\subset\Omega$ and let $0<\alpha <1$.
\begin{enumerate}[$(1)$]
    \item Assume that $f \in L^{-s+\alpha/2, -2s+\alpha}(0,0)$. Then there exist  $0< \delta < 1$, depending only on $n,$ ellipticity, $\alpha, s$, and a constant $C_1>0$ such that if 
    $$\sup_{0 < r \leq 1} \frac{1}{r^n} \int_{B_r} |A(x) - A(0)|^2 dx < \delta ^2$$
    then there a exists  constant $c$ such that
    $$ \frac{1}{r^{n+2}}\int_{Q_r} |u(t,x)-c|^2 \, dt \, dx \leq C_1 r^{2 \alpha}$$
    for all $r>0$ small. Moreover,
    $$|c|+C_1^{1/2}\leq C_0\big(\|u\|_{\Dom(H^s)}+[f]_{L^{-s+\alpha/2, -2s+\alpha}(0,0)}\big)$$
where $C_0>0$ depends on $A(x)$, $n$, $s$, $\alpha$ and ellipticity.
    \item Assume that $f \in L^{-s+(1+\alpha)/2, -2s+\alpha+1}(0,0)$. Then there exist  $0< \delta < 1$, depending only on $n,$ ellipticity, $\alpha, s$, and a constant $C_1>0$ such that if 
    $$\sup_{0 < r \leq 1} \frac{1}{r^{n+2 \alpha}} \int_{B_r} |A(x) - A(0)|^2\, dx < \delta ^2$$
    then there exists a linear function $\ell(x) = \mathcal{A} + \mathcal{B} \cdot x$ such that
    $$ \frac{1}{r^{n+2}}\int_{Q_r} |u(t,x)-\ell(x)|^2 \, dt \, dx \leq C_1 r^{2(1+\alpha)} $$
    for all $r>0$ small. Moreover,
    $$|\mathcal{A}|+|\mathcal{B}|+C_1^{1/2}\leq C_0\big(\|u\|_{\Dom(H^s)}+[f]_{L^{-s+(1+\alpha)/2, -2s+\alpha+1}(0,0)}\big)$$
where $C_0>0$ depends on $A(x)$, $n$, $s$, $\alpha$ and ellipticity.
\end{enumerate}
\end{thm}

Therefore, the rest of this section is devoted to the proofs of Theorems \ref{thm:thm_1_holder} and \ref{thm:thm_2_holder}.

%%%%%%%%%%%%%%%%%%%%%%%%%%%%%%%%%%%%%%%%%%%%%%%%%%%%%%
\subsection{Proof of Theorem \ref{thm:thm_1_holder}(1)}
%%%%%%%%%%%%%%%%%%%%%%%%%%%%%%%%%%%%%%%%%%%%%%%%%%%%%%

In view of the extension problem characterization in Theorem \ref{th_exten},
we only need to prove the theorem for $u(t,x) = U(t,x,0)$, where $U$ is a solution to \eqref{eq:exten_1_mod} in $Q_1^\ast$ with $F\equiv0$.
We will consider normalized solutions $U$ as defined next. Without loss of generality,
we can assume that $A(0) = I$ and $f(0,0) =0$ (otherwise, one needs to take $U-\frac{y^{1-a}}{1-a}f(0,0)$).
Given $\delta>0$, we say that $U$ is a $\delta$-normalized solution if the following conditions hold:
\begin{enumerate}
    \item $\displaystyle\sup_{0 < r \leq 1} \frac{1}{r^n}
    \int_{B_r} |A(x) - I|^2 dx < \delta^2$;
    \item $\displaystyle [f]^2_{ L^{\alpha/2, \alpha}(0,0)} = \sup_{0 < r \leq 1} \frac{1}{r^{n + 2 + 2\alpha}} \int_{Q_r} |f|^2 \,dt \, dx  < \delta^2$;
    \item $\displaystyle \int_{Q_1}U(t,x,0)^2 \, dt \,dx + \int_{Q^*_1} y^a U^2 \, dt \, dX \leq 1$.
\end{enumerate}
Notice that $(1)$ can always be assumed by scaling, while $(2)$ and$(3)$ hold after normalizing
\begin{equation}\label{eq:normalization}
U(x,y)\bigg(\int_{Q_1}U(t,x,0)^2 \, dt \,dx + \int_{Q^*_1} y^a U^2 \, dt \, dX+\frac{1}{\delta} [f]^2_{ L^{\alpha/2, \alpha}(0,0)}\bigg)^{-1}.
\end{equation}

\begin{lem}\label{lem:induc_step_1}
Given $0 < \alpha + 2s < 1$, there exist $0< \delta, \lambda <1$ depending on $n$, $s$ and ellipticity, a constant $c$ and a universal constant $D>0$ such that,
for any $\delta$-normalized solution $U$ to \eqref{eq:exten_1_mod},
$$\frac{1}{\lambda^{n+2}} \int_{Q_{\lambda}} |U(t,x,0)-c|^2 \, dt \, dx + \frac{1}{\lambda^{n+3+a}} \int_{Q^*_{\lambda}} y^a|U-c|^2 \, dt \, dX < \lambda^{2(\alpha + 2s)} $$
and $|c| \leq D.$
\end{lem}

\begin{proof}
Let $0 < \varepsilon < 1$ be fixed. We use Corollary \ref{cor:approx} to get a function $W$ which satisfies \eqref{eq:harmonic}. Then,
since $U$ is a normalized solution,
$$\int_{Q^*_{1/2}} y^a|W|^2\, dt \, dX \leq 2 \int_{Q^*_{1/2}} y^a|U-W|^2\, dt \, dX + 2 \int_{Q^*_{1/2}}y^a U^2 \, dt \,dX \leq 2 \varepsilon^2 + 2 \leq 4.$$
Define $c = W(0,0,0)$. Hence, by Proposition \ref{Prop:Harmonic}$(2)$,
we get that $|c| \leq D$, for some universal constant $D$. Now, for any $(t,X) \in Q^*_{1/4}$, by Proposition \ref{Prop:Harmonic}, 
\begin{align*}
|W(t,X) -c| &\leq |W(t,x,y) - W(t,x,0)| + |W(t,x,0)-W(t,0,0)| + |W(t,0,0)-c| \\
&\leq  N(y^2+ |x| + |t|) \leq N( |X| + |t|^{1/2}) 
\end{align*}
for some universal constant $N>0$. Then for any $0< \lambda < 1/4$,
\begin{align*}
 \frac{1}{\lambda^{n+3+a}} &\int_{Q^*_{\lambda}}y^a |U-c|^2  \, dt \, dX \\ 
&\leq  \frac{2}{\lambda^{n+3+a}} \int_{Q^*_{\lambda}}y^a |U-W|^2 \, dt \, dX + \frac{2}{\lambda^{n+3+a}} \int_{Q^*_{\lambda}}y^a |W-c|^2 \, dt \, dX \\
&\leq   \frac{2 \varepsilon^2 }{\lambda^{n+3+a}} + \frac{2 N^2 }{\lambda^{n+3+a}} \int_{Q^*_{\lambda}} y^a(|X|^2+ |t|)\, dt \, dX \\ 
&\leq \frac{2 \varepsilon^2 }{\lambda^{n+3+a}} + c_{n,a} \lambda^2 .
\end{align*}

Next we apply the trace inequality of Lemma \ref{lem:trace_new} to $(U-c)$ to get
\begin{align*}
\lambda^{1+a} \int_{Q_{\lambda}} |U(t,x,0)-c|^2 \,dt\, dx &\leq C\int_{Q^*_{\lambda}}  y^a|U-c|^2 \,dt\, dX + C \lambda^2 \int_{Q^*_{\lambda}} y^a |\nabla U|^2 \,dt\, dX \\
&\leq 2C \varepsilon^2 + Cc_{n,a} \lambda^{n+5+a}+C \lambda^2 \int_{Q^*_{\lambda}} y^a |\nabla U|^2 \, dt\, dX.
\end{align*}

Now we estimate the last integral by applying Lemma \ref{lem:cacc_mod} to $(U-c)$. For this purpose, take $\eta$ such that $\eta = 1$ in $Q^*_\lambda$,
$\eta = 0$ outside $Q^*_{2 \lambda}$, and $|\partial_t \eta| + |\nabla \eta| \leq \frac{2}{\lambda}$ in $Q^*_{2 \lambda}$. Then
\begin{align*}
\lambda^2\int_{Q^*_{\lambda}} &y^a |\nabla U|^2\, dt \, dX \\ 
&\leq C \lambda^2\Bigg( \int_{Q^*_{2 \lambda}}  y^a \frac{1}{\lambda^2} |U-c|^2  \, dt \, dX + \int_{Q_{2\lambda}} |U(t,x,0)-c||f(t,x)|\, dt \,dx  \Bigg) \\ 
&\leq  C \int_{Q^*_{2 \lambda}} y^a|U-c|^2 dt \, dX +C \big( \norm{U(\cdot,\cdot,0)}_{L^2(Q_{2\lambda})}  + |c| |Q_{2 \lambda}|^{1/2} \big)\norm{f}_{L^2(Q_{2 \lambda})} \\
&\leq2C \varepsilon^2 + Cc_{n,a} \lambda^{n+5+a}+ C(1+ |c|)\delta.
\end{align*}

Thus, for any $0 < \lambda < 1/8 $, 
\begin{multline*}
\frac{1}{\lambda^{n+2}} \int_{Q_{\lambda}} |U(t,x,0)-c|^2 dt \, dx +\frac{1}{\lambda^{n+3+a}} \int_{Q^*_{\lambda}} y^a|U-c|^2 \, dt \, dX\\
  < \frac{C\varepsilon^2}{\lambda^{n+3+a}} + c_{n,a} \lambda^2 + \frac{C \delta}{\lambda^{n+3+a}}.
\end{multline*}
Next if we make $\lambda$ sufficiently small we have $c_{n,a}\lambda^2 \leq \frac{1}{3}\lambda^{2(\alpha + 2s)}$. Then we can choose $\varepsilon$
small such that $\frac{C\varepsilon^2}{\lambda^{n+3+a}}\leq\frac{1}{3}\lambda^{2(\alpha + 2s)}$. 
Finally, with this $\varepsilon$ in Corollary \ref{cor:approx}, we can let $\delta$ small enough such that
$C(1+ |c|)\delta\leq\frac{1}{3}\lambda^{2(\alpha + 2s)}$. 
\end{proof}

\begin{lem}\label{lem:induction}
Assume the conditions on Lemma \ref{lem:induc_step_1}.
Then there exist a sequence of constants $c_k$, $k\geq 0$, and a universal constant $D>0$ such that 
$$|c_k - c_{k+1}| \leq D \lambda^{k(\alpha + 2s)}$$ 
and 
$$\frac{1}{\lambda^{k(n+2)}}\int_{Q_{\lambda ^k}} |U(t,x,0)-c_k|^2 dt \, dx + \frac{1}{\lambda^{k(n+3+a)}}\int_{Q^*_{\lambda ^k}} y^a|U-c_k|^2 dt \, dX  < \lambda^{2k(\alpha + 2s)}$$
for all $k\geq0$.
\end{lem}

\begin{proof}
We prove this lemma by induction. First we consider the base $k=0$. We let $c_0=0$ and notice that the estimates on $U$ hold
because $U$ is a normalized solution. Next, we let $c_1$ be the constant $c$ from Lemma \ref{lem:induc_step_1},
so clearly the conclusion holds in this case. Now we assume that the lemma is true for some $k \geq 1$. We define 
$$\lU(t, X) = \frac{U(\lambda^{2k}t, \lambda^kX) - c_k}{\lambda^{k(\alpha + 2s)}}\qquad\hbox{for}~(t,X) \in Q^*_1.$$
Recall that, in particular, $U$ satisfies
$$-\int_{Q^*_{\lambda^k}}y^aU \partial_t \phi\,dX \, dt + \int_{Q^*_{\lambda^k}} y^a B(x) \nabla U \nabla \phi \, dX \, dt = \int_{Q_{\lambda^k}}f(t,x) \phi(t,x,0) \, dx \, dt$$
for suitable test functions $\phi$. Therefore, by changing variables here, it is easy to see that $\lU$ satisfies
$$-\int_{Q^*_1}y^a\lU \partial_t\tilde{\phi}\,d X \, dt + \int_{Q^*_1} y^a \tilde{B}(x) \nabla \lU \nabla \tilde{\phi} \, d X \, d t = \int_{Q_1}\tilde{f}(t, x) \tilde{\phi}( t, x,0) \, dx \, dt$$
where $\tilde{\phi}(t,X)=\phi(\lambda^{2k}t, \lambda^k X)$, $\tilde{B}(x) = B(\lambda^k x)$, $\tilde{f}(t,x) = \lambda^{-k\alpha} f(\lambda^{2k}t, \lambda^k x)$.
Furthermore, $\tilde{A}(0) = I$, $\tilde{f}(0,0) = 0$ and, by changing variables and using the induction hypotheses,
$$\frac{1}{r^{n}} \int_{B_r} (\tilde{A}(x) -I)^2\, dx+\frac{1}{r^{n+2+ 2 \alpha}} \int_{Q_r} |\tilde{f}(t, x)|^2 \, dt\,dx < \delta^2$$
and
$$\int_{Q_1} \tilde{U}(t,x,0)^2 dx\, dt + \int_{Q_1^*} y^a \tilde{U}^2 dX\, dt \leq1.$$
In other words, $\lU$ is a $\delta$-normalized weak solution to
$$\begin{cases}
y^a \partial_t \tilde{U} - \dvie(y^a \tilde{B}(x) \nabla \lU)= 0&\hbox{in}~Q^*_1 \\ 
-y^a \lU_y|_{y=0}= \tilde{f}&\hbox{on}~Q_1.
\end{cases}$$
Thus we can apply Lemma \ref{lem:induc_step_1} to $\lU$ to get the existence of a constant $c$ such that
$$\frac{1}{\lambda^{n+2}} \int_{Q_{\lambda}} |\lU(t,x,0)-c|^2 \, dt \, dx + \frac{1}{\lambda^{n+3+a}} \int_{Q^*_{\lambda}} y^a|\lU-c|^2 \, dt \, dX  < \lambda^{2(\alpha + 2s)}.$$
If we change variables back we obtain
\begin{align*}
 \frac{1}{\lambda^{(n+2)(k+1)}}& \int_{Q_{\lambda^{k+1}}} |U(t,x,0)-c_k - c \lambda^{k(\alpha+ 2s)}|^2 \, dt \, dx \\
   &+ \frac{1}{\lambda^{(n+3+a)(k+1)}} \int_{Q_{\lambda^{k+1}}} y^a|U-c_k - c \lambda^{k(\alpha+ 2s)}|^2 \, dt \, dX< \lambda^{2(k+1)(\alpha +2s)}.
\end{align*}
Defining $c_{k+1} = c_k + \lambda^{k(\alpha + 2s)} c$ we see that $|c_{k+1} - c_k| \leq D \lambda^{k(\alpha + 2s)}$.
\end{proof}

\begin{proof}[Proof of Theorem \ref{thm:thm_1_holder}{$(1)$}]
If $\{c_k\}_{k\geq0}$ is the sequence of constants from Lemma \ref{lem:induction} then we see that $c_{\infty} = \lim_{k \rightarrow \infty} c_k$ exists
and is finite. Indeed, to show that $\{c_k\}_{k\geq0}$ is a Cauchy sequence of real numbers,
let $m,k\geq0$ and suppose that $m=k+j$ for some $j\geq1$. Then
\begin{align*}
|c_k-c_m| &= |c_k-c_{k+j}| \leq \sum_{\ell=0}^{j-1}|c_{k+\ell}-c_{k+\ell+1}| \\
&\leq D\sum_{\ell=0}^{j-1}\lambda^{(k+\ell)(\alpha+2s)} \leq D\lambda^{k(\alpha+2s)}\sum_{\ell=0}^\infty\lambda^{\ell(\alpha+2s)} \\
&\leq C(D,\lambda,\alpha,s)\lambda^{k(\alpha+2s)}\to0\qquad\hbox{as}~k\to\infty.
\end{align*}
Given any $0<r < 1/8$, let $k\geq 0$ such that $ \lambda^{k+1}< r \leq \lambda^k$. Then, by Lemma \ref{lem:induction},
\begin{align*}
\frac{1}{r^{n+2}}& \int_{Q_r} |U(t,x,0)-c_{\infty}|^2\,dt\,dx  \\
&\leq \frac{2}{r^{n+2}} \int_{Q_r} |U(t,x,0)-c_k|^2 \,dt\,dx + 2C_n|c_k - c_{\infty}|^2 \\
&\leq \frac{2}{\lambda^{n+2}}\frac{1}{\lambda^{k(n+2)}}\int_{Q_{\lambda^{k}}} |U(t,x,0)-c_k|^2 \,dt\,dx+\frac{C_n}{(1-\lambda^{\alpha+2s})^2}D^2 \lambda^{2k(\alpha + 2s)} 
\leq C_1r^{2(\alpha + 2s)} 
\end{align*}
where $C_1= C_1(n,\lambda,D,\alpha,s)>0$.
\end{proof}

%%%%%%%%%%%%%%%%%%%%%%%%%%%%%%%%%%%%%%%%%%%%%%%%%%%%%
\subsection{Proof of Theorem \ref{thm:thm_1_holder}(2)}\label{subsec: thm_interior}
%%%%%%%%%%%%%%%%%%%%%%%%%%%%%%%%%%%%%%%%%%%%%%%%%%%%%

As before, we will prove Theorem \ref{thm:thm_1_holder}(2) for $u(t,x) = U(t,x,0)$, where $U$ is a solution to \eqref{eq:exten_1_mod} in $Q_1^\ast$.
We will consider normalized solutions $U$ as defined next. Again, without loss of generality,
we can assume that $A(0) = I$ and $f(0,0) =0$.
Given $\delta>0$, we say that $U$ is a $\delta$-normalized solution (with $F$ not identically $0$) if the following conditions hold:
\begin{enumerate}
    \item $\displaystyle\sup_{0 < r \leq 1} \frac{1}{r^{n+2(\alpha+2s-1)}}
    \int_{B_r} |A(x) - I|^2 dx < \delta^2$;
    \item $\displaystyle [f]^2_{ L^{\alpha/2, \alpha}(0,0)} = \sup_{0 < r \leq 1} \frac{1}{r^{n + 2 + 2\alpha}} \int_{Q_r} |f|^2 \,dt \, dx  < \delta^2$;
    \item $\displaystyle\sup_{0<r\leq1}\frac{1}{r^{n+3+a+2(\alpha+2s-1)}}\int_{Q_r^*}y^a|F|^2\,dt\,dX<\delta^2$;
     \item $\displaystyle \int_{Q_1}U(t,x,0)^2 \, dt \,dx + \int_{Q^*_1} y^a U^2 \, dt \, dX \leq 1$.
\end{enumerate}
Notice that $(1)$ can always be assumed by scaling, and $(2)$, $(3)$ and $(4)$ hold after an
appropriate normalization, see \eqref{eq:normalization}.

\begin{lem}\label{lem:poly_step_1}
Given $1 < \alpha + 2s < 2$, there exist $0< \delta, \lambda <1$ depending on $n$, $s$ and ellipticity,
a linear function $\ell(x) =\mathcal{A} + \mathcal{B} \cdot x$ and a universal constant $D>0$
such that for any $\delta$-normalized solution $U$ to \eqref{eq:exten_1_mod},
$$\frac{1}{\lambda^{n+2}} \int_{Q_{\lambda}} |U(t,x,0)-\ell(x)|^2 \, dt \, dx + \frac{1}{\lambda^{n+3+a}} \int_{Q^*_{\lambda}} y^a|U-\ell(x)|^2 \, dt \, dX < \lambda^{2(\alpha + 2s)} $$
and $|\mathcal{A}| + |\mathcal{B}| \leq D.$
\end{lem}

\begin{proof}
Let $0< \varepsilon < 1.$ Then, as in Lemma \ref{lem:induc_step_1}, there exists a
function $W$ which satisfies Corollary \ref{cor:approx}, the smoothness estimates of Proposition \ref{Prop:Harmonic} and also
$$\int_{Q^*_{1/2}} y^a|W|^2  \, dt \, dX \leq 4.$$
Now define $$\ell(x) = W(0,0,0) + \nabla_x W(0,0,0) \cdot x=\mathcal{A}+\mathcal{B}\cdot x.$$
By Proposition \ref{Prop:Harmonic}, there exists a universal constant $D$ such that $|\mathcal{A}|+|\mathcal{B}| \leq D$.
Next, for any $(t,X) \in Q^*_{1/4}$ we have, for some universal constant $N>0$,
\begin{align*}
|W(t,x,y)-\ell(x)| &\leq |W(t,x,y)-W(t,x,0)| + |W(t,x,0)-W(0,x,0) | \\ 
&\quad+| W(0,x,0) - W(0,0,0) - \nabla_x W(0,0,0) \cdot x| \\ 
&\leq C|W_y(t,x,\xi)|y + Ct+C|x|^2 \\
&\leq C\xi y + Ct+C|x|^2\leq N(|X|^2 + t)
\end{align*}
where we used the mean value theorem for some $0\leq\xi\leq y$ and Proposition \ref{Prop:Harmonic}(3).
Then, for any $0< \lambda < 1/4,$ 
\begin{align*}
\frac{1}{\lambda^{n+3+a}}& \int_{Q^*_{\lambda}} y^a|U-\ell(x)|^2  \, dt \, dX \\ 
&\leq \frac{2}{\lambda^{n+3+a}} \int_{Q^*_{\lambda}} y^a|U-W|^2  \, dt \, dX + \frac{2}{\lambda^{n+3+a}} \int_{Q^*_{\lambda}} y^a|W-\ell(x)|^2  \, dt \, dX \\ 
&\leq   \frac{2 \varepsilon^2 }{\lambda^{n+3+a}} + \frac{2 N^2 }{\lambda^{n+3+a}} \int_{Q^*_{\lambda}}y^a (|X|^4+ |t|^2) \, dt \, dX \\ 
&\leq \frac{2 \varepsilon^2 }{\lambda^{n+3+a}} + c_{n,a} \lambda^4 .
\end{align*}

In the next step, we apply the trace inequality (Lemma \ref{lem:trace_new}) to $U-\ell$. Hence, for $0< \lambda < 1/8,$
\begin{multline*}
\lambda^{1+a} \int_{Q_{\lambda}} |U(t,x,0)-\ell(x)|^2\, dt\, dx \\
\leq C \int_{Q^*_{\lambda}} y^a|U-\ell(x)|^2 \,dt\, dX +  C\lambda^2 \int_{Q^*_{\lambda}} y^a|\nabla (U- \ell)|^2 \, dt\, dX.
\end{multline*}

Observe that $V=U-\ell$ is a weak solution to 
$$\begin{cases}
y^a \partial_t V - \dvie(y^a B(x) \nabla V) = - \dvie(y^a(F+G) )&\hbox{in}~Q^*_1 \\
-y^aV_y|_{y=0}=f &\hbox{on}~Q_1
\end{cases}$$
where the vector field $G$ is given by $G= ((I-A(x))\nabla_x \ell, 0 )$ and $G(0)=0$.
Thus, by Lemma \ref{lem:cacc_mod},
\begin{align*}
\int_{Q^*_{\lambda}} |\nabla (U-\ell)|^2 y^a\, dt \, dX &\leq C \int_{Q^*_{2 \lambda}} y^a \bigg(\frac{1}{\lambda^2}|U-\ell|^2 + |F+G|^2 \bigg) \, dt \, dX \\ 
& \quad    + C\int_{Q_{2\lambda}} |U(t,x,0)-\ell(x)||f(t,x)|\, dt \,dx \\
&\leq  \frac{C}{\lambda^2} \int_{Q^*_{2 \lambda}} y^a|U-\ell|^2\, dt \, dX + C\norm{F+G}_{L^2(Q^*_{2 \lambda},y^adtdX)}^2 \\ 
& \quad + C \big( \norm{U(\cdot,\cdot,0)}_{L^2(Q_{2\lambda})}  + \norm{\ell}_{L^2(Q_{2 \lambda})} \big)\norm{f}_{L^2(Q_{2 \lambda})} \\ 
&\leq \frac{C}{\lambda^2} \int_{Q^*_{2 \lambda}}y^a|U-\ell|^2\, dt \, dX  + C\delta^2\lambda^{n+3+a+2(\alpha+2s-1)}\\
&\quad + C(1+ D)\delta \\
&\leq \frac{C}{\lambda^2} \int_{Q^*_{2 \lambda}}y^a|U-\ell|^2\, dt \, dX+ C\delta.
\end{align*}
Thus,
\begin{multline*}
\frac{1}{\lambda^{n+2}} \int_{Q_{\lambda}} |U(t,x,0)-\ell(x)|^2\, dt \, dx + \frac{1}{\lambda^{n+3+a}} \int_{Q^*_{\lambda}} y^a|U-\ell(x)|^2 \, dt \, dX\\
 \leq \frac{C\varepsilon^2}{\lambda^{n+3+a}} + c_{n,a} \lambda^4 + \frac{C \delta}{\lambda^{n+3+a}}<\lambda^{2(\alpha+2s)}
\end{multline*}
where the last inequality follows by first choosing $\lambda$ small, then $\varepsilon$ sufficiently small and, for this $\varepsilon>0$,
a $0<\delta<1$ in Lemma \ref{cor:approx} small enough.
\end{proof}

\begin{lem}\label{lem:poly_induction}
Assume the conditions on Lemma \ref{lem:poly_step_1}.
Then there exist a sequence of linear functions $\ell_k(x) = \mathcal{A}_k + \mathcal{B}_k \cdot x$,
$k\geq 0$, and a universal constant $D>0$ such that 
$$|\mathcal{A}_k - \mathcal{A}_{k+1}|+\lambda^k|\mathcal{B}_k - \mathcal{B}_{k+1}| \leq D \lambda^{k(\alpha + 2s)}$$ 
and
$$\frac{1}{\lambda^{k(n+2)}}\int_{Q_{\lambda ^k}} |U(t,x,0)-\ell_k|^2 \,dt \, dx + \frac{1}{\lambda^{k(n+3+a)}}\int_{Q^*_{\lambda ^k}} y^a|U-\ell_k|^2\, dt \, dX  < \lambda^{2k(\alpha + 2s)}$$
for all $k\geq0$.
\end{lem}

\begin{proof}
The proof is by induction. For the base step $k=0$, we set $\ell_0(x)  = 0$ and hence the estimates on $U$ are true because $U$ is a $\delta$-normalized solution.
For $k=1$ we choose $\ell_1(x) = \ell(x)$ from Lemma \ref{lem:poly_step_1} and obviously the conclusion holds.
Suppose the result is true for some $k \geq 1$. Define
$$\lU(t, X) = \frac{U(\lambda^{2k}t, \lambda^kX) - \ell_k(\lambda^kx)}{\lambda^{k(\alpha + 2s)}}\qquad\hbox{for}~(t,X) \in Q^*_1.$$
Recall that $U$ satisfies
\begin{multline*}
\int_{Q^*_{\lambda^k}}y^a U \partial_t \phi \, dt\, dX  + \int_{Q^*_{\lambda^k}} y^a B(x) \nabla U \nabla \phi \,dt \, dX \\
 = \int_{Q_{\lambda^k}^*}y^aF\nabla\phi\,dt\,dX+\int_{Q_{\lambda^k}}f(t,x) \phi(t,x,0) \,dt \, dx
\end{multline*}
for suitable test functions $\phi$.
Now, by the change of variables $X = \lambda^k X, t = \lambda^{2k} t$, we find that $\tilde{U}$ is a weak solution to
$$\begin{cases}
y^a \partial_t \tilde{U} - \dvie(y^a \tilde{B}(x) \nabla \lU)= - \dvie(y^a(\tilde{F}+\tilde{G}))&\hbox{in}~Q^*_1 \\ 
-y^a \lU_y|_{y=0} = \tilde{f}&\hbox{on}~Q_1
\end{cases}$$
where $\tilde{B}(x) = B(\lambda^k x)$, $\tilde{F}(t,X)=\lambda^{-k(\alpha+2s-1)}F(\lambda^{2k}t,\lambda^kX)$,
$\tilde{f}(t,x) = \lambda^{-k \alpha} f(\lambda^{2k}t, \lambda^k x) $
and
$$\tilde{G} = \bigg( \frac{I-\tilde{B}(x)}{\lambda^{k(\alpha + 2s-1)}} \nabla_x \ell_k(\lambda^k x), 0\bigg)\qquad\hbox{with}~\tilde{G}(0) = 0.$$
Moreover, by the hyptheses on $f$, $A(x)$ and $F$,
$$\frac{1}{r^{n+2+2 \alpha}} \int_{Q_r} |\tilde{f}|^2 \, dt \, dx  < \delta^2$$
and
\begin{align*}
\frac{1}{r^{n+3+a+2(\alpha + 2s-1)}} &\int_{Q^*_r} y^a |\tilde{F}+\tilde{G}|^2 \, dt\, dX  \\ 
& \leq  \frac{2}{(\lambda^k r)^{n+3+a+2(\alpha + 2s-1)}}  \int_{Q^*_{\lambda^kr}} y^a (|F|^2+|I-B(x)|^2|\mathcal{B}_k|^2)\, dt\, dX   \\
&\leq 2(1+D^2C^2) \delta^2
\end{align*}
where we used that
$$ |\mathcal{B}_k| \leq \sum^k_{j=1} |\mathcal{B}_j -\mathcal{B}_{j-1}| \leq D \sum^{\infty}_{j=0} \lambda^{j(\alpha + 2s-1)} \leq DC$$
Additionally, by changing variables and the induction hypothesis,
$$\int_{Q_1} \tilde{U}(t,x,0)\, dt \, dx + \int_{Q^*_1} y^a \tilde{U}^2 \,dt \, dX\leq 1$$
so that $\tilde{U}$ is a $\delta$-normalized solution.
Whence, by Lemma \ref{lem:poly_step_1}, there exists a linear function $\ell(x)$ such that 
$$ \frac{1}{\lambda^{n+2}} \int_{Q_{\lambda}} |\tilde{U}(t,x,0) - \ell(x)|^2\,dt\, dx +  \frac{1}{\lambda^{n+3+a}} \int_{Q^*_{\lambda}} y^a|\tilde{U} - \ell|^2 \,dt\, dX < \lambda^{2(\alpha + 2s)}.$$
By changing variables back,
\begin{multline*}
 \frac{1}{\lambda^{(k+1)(n+2)}} \int_{Q_{\lambda^{k+1}}} |U(t,x,0) - \ell_{k+1}(x)|^2 \, dt \, dx \\ 
+\frac{1}{\lambda^{(k+1)(n+3+a)}} \int_{Q^*_{\lambda^{k+1}}}y^a|U - \ell_{k+1}|^2  \, dt \, dX  < \lambda^{2(k+1)(\alpha + 2s)}
\end{multline*}
where $\ell_{k+1}(x) = \ell_k(x) + \lambda^{k(\alpha+2s)} \ell(\lambda^{-k} x)$. Then $$|\ell_{k+1}(x) - \ell_k(x)|= \lambda^{k(\alpha + 2s)} |\ell(\lambda^{-k} x)| \leq D \lambda^{k(\alpha + 2s)}(1+ \lambda^{-k}|x|)$$
so that $|\mathcal{A}_{k+1} - \mathcal{A}_k| = |\ell_{k+1}(0) - \ell_k(0)| \leq D \lambda^{k(\alpha + 2s)} $ and, by construction,
$|\mathcal{B}_{k+1}- \mathcal{B}_k| \leq \lambda^{k(\alpha+2s-1)}|\mathcal{B}| \leq D \lambda^{k(\alpha+2s-1)} $
\end{proof}

\begin{proof}[Proof of Theorem \ref{thm:thm_1_holder}$(2)$]
It follows the same procedure as the proof of Theorem \ref{thm:thm_1_holder}$(1)$,
but instead we need to use now Lemmas \ref{lem:poly_step_1} and \ref{lem:poly_induction}. 
\end{proof}

%%%%%%%%%%%%%%%%%%%%%%%%%%%%%%%%%%%%%%%%%%%%%%%%%%%%%
\subsection{Proof of Theorem \ref{thm:thm_2_holder} }
%%%%%%%%%%%%%%%%%%%%%%%%%%%%%%%%%%%%%%%%%%%%%%%%%%%%%

The proof follows very similar lines to those for Theorem \ref{thm:thm_1_holder} with minor changes.
Indeed, in the proof of Theorem \ref{thm:thm_1_holder}$(1)$ we need to replace the exponent $\alpha$
by $-2s+\alpha$, while in the proof of Theorem \ref{thm:thm_1_holder}$(2)$ we substitute the exponent $\alpha$
by $-2s+\alpha+1$. Notice also that we do not need the normalization $f(0,0)=0$.

%%%%%%%%%%%%%%%%%%%%%%%%%%%%%%%%%%%%%%%%%%%%%%%%%%%%%
\section{Boundary regularity for fractional heat equations}\label{sec:regularity_bdd_heat}
%%%%%%%%%%%%%%%%%%%%%%%%%%%%%%%%%%%%%%%%%%%%%%%%%%%%%

In this Section we perform a detailed analysis of boundary regularity and asymptotic behavior of half space
solutions for master equations driven by fractional powers of heat operators.
First we state known estimates for the fractional heat operator from \cite{Stinga-Torrea-SIAM}.
In the following we let $\Lambda^{1/2,1}(\R^{n+1})$ be the H\"older--Zygmund space of continuous
functions $u=u(t,x)$ such that the norm
$$\|u\|_{\Lambda^{1/2,1}(\R^{n+1})}=\|u\|_{L^\infty(\R^{n+1})}+\sup_{(t,x),(\tau,z)\in\R^{n+1}}
\frac{|u(\tau,x-z)+u(\tau,x+z)-2u(t,x)|}{|t-\tau|^{1/2}+|z|}$$
is finite.
 
\begin{prop}\label{prop:global_bdd_diri}
Let $u,f\in L^\infty(\R^{n+1})$ be such that
$$(\partial_t-\Delta)^{s}u=f\qquad\hbox{in}~\R^{n+1}.$$
\begin{enumerate}[$(1)$]
    \item Suppose that $f \in C^{\alpha/2, \alpha}(\R^{n+1})$ for $0< \alpha \leq 1.$
    \begin{enumerate}[$(a)$]
        \item If $\alpha + 2s$ is not an integer then $u \in C^{\alpha/2 + s, \alpha+2s}(\R^{n+1})$, with the estimate 
         $$\norm{u}_{C^{\alpha/2 + s, \alpha+2s}(\R^{n+1})} \leq C\big( \norm{f}_{C^{\alpha/2, \alpha}(\R^{n+1})} + \norm{u}_{L^{\infty}(\R^{n+1})} \big).$$
        \item If $\alpha + 2s=1$ then $u(t,x)$ is in the H\"older--Zygmund space $\Lambda^{1/2,1}(\R^{n+1})$, with the estimate 
        $$\norm{u}_{\Lambda^{1/2,1}(\R^{n+1})}\leq C \big( \norm{f}_{C^{\alpha/2, \alpha}(\R^{n+1})} + \norm{u}_{L^{\infty}(\R^{n+1})} \big).$$
    \end{enumerate}
    The constants $C>0$ above depend only on $n$, $s$ and $\alpha$.
    \item Suppose that $f \in L^{\infty}(\R^{n+1})$.
        \begin{enumerate}[$(a)$]
        \item If $s\neq1/2$ then $u \in C^{s, 2s}(\R^{n+1})$, with the estimate
        $$\norm{u}_{C^{s, 2s}(\R^{n+1})} \leq C \big(\norm{f}_{L^{\infty}(\R^{n+1})} + \norm{u}_{L^{\infty}(\R^{n+1})} \big).$$ 
        \item If $s = 1/2$ then $u$ is in the H\"older-Zygmund space $\Lambda^{1/2, 1}(\R^{n+1})$, with the estimate 
        $$\norm{u}_{\Lambda^{1/2,1}(\R^{n+1})} \leq C \big( \norm{f}_{L^{\infty}(\R^{n+1})} + \norm{u}_{L^{\infty}(\R^{n+1})} \big).$$
    \end{enumerate}
    The constants $C>0$ above depend only on $n$ and $s$.
\end{enumerate}
\end{prop}

%%%%%%%%%%%%%%%%%%%%%%%%%%%%%%%%%%%%%%%%%%%%%%%%%%%%%
\subsection{Boundary regularity in the half space -- Dirichlet}
%%%%%%%%%%%%%%%%%%%%%%%%%%%%%%%%%%%%%%%%%%%%%%%%%%%%%

In the half space $\R \times \R^n_+$ we consider the heat operator $\partial_t - \Delta^+_D$,
where $\Delta_D^+$ is the Dirichlet Laplacian in $\R^n_+= \{x\in\R^n: x_n > 0 \}$.
For a function $u(t,x)$ defined on $\R \times \overline{\R^n_+}$ with $u(t,x',0) = 0$ and $0<s<1$ we define
\begin{align*}
(\partial_t-\Delta^+_D)^su(t,x) &= \frac{1}{\Gamma(-s)} \int^{\infty}_0 \big(e^{\tau \Delta^+_D}u(t-\tau,x) - u(t,x) \big) \frac{d \tau}{\tau^{1+s}}
\end{align*}
where $\{e^{\tau\Delta_D^+}\}_{\tau\geq0}$ is the semigroup generated by $\Delta_D^+$.
Let $x^* = (x',-x_n)$ for $x \in \R^n$ and $u_0(t,x)$ be the odd extension of $u(t,x)$ about the $x_n$ axis given by
$$u_0(t,x)=
\begin{cases}
u(t,x)&\hbox{if}~x_n \geq 0 \\
-u(t,x^*)=-u(t,x',-x_n)&\hbox{if}~x_n <0.
\end{cases}$$
Now
\begin{align*}
e^{\tau \Delta^+_D} u(t-\tau,x) &= e^{\tau \Delta} u_0(t-\tau,x) \\
&= \frac{1}{(4 \pi \tau)^{n/2}} \int_{\R^n_+}\left( e^{-|x-z|^2/(4 \tau)}- e^{-|x-z^*|^2/(4 \tau)} \right)u(t-\tau,z)\, dz
\end{align*}
for any $\tau>0$, $x \in \R^n_+$. Hence, for $x\in \R^n_+,$
\begin{multline*}
(\partial_t-\Delta^+_D)^su(t,x)= \\
\frac{1}{(4 \pi)^{n/2} \Gamma(-s)}\int^{\infty}_0 \int_{\R^n_+} \left(\frac{e^{-|x-z|^2/(4 \tau)}- e^{-|x-z^*|^2/(4 \tau)}}{\tau^{n/2+1+s}}\right) (u(t-\tau,z) - u(t,x))\, dz\,d\tau
\end{multline*}
and
\begin{multline*}
(\partial_t - \Delta^+_D)^{-s} f(t,x) = \frac{1}{\Gamma(s)} \int^{\infty}_0 e^{\tau\Delta^+_D }f(t-\tau,x) \frac{d \tau }{\tau^{1-s}} \\
=  \frac{1}{(4\pi)^{n/2}\Gamma(s)} \int^{\infty}_0\int_{\R^n_+}\left(\frac{e^{-|x-z|^2/(4 \tau)}- e^{-|x-z^*|^2/(4 \tau)}}{\tau^{n/2+1-s}} \right)f(t-\tau,z) \,dz\,d\tau.
\end{multline*}

\begin{thm}[Boundary regularity in half space -- Dirichlet]\label{thm:global_bdd_diri}
Let $u,f\in L^\infty(\R\times\R^n_+)$ be such that
$$\begin{cases}
(\partial_t - \Delta^+_D)^s u = f&\hbox{in}~\R \times \R^n_+ \\ 
u= 0&\hbox{on}~\R \times \partial \R^n_+=\R\times \{x\in\R^n:x_n =0 \}.
\end{cases}$$
\begin{enumerate}[$(1)$]
    \item Suppose that $f \in C^{\alpha/2,\alpha}(\R \times \overline{\R^n_+})$ for some $0< \alpha \leq 1$.
     In addition, assume that $f(t,x',0) = 0$, for all $t\in\R$, $x' \in \R^{n-1}$.
     \begin{enumerate}[$(a)$]
        \item If $\alpha + 2s$ is not an integer then $u \in C^{\alpha/2 + s, \alpha+2s}(\R \times \overline{\R^n_+})$, with the estimate 
         $$\norm{u}_{C^{\alpha/2 + s, \alpha+2s}(\R \times \overline{\R^n_+})} \leq C\big( \norm{f}_{C^{\alpha/2, \alpha}(\R \times \overline{\R^n_+})}
          + \norm{u}_{L^{\infty}(\R \times \overline{\R^n_+})} \big).$$
        \item If $\alpha + 2s=1$ then $u(t,x)$ is in the H\"older--Zygmund space $\Lambda^{1/2,1}(\R \times \overline{\R^n_+}),$ with the estimate 
        $$\norm{u}_{\Lambda^{1/2,1}(\R \times \overline{\R^n_+})}\leq C \big( \norm{f}_{C^{\alpha/2, \alpha}(\R \times \overline{\R^n_+})}
        + \norm{u}_{L^{\infty}(\R \times \overline{\R^n_+})} \big).$$   
       \end{enumerate}
          The constants $C>0$ above depend only on $n$, $s$ and $\alpha$.
    \item Let $f \in L^{\infty}(\R \times \overline{\R^n_+})$.
    \begin{enumerate}
        \item If $s\neq1/2$ then $u \in C^{s, 2s}(\R \times \overline{\R^n_+}) $ with the estimate
        $$\norm{u}_{C^{s, 2s}(\R \times \overline{\R^n_+})}
        \leq C \left( \norm{f}_{L^{\infty}(\R \times \overline{\R^n_+})} + \norm{u}_{L^{\infty}(\R \times \overline{\R^n_+})} \right) $$ 
        \item If $s = 1/2$ then $u$ is in parabolic H\"older-Zygmund space $\Lambda^{1/2, 1}(\R \times \overline{\R^n_+})$ with the estimate 
        $$\norm{u}_{\Lambda^{1/2,1}(\R \times \overline{\R^n_+})} \leq 
        C \left( \norm{f}_{L^{\infty}(\R \times \overline{\R^n_+})} + \norm{u}_{L^{\infty}(\R \times \overline{\R^n_+})} \right) $$
    \end{enumerate}
        The constants $C>0$ above depend only on $n$ and $s$.
\end{enumerate}
\end{thm}

\begin{proof}
This result follows by observing that if $f_0$ and $u_0$ are the odd reflections of $f$ and $u$ with respect to the variable $x_n$, respectively,
then $(\partial_t - \Delta)^s u_0=f_0$ in $\R^{n+1}$. Thus we can invoke Proposition \ref{prop:global_bdd_diri}.
From the pointwise formula we see that
$(\partial_t - \Delta)^s u_0(t,x) = (\partial_t - \Delta^+_D)^s u(t,x) = f(t,x)=f_0(t,x)$ when $x \in \R^n_+$. Now, for some $(t,x)$ such that $x_n <0$ we have
\begin{align*}
(\partial_t- \Delta)^s u_0(t,x) &= \frac{1}{\Gamma(-s)} \int^{\infty}_0 \left(e^{\tau \Delta} u_0(t,x) - u_0(t,x) \right)\frac{d \tau}{\tau^{1+s}} \\ 
&= \frac{1}{\Gamma(-s)}  \int^{\infty}_0 \left( u(t,x^*) - e^{\tau \Delta^+_D} u(t,x^*)\right) \frac{d \tau}{\tau^{1+s}} \\ 
&= - (\partial_t - \Delta^+_D)^s u(t,x^*)= - f(t,x^*)=f_0(t,x).
\end{align*}
Also we can see that if $(t,x)$ is such that $x_n = 0$ then $u_0(t,x)=0$ and
$$(\partial_t - \Delta)^s u_0(t,x)=\frac{1}{(4\pi)^{n/2} |\Gamma(-s)|}\int^{\infty}_0 \int_{\R^n}\frac{e^{-|z|^2/(4 \tau)}}{\tau^{n/2+1+s}}u_0(t-\tau,x-z)\,dz\,d \tau=0$$
because $u_0(t-\tau,x-z)$ is an odd function in the variable $z_n$.
\end{proof}

%%%%%%%%%%%%%%%%%%%%%%%%%%%%%%%%%%%%%%%%%%%%%%%%%%%%%
\subsection{Boundary behavior in the half space -- Dirichlet}\label{subsection:halfspace}
%%%%%%%%%%%%%%%%%%%%%%%%%%%%%%%%%%%%%%%%%%%%%%%%%%%%%

We collect some particular one dimensional pointwise solutions that will be useful in our proofs.
Consider the problem 
$$\begin{cases}
(\partial_t - D_{xx}^+)^s u = f&\hbox{in}~\R\times\R_+ \\ 
u(t,0) =0&\hbox{in}~\R
\end{cases}$$
where $D^+_{xx}$ denotes the Dirichlet Laplacian in the half line $[0,\infty)$ and 
$$f(t,x)=\begin{cases}
1&\hbox{when}~0<s < 1/2 \\
\chi_{[0,1]}(x)&\hbox{when}~1/2 \leq s < 1.
\end{cases}$$
Since $f$ is independent of $t$, we have that $u$ is also independent of $t$ and solves
$$\begin{cases}
(- D_{xx}^+)^s u = f&\hbox{in}~\R_+ \\ 
u(0) =0.
\end{cases}$$
Then we have the following results (see also \cite{Stinga-Caffa}).

\smallskip

\noindent\textbf{Case 1: $0< s < 1/2$.}
There exists a constant $c_s>0$ such that
$$u(t,x) = c_s x^{2s}\qquad\hbox{for}~(t,x)\in\R\times \R^+.$$

\smallskip

\noindent\textbf{Case 2: $s=1/2$.}
We have
\begin{align*}
    u(t,x) &= c \int^1_0 \big( \log|x+z| - \log|x-z| \big) \,dz \\
    &= cx \int^{1/x}_0 \big(\log|1+\omega| - \log|1-\omega|\big) \,d \omega.
\end{align*}
For $0<x<1$, 
\begin{align*}
    u(t,x) &= cx \int^{1}_0 \big(\log(1+\omega)- \log(1-\omega)\big)\, d \omega  + cx \int^{1/x}_1 \big(\log(1+\omega) - \log(\omega-1)\big)\, d \omega\\
    &= c \big( (1+x)\log(1+x) -(1-x)\log(1-x) - 2x\log x\big).
\end{align*}
Hence, there exists $C>0$ such that, for any $0< x< 1$,
$$u(t,x) =-Cx\log x+\eta_1(x)$$
where $\eta_1(x)\sim x$ as $x\to0$. Therefore,
$$u(t,x) \sim -x \log x\qquad\hbox{as}~x\to0,~\hbox{uniformly in}~t\in\R.$$
On the other hand, if $x\geq 1$ then, 
\begin{align*}
    u(t,x)  &= cx \int^{1/x}_0 \big( \log(1+\omega) - \log(1-\omega)) \,d\omega \\
    &= cx \big[ (1/x+1) \log(1/x+1) + (1-1/x) \log(1-1/x) \big].
\end{align*}
Hence, for any $x\geq 1$, 
$$ u(t,x) =x \eta_2\Big(\frac{1}{x}\Big)$$
where 
$$ \eta_2(x) = c\big[ (1+x) \log(1+x) + (1-x) \log(1-x) \big].$$
To study the behavior of $u(t,x)$ near infinity we need to study the behavior of $\eta_2(x)$ near $0$. Using the series expansion for $\log(1\pm x)$ we see that 
$\eta_2(x) \sim x^2$ as $x \to 0$. Therefore, 
$$u(t,x) \sim \frac{1}{x}\qquad\hbox{as}~x \to \infty.$$

\smallskip

\noindent\textbf{Case 3: $1/2< s < 1$.}
We have 
\begin{align*}
    u(t,x) &= c \int^1_0 \Big( |x-z|^{2s-1} - |x+z|^{2s-1}\Big)\,dz\\
    &= cx^{2s} \int^{1/x}_0 \Big( |1-\omega|^{2s-1} - (1+ \omega)^{2s-1} \Big)\, d \omega.
\end{align*}
Let us consider $0<x<1$. Then 
\begin{align*}
    u(t,x) &= cx^{2s} \bigg[ \int^1_0 (1-\omega)^{2s-1} \,d \omega + \int^{1/x}_1 (\omega -1)^{2s-1}\,d\omega- \int^{1/x}_0 (1+\omega)^{2s-1}\, d \omega \bigg] \\
    &= c_s\big[ 2x^{2s} +(1-x)^{2s} -(1+x)^{2s}\big].
\end{align*}
On the other hand, if $x \geq 1$, then
\begin{align*}
    u(t,x) &= cx^{2s} \int^{1/x}_0 \Big( (1-\omega)^{2s-1} - (1+ \omega)^{2s-1} \Big)\, d \omega  \\
    &=c_sx^{2s}\big[ 2 - (1-1/x)^{2s} - (1+1/x)^{2s} \big].
\end{align*}
Whence, there exists $c_s>0$ such that
$$u(t,x)=
\begin{cases}
2c_sx^{2s}+\eta_{s1}(x)&\hbox{if}~ 0<x<1, \\
c_s x^{2s}\Big(2 - \eta_{s2}\Big( \frac{1}{x}\Big)\Big)&\hbox{if}~ x \geq 1,
\end{cases}
$$
where $\eta_{s1}$ and $\eta_{s2}$ are smooth up to $x=0$. Using the series expansions of $(1\pm x)^{2s}$, we get
\begin{equation}\label{eq:etas}
\eta_{s1}(x) \sim -4sx\quad \hbox{and}\quad\eta_{s2}(x) \sim 2+ 2s(2s-1)x^2 \qquad \hbox{as}~x\to 0.
\end{equation}
Using these estimates we conclude that
$$u(t,x) \sim x\qquad\hbox{as}~x\to0,~\hbox{uniformly in}~t\in\R,$$ 
and
$$u(t,x) \sim x^{2s-2} \qquad\hbox{as}~x\to\infty,~\hbox{uniformly in}~t\in\R.$$ 

\smallskip

Consider next the problem in a higher dimensional half space
\begin{equation}\label{eq:half_space}
\begin{cases}
(\partial_t - \Delta^+_D)^s w= g&\hbox{in}~\R\times \R^n_+ \\ 
w(t,x',0)=0&\hbox{on}~\R \times \partial \R^n_+
\end{cases}
\end{equation}
where
\begin{equation}\label{eq:g_func}
g(t,x)=\begin{cases}
1&\hbox{when}~0<s < 1/2 \\
\chi_{[0,1]}(x_n)&\hbox{when}~1/2 \leq s < 1.
\end{cases}
\end{equation}

The study of these solutions relies on the following observation.
Suppose that $g : \R^{n+1} \rightarrow \R$ is a function depending only on the $x_n$ variable, that is,
$g(t,x) = \phi(x_n)$ for some function $\phi : \R \rightarrow \R,$ for all $(t,x)\in \R^{n+1}$.
Let $w$ satisfy
$$(\partial_t - \Delta)^s w = g\qquad\hbox{in}~\R^{n+1}.$$
Then $w$ is a function that depends only on $x_n$. More precisely,
$w(t,x) = \psi(x_n)$ for all $(t,x)\in \R^{n+1}$, where $\psi : \R \to \R$ solves the one dimensional problem 
$$ (- D_{xx})^s \psi = \phi\qquad\hbox{in}~\R.$$ 
Indeed, that $w$ does not depend on $t$ is clear because $g$ does not depend on $t$.
Then $w$ will satisfy $(-\Delta)^sw=g$ and therefore the conclusion follows as in \cite{Stinga-Caffa}.

Thus, the pointwise solution $w(t,x)$ to \eqref{eq:half_space} with $g$ as in \eqref{eq:g_func} will be 
\begin{equation}\label{eq:diri_bdd}
w(t,x)=
\begin{cases}
c_s x^{2s}_n&\hbox{if}~0< s< 1/2, \\
-Cx_n \log x_n + \eta_1(x_n)&\hbox{for}~0<x_n <1,~\hbox{if}~s=1/2, \\
x_n \eta_2\Big(\frac{1}{x_n}\Big)&\hbox{for}~ x_n\geq1, ~\hbox{if}~s=1/2, \\
2c_s x^{2s}_n+\eta_{s1}(x_n)&\hbox{for}~0<x_n <1 ,~\hbox{if}~1/2<s<1 \\
c_s x^{2s}_n\Big(2 - \eta_{s2}\Big( \frac{1}{x_n}\Big)\Big)&\hbox{for}~ x_n\geq1 ,~\hbox{if}~1/2<s<1 
\end{cases}
\end{equation}
for some constants $c_s,C>0$.

Now, if we consider the following extension problem
\begin{equation}\label{eq:diri_dbb_exten}
\begin{cases}
y^a\partial_t W - \dvie(y^a \nabla W)=0&\hbox{in}~\R \times \R^n_+ \times (0, \infty) \\ 
-y^a W_y\big|_{y=0}=\theta g&\hbox{on}~\R \times \R^n_+ \\ 
W=0&\hbox{on}~\R \times \partial \R^n_+ \times [0,\infty)
\end{cases}
\end{equation}
with $g$ as in \eqref{eq:g_func} and $\theta \in \R$, then the pointwise solution $W(t,x,y)$ will satisfy
$$W(t,x,0) = \theta w(t,x)\qquad\hbox{for all}~(t,x)\in\R\times\R^n_+$$
where $w(t,x)$ is as in \eqref{eq:diri_bdd}. Though these solutions $W$ can be computed explicitly, we will only need
bounds for them and their derivatives in the $x_n$-direction (see the proof of the following Lemmas).

\begin{lem}\label{lem:W}
The solution $W(t,x,y)$ to \eqref{eq:diri_dbb_exten} satisfies the following estimates. 
\begin{enumerate}[$(1)$]
    \item If $s< 1/2$ then $|W(t,x,y)| \leq C|\theta|x^{2s}_n$ for all $(t,x,y)\in\R\times\R^n_+\times(0,\infty)$, where $C>0$ depends only on $s$.
    \item If $s\geq 1/2$ then $\|W\|_{L^\infty(\R\times\R^n_+\times(0,\infty))} \leq C|\theta|$, where $C>0$ depends only on $s$.
\end{enumerate}
\end{lem}

\begin{proof}
After dividing by $\theta$, we can assume that $\theta=1$.
Recall that the solution $W$ to \eqref{eq:diri_dbb_exten} is given by
\begin{align*}
W(t,x,y) &= \frac{y^{2s}}{4^s \Gamma(s)} \int^{\infty}_0 e^{-y^2/(4\tau)} e^{\tau \Delta} w_o(t-\tau,x)\, \frac{d\tau}{\tau^{1+s}}
\end{align*}
where $w_o$ denotes the odd reflection of $w$ with respect to the $x_n$ variable.

Consider first the case of $s<1/2$. Then $w(t,x) = c_s x^{2s}_n$ and
\begin{align*}
    e^{\tau \Delta} &w_o(t-\tau,x) \\
    & = \frac{C}{\tau^{1/2}}\bigg[\int^{x_n}_{-\infty} e^{-z^2_n /(4\tau)} (x_n-z_n)^{2s}\,dz_n
    - \int^{\infty}_{x_n} e^{-z^2_n / (4\tau)} (z_n-x_n)^{2s}\, dz_n\bigg] \\
    &= \frac{Cx^{2s+1}_n}{\tau^{1/2}}\bigg[\int^{1}_{-\infty} e^{-x^2_n \omega^2 / (4\tau)} (1-\omega)^{2s}\, d\omega -
     \int^{\infty}_{1} e^{-x^2_n \omega^2 / (4\tau)} (\omega-1)^{2s}\,d\omega\bigg] \\
         &= \frac{Cx^{2s+1}_n}{\tau^{1/2}}\bigg[\int_{-1}^{\infty} e^{-x^2_n \omega^2 / (4\tau)} (1+\omega)^{2s}\, d\omega -
     \int^{\infty}_{1} e^{-x^2_n \omega^2 / (4\tau)} (\omega-1)^{2s}\,d\omega\bigg] \\
   &\leq \frac{Cx^{2s+1}_n}{\tau^{1/2}} \bigg[\int^{2}_{-1} e^{-x^2_n\omega^2 /(4\tau)}\, d\omega
   +\int^{\infty}_{2} e^{-x^2_n\omega^2 /(4\tau)} [(1+\omega)^{2s}-(\omega-1)^{2s}]\, d\omega\bigg].
\end{align*}
The first integral above can be estimated by
\begin{equation}\label{eq:xn}
 \int^{\infty}_{-\infty}e^{-x^2_n\omega^2 /(4\tau)}\,d \omega = C \frac{\tau^{1/2}}{x_n}.
 \end{equation}
For the second integral we use the mean value theorem to estimate
$(1+\omega)^{2s} - (\omega -1)^{2s} \leq C$, whenever $2< \omega < \infty$. 
Therefore, by applying again \eqref{eq:xn}, we conclude that
$$e^{\tau \Delta} w_o(t-\tau,x)\leq C_sx_n^{2s}.$$
Hence, from the explicit formula for $W$ we conclude $(1)$.

For the case when $s\geq1/2$, notice that $w$ in \eqref{eq:diri_bdd} is bounded,
so that there exists $C_s>0$ such that $|e^{\tau\Delta_D^+}w(t-\tau,x)|\leq C_s$ for all $t\in\R$, $\tau>0$ and $x\in\R^n_+$.
Whence $(2)$ follows from the explicit formula for $W$.
\end{proof}

\begin{lem}\label{lem:nablaW}
The solution $W(t,x,y)$ to \eqref{eq:diri_dbb_exten} satisfies the following estimates, 
\begin{enumerate}[$(1)$]
    \item If $s<1/2$, then $|\partial_{x_n} W(t,x,y)| \leq C y^{2s-1}$ for all $(t,x,y)\in (Q^+_1)^* $, where $C>0$ depends only on $s$ and $\theta$.
    \item If $s= 1/2$ then $|\partial_{x_n} W(t,x,y)| \leq C |\log (x_n^2+y^2)|$ for all $(t,x,y)\in (Q^+_{1/2})^* $, where $C>0$ depends only on $s$ and $\theta$.
    \item If $s > 1/2$ then $|\partial_{x_n} W(t,x,y)| \leq C$
    for all $(t,x,y)\in (Q^+_1)^* $, where $C>0$ depends only on $s$ and $\theta$.
\end{enumerate}
\end{lem}

\begin{proof}
The solution $W$ to \eqref{eq:diri_dbb_exten} for $\theta=1$ is given by
\begin{equation}\label{eq:Wo}
\begin{aligned}
W(t,x,y) &= \frac{y^{2s}}{4^s \Gamma(s)} \int^{\infty}_0 e^{-y^2/(4\tau)} e^{\tau \Delta^+_D} w(t-\tau,x)\, \frac{d\tau}{\tau^{1+s}} \\
&= \frac{1}{\Gamma(s)} \int^{\infty}_0 e^{-y^2/(4 \tau)}e^{\tau \Delta_D^+} g(t-\tau,x)\,\frac{d\tau}{\tau^{1-s}}.
\end{aligned}
\end{equation}

Consider first the case of $s<1/2$. Using the second formula in \eqref{eq:Wo} and the fact that $g$ depends only on $x_n$, we get that
$$W(t,x,y) =C_s\int_0^\infty e^{-y^2/(4\tau)}\int_0^\infty\bigg(\frac{e^{-|x_n-z_n|^2/(4\tau)}}{\tau^{1/2}}
-\frac{e^{-|x_n+z_n|^2/(4\tau)}}{\tau^{1/2}}\bigg)g(z_n)\,dz_n\,\frac{d\tau}{\tau^{1-s}}.$$
We would like to apply Fubini's Theorem above. Since $g$ is bounded and $x_n,z_n>0$, we only need to check that
$$0\leq I:=\int_0^\infty e^{-y^2/(4\tau)}\int_0^\infty\bigg(\frac{e^{-|x_n-z_n|^2/(4\tau)}}{\tau^{1/2}}
-\frac{e^{-|x_n+z_n|^2/(4\tau)}}{\tau^{1/2}}\bigg)\,dz_n\,\frac{d\tau}{\tau^{1-s}}<\infty.$$
Indeed
\begin{align*}
I &=  \int^\infty_0 \frac{e^{-y^2/(4 \tau)}}{\tau^{1/2}} \bigg[ \int^\infty_0 e^{-(x_n-z_n)^2/(4 \tau)} dz_n - \int^0_{-\infty}e^{-(x_n-z_n)^2/(4 \tau)} dz_n \bigg] \frac{d \tau}{\tau^{1-s}}\\
    &=  \int^\infty_0 e^{-y^2/(4 \tau)}\bigg[\int^{x_n/(2 \sqrt{\tau})}_{-\infty} e^{-\omega^2}\, d \omega -  \int^\infty_{x_n/(2 \sqrt{\tau})}e^{-\omega^2}\, d \omega  \bigg]\frac{d \tau}{\tau^{1-s}}\\
    &=  \int^\infty_0 e^{-y^2/(4 \tau)}\operatorname{erf}(x_n/(2 \sqrt{\tau}))\frac{d \tau}{\tau^{1-s}}
\end{align*}
where we have denoted $\displaystyle\operatorname{erf}(r)=\int_{-r}^re^{-\omega^2}\,d\omega$.
One one hand, if $0<\tau<1$ then $\operatorname{erf}(x_n/(2 \sqrt{\tau})) < C$, so that
$$\int^1_0 e^{-y^2/(4 \tau)}\operatorname{erf}(x_n/(2 \sqrt{\tau}))\,\frac{d \tau}{\tau^{1-s}} \leq C\int^1_0 \,\frac{d \tau}{\tau^{1-s}}<\infty.$$
On the other hand, when $\tau$ is large, by using the Taylor expansion of $e^{-\omega^2}$, we can estimate
$\operatorname{erf}(x_n/(2 \sqrt{\tau})) \sim Cx_n/(2 \sqrt{\tau})$ so we have
$$\int_1^\infty e^{-y^2/(4\tau)}\operatorname{erf}(x_n/(2 \sqrt{\tau}))\,\frac{d\tau}{\tau^{1-s}}\leq
Cx_n\int^\infty_1 \tau^{s-1/2}\, \frac{d \tau}{\tau}<\infty.$$
Hence $I$ is convergent. Thus, for each fixed $(t,x,y)$, after Fubini's Theorem, 
\begin{align*}
     W(t,x,y) &=C_s \int_{0}^\infty g(z_n) \int^\infty_0\bigg(\frac{e^{-(y^2+|x_n-z_n|^2)/(4\tau)}}{\tau^{1/2-s}}-\frac{e^{-(y^2+|x_n+z_n|^2)/(4\tau)}}{\tau^{1/2-s}}\bigg)\,\frac{d\tau}{\tau} \, dz_n \\
     &= C_s \int^\infty_0\bigg( \frac{1}{(y^2+(x_n-z_n)^2)^{(1-2s)/2}}- \frac{1}{(y^2+(x_n+z_n)^2)^{(1-2s)/2}}\bigg)dz_n.
\end{align*}
Since $s<1/2$, it is easy to check that we can differentiate inside the integral to finally obtain
$$\partial_{x_n}W(t,x,y) = \frac{C_s}{(x^2_n+y^2)^{(1-2s)/2}} $$
from which the estimate in $(1)$ follows.

For $s=1/2$, we use the second formula in \eqref{eq:Wo} and a similar computation as in \cite{users} to find that,
since $g_o$ is independent of $t$ and has zero mean,
\begin{align*}
    W(t,x,y) &= \frac{1}{2\pi}\int^{\infty}_0 \int^{\infty}_{-\infty}e^{-(y^2+z^2_n)/(4\tau)} g_o(t-\tau,x_n-z_n) \, dz_n \,\frac{d\tau}{\tau} \\
    &=\frac{1}{2\pi}\int^{\infty}_0 \int^{\infty}_{-\infty}e^{-r(y^2+z^2_n)} g_o(x_n-z_n) \, dz_n\, \frac{dr}{r}\\
    &=\frac{1}{2\pi}\int^{\infty}_0 \int^{\infty}_{-\infty}\big(e^{-r(y^2+z^2_n)}-\chi_{(0,1)}(r) \big) g_o(x_n-z_n) \, dz_n \,\frac{dr}{r} \\
    &=-\frac{1}{2\pi}\int^\infty_{-\infty} \log\big((x_n-z_n)^2+y^2\big)g_o(z_n)\, dz_n \\
    &=\frac{1}{2\pi}\int^\infty_{0}\Big[\log\big((x_n+z_n)^2+y^2\big)-\log\big((x_n-z_n)^2+y^2\big)\Big]g(z_n)\, dz_n .
\end{align*}
Next, since $g(z_n)=\chi_{[0,1]}(z_n)$, by using integration by parts,
\begin{align*}
    W(t,x,y)&= (1+x_n) \log((1+x_n)^2+y^2) - (1-x_n) \log((1-x_n)^2+y^2)  - 2x_n \log(x_n^2+y^2) \\
    &\quad + 2y \arctan((1+x_n)/y)-2y \arctan((1-x_n)/y) - 4y \arctan(x_n/y)
\end{align*}
Therefore,
$$\partial_{x_n}W(t,x,y)=\log((1+x_n)^2+y^2)+ \log((1-x_n)^2+y^2)- 2\log(x^2_n + y^2)$$
from which $(2)$ follows.

To prove $(3)$ for $s>1/2$, we notice that 
$$w(t,x)=
\begin{cases}
 c_s\big[ 2x^{2s}_n +(1-x_n)^{2s} -(1+x_n)^{2s}\big] \quad \hbox{for}~ 0<x_n<1, \\
c_sx_n^{2s}\big[ 2 - (1-1/x_n)^{2s} - (1+1/x_n)^{2s} \big] \quad \hbox{for}~ x_n\geq 1
\end{cases}
$$
Then, for  $0<x_n<1$,
$$\partial_{x_n}w(t,x)=c_s\big[ 2x^{2s-1}_n -(1-x_n)^{2s-1} -(1+x_n)^{2s-1}\big]$$
and, for $x_n\geq1$,
\begin{multline*}
\partial_{x_n}w(t,x)=
c_sx_n^{2s-1}\big[ 2 - (1-1/x_n)^{2s} - (1+1/x_n)^{2s} \big] \\
+ c_sx_n^{2s-2}\big[ (1-1/x_n)^{2s-1} + (1+1/x_n)^{2s-1}\big].
\end{multline*}
Now using the estimate for $\eta_{s2}(1/x_n)$ in \eqref{eq:etas},
we conclude that $\partial_{x_n}w \sim C$ as $x_n\to 0$, and $\partial_{x_n} w \sim x^{2s-2}_n$ as $x_n \to \infty$. Then we see that $|\partial_{x_n}w|$ is bounded everywhere. From here and the first formula in \eqref{eq:Wo}
is it easy to check that $|\partial_{x_n}(e^{\tau\Delta_D^+}w(t-\tau,x)|\leq C$ for all $\tau>0$ and $(t,x)\in\R\times\R^n_+$, which in turn 
establishes $(3)$.
\end{proof}

%%%%%%%%%%%%%%%%%%%%%%%%%%%%%%%%%%%%%%%%%%%%%%%%%%%%%
\subsection{Boundary regularity in the half space -- Neumann}
%%%%%%%%%%%%%%%%%%%%%%%%%%%%%%%%%%%%%%%%%%%%%%%%%%%%%

In the half space $\R \times \R^n_+$ we consider the heat operator $\partial_t - \Delta^+_N$,
where $\Delta_N^+$ is the Neumann Laplacian in $\R^n_+$.
For a function $u(t,x)$ defined on $\R \times \overline{\R^n_+}$ with $u_{x_n}(t,x',0) = 0$ and $0<s<1$ we define
\begin{align*}
(\partial_t-\Delta^+_N)^su(t,x) &= \frac{1}{\Gamma(-s)} \int^{\infty}_0 \big(e^{\tau \Delta^+_N}u(t-\tau,x) - u(t,x) \big) \frac{d \tau}{\tau^{1+s}}
\end{align*}
where $\{e^{\tau\Delta_N^+}\}_{\tau\geq0}$ is the semigroup generated by $\Delta_N^+$.
As before, let $x^* = (x',-x_n)$ for $x \in \R^n$. Denote by
$u_e(t,x)$ the even extension of $u(t,x)$ about the $x_n$ axis given by
$$u_e(t,x)=
\begin{cases}
u(t,x)&\hbox{if}~x_n \geq 0 \\
u(t,x^*)=u(t,x',-x_n)&\hbox{if}~x_n <0.
\end{cases}$$
Now
\begin{align*}
e^{\tau \Delta^+_N} u(t-\tau,x) &= e^{\tau \Delta} u_e(t-\tau,x) \\
&= \frac{1}{(4 \pi \tau)^{n/2}} \int_{\R^n_+}\left( e^{-|x-z|^2/(4 \tau)}+e^{-|x-z^*|^2/(4 \tau)} \right)u(t-\tau,z)\, dz
\end{align*}
for any $\tau>0$, $x \in \R^n_+$. Hence, for $x\in \R^n_+,$
\begin{multline*}
(\partial_t-\Delta^+_N)^su(t,x)= \\
\frac{1}{(4 \pi)^{n/2} \Gamma(-s)}\int^{\infty}_0 \int_{\R^n_+} \left(\frac{e^{-|x-z|^2/(4 \tau)}+
e^{-|x-z^*|^2/(4 \tau)}}{\tau^{n/2+1+s}}\right) (u(t-\tau,z) - u(t,x))\, dz\,d\tau
\end{multline*}
and
\begin{multline*}
(\partial_t - \Delta^+_N)^{-s} f(t,x) = \frac{1}{\Gamma(s)} \int^{\infty}_0 e^{-\tau(\partial_t - \Delta^+_N) }f(t,x) \frac{d \tau }{\tau^{1-s}} \\
=  \frac{1}{(4\pi)^{n/2}\Gamma(s)} \int^{\infty}_0\int_{\R^n_+}\left(\frac{e^{-|x-z|^2/(4 \tau)}+e^{-|x-z^*|^2/(4 \tau)}}{\tau^{n/2+1-s}} \right)f(t-\tau,z) \,dz\,d\tau.
\end{multline*}

\begin{thm}[Boundary regularity in half space -- Neumann]\label{thm:half_neumann}
Let $u,f\in L^\infty(\R\times\R^n_+)$ be such that $\int_{\R^n_+} f(t,x)\,dx=0$ for all $t\in\R$ and
$$\begin{cases}
(\partial_t-\Delta_N^+)^su=f&\hbox{in}~\R\times\R^n_+\\
-u_{x_n}=0&\hbox{on}~\R\times\partial\R^n_+.
\end{cases}$$
\begin{enumerate}[$(1)$]
    \item Suppose that $f \in C^{\alpha/2, \alpha}(\R \times \overline{\R^n_+})$ for $0< \alpha \leq 1$.
    \begin{enumerate}[$(a)$]
        \item If $\alpha + 2s$ is not an integer then $u \in C^{\alpha/2 + s, \alpha+2s}(\R \times \overline{\R^n_+})$ with the estimate 
         $$\norm{u}_{C^{\alpha/2 + s, \alpha+2s}(\R \times \overline{\R^n_+})} \leq
         C \big( \norm{f}_{C^{\alpha/2, \alpha}(\R \times \overline{\R^n_+})} + \norm{u}_{L^{\infty}(\R \times \overline{\R^n_+})} \big).$$
        \item If $\alpha + 2s=1$ then $u(t,x)$ is in the H\"older--Zygmund space $\Lambda^{1/2,1}(\R \times \overline{\R^n_+})$
        with the estimate 
        $$\norm{u}_{\Lambda^{1/2,1}(\R \times \overline{\R^n_+})}
        \leq C \big( \norm{f}_{C^{\alpha/2, \alpha}(\R \times \overline{\R^n_+})} + \norm{u}_{L^{\infty}(\R \times \overline{\R^n_+})} \big).$$
    \end{enumerate}
    The constants $C>0$ depend only on $n$, $s$ and $\alpha$.
    \item Let $f \in L^{\infty}(\R \times \overline{\R^n_+})$.
    \begin{enumerate}[$(a)$]
        \item If $s\neq1/2$ then $u \in C^{s,2s}(\R \times \overline{\R^n_+})$ with the estimate
        $$\norm{u}_{C^{s, 2s}(\R \times \overline{\R^n_+})}
        \leq C \big( \norm{f}_{L^{\infty}(\R \times \overline{\R^n_+})} + \norm{u}_{L^{\infty}(\R \times \overline{\R^n_+})} \big).$$ 
        \item If $s = 1/2$ then $u$ is in the H\"older--Zygmund space $\Lambda^{1/2, 1}(\R \times \overline{\R^n_+})$ with the estimate 
        $$\norm{u}_{\Lambda^{1/2,1}(\R \times \overline{\R^n_+})} \leq
        C \big( \norm{f}_{L^{\infty}(\R \times \overline{\R^n_+})} + \norm{u}_{L^{\infty}(\R \times \overline{\R^n_+})} \big).$$
    \end{enumerate} 
    The constants $C>0$ above depend only on $n$ and $s$.
\end{enumerate}
\end{thm}

\begin{proof}
We show this result by noticing that if $f_e$ and $u_e$ are the even reflections of $f$ and $u$ with respect to the variable $x_n$, respectively,
then $(\partial_t - \Delta)^s u_e=f_e$ in $\R^{n+1}$, so that Proposition \ref{prop:global_bdd_diri} applies.
From the pointwise formula we see that $(\partial_t-\Delta)^su_e(t,x) = (\partial_t-\Delta^+_N)^s u(t,x)=f(t,x)=f_e(t,x)$ for $x \in \R^n_+$.
Now, for $x \in \R^n$ is such that $x_n < 0$,
\begin{align*}
    (\partial_t -\Delta)^su_e(t,x) &= \frac{1}{\Gamma(-s)} \int^{\infty}_0 \left( e^{\tau \Delta}u_e(t-\tau,x) - u_e(t,x) \right)\frac{d \tau}{\tau^{1+s}} \\
    &= \frac{1}{\Gamma(-s)} \int^{\infty}_0 \left( e^{\tau \Delta^+_N}u(t-\tau,x^*) - u(t,x^*) \right)\frac{d \tau}{\tau^{1+s}} \\
    &= (\partial_t-\Delta_N^+)^su(t,x^*)= f(t,x^*) = f_e(t,x).
\end{align*}
\end{proof}

%%%%%%%%%%%%%%%%%%%%%%%%%%%%%%%%%%%%%%%%%%%%%%%%%%%%%
\section{Global regularity}\label{sec:regularity_global}
%%%%%%%%%%%%%%%%%%%%%%%%%%%%%%%%%%%%%%%%%%%%%%%%%%%%%

In this Section we present the proofs of Theorems \ref{thm:1.3_f_nonzero}, \ref{thm:1.3_f_zero}, \ref{thm:boundaryNeumannHolder} and \ref{thm:1.5}.

We assume that $\Omega \subset \R^n_+$
is a bounded domain such that its boundary contains a flat portion on $\{x_n = 0 \}$ in such a way that $B^+_1 \subset \Omega$. 

We say that $f\in L^2(Q_1^+)$ is in $L^{\alpha/2, \alpha}_+(0,0)$, $0<\alpha\leq1$, if
 \begin{equation}\nonumber
    [f]^2_{L^{2,\alpha/2,\alpha}_+(0,0)} := \sup_{0 < r \leq 1} \frac{1}{r^{n+2+2 \alpha}} \int_{Q^+_r} |f-f(0,0)|^2\,dt\, dx< \infty 
 \end{equation}
 where $\displaystyle f(0,0)= \lim_{r \rightarrow 0} \frac{1}{|Q_r^+|} \int_{Q_r^+} f(t,x) \, dt \, dx $.
 
 Theorem \ref{thm:1.3_f_nonzero} follows from the
 next result after flattening the boundary, translation and rescaling, and by taking into account estimate \eqref{eq:L2vsHs}
 and the properties of half space solutions, see subsection \ref{subsection:halfspace},
 and Theorem \ref{thm:Campanato}. 
 
\begin{thm}\label{thm:bdd_holder_1_f_non-zero}
Let $u\in\Dom(H^s)$ be a solution to \eqref{eq:Hsuf} with Dirichlet boundary condition
and assume that $f\in L^{\alpha/2,\alpha}_+(0,0)$, for some $0< \alpha<1$.
Let $w$ be the half space solution to \eqref{eq:half_space}.
\begin{enumerate}[$(1)$]
    \item Assume that $ 0 < \alpha + 2s <1$. There exist $0< \delta < 1$, depending only on $n$, ellipticity, $\alpha$ and $s$, and a constant $C_1 >0$ such that if 
$$\sup_{0 < r \leq 1} \frac{1}{r^{n+2\alpha}} \int_{B^+_r} |A(x) - A(0)|^2\,dx < \delta ^2$$
then  
$$ \frac{1}{r^{n+2}}\int_{Q^+_r} |u(t,x)-f(0,0)w(t,x)|^2 \, dt \, dx \leq C_1 r^{2(\alpha + 2s)} $$
for all $r>0$ small. Moreover, 
$$C^{1/2}_1 \leq C_0\big(1+\|u\|_{\Dom(H^s)} + |f(0,0)| + [f]_{L^{\alpha/2,\alpha}_+(0,0)} \big) $$
where $C_0>0$ depends on $A(x)$, $n$, $s$, $\alpha$ and ellipticity.
\item Assume that $s = 1/2$ and $1<\alpha + 2s <2$. Let $0<\varepsilon<1/2$ such that
$0<\alpha+\varepsilon<1$. There exists $0<\delta <1$, depending only on $n$, ellipticity, $\alpha$ and $s$, and a constant $C_1>0$ such that if 
$$\sup_{0 < r \leq 1/2} \frac{1}{r^{n+2(\alpha+\varepsilon)}} \int_{B^+_r}|A(x) - A(0)|^2 \,dx < \delta ^2$$
then there exists a linear function $l(x) =  \mathcal{B} \cdot x$ such that 
$$ \frac{1}{r^{n+2}}\int_{Q^+_r} |u(t,x)-f(0,0) w(t,x)-l(x)|^2 \, dt \, dx \leq C_1 r^{2(\alpha + 1)}$$
for all $r>0$ small. Moreover, 
$$C^{1/2}_1 + |\mathcal{B}| \leq C_0\big(1+\|u\|_{\Dom(H^s)} + |f(0,0)| + [f]_{L^{\alpha/2,\alpha}_+(0,0)} \big)$$
where $C_0>0$ depends on $A(x)$, $n$, $s$, $\alpha$ and ellipticity.
\item Assume that $s > 1/2$ and $1< \alpha + 2s<2$. There exists $0< \delta < 1,$ depending only on $n$, ellipticity, $\alpha$ and $s$, and a constant $C_1 >0$ such that if 
$$\sup_{0 < r \leq 1} \frac{1}{r^{n+2(\alpha + 2s-1)}} \int_{B^+_r} |A(x) - A(0)|^2 \,dx < \delta ^2$$
 then there exists a linear function $l(x) =  \mathcal{B} \cdot x$ such that 
$$ \frac{1}{r^{n+2}}\int_{Q^+_r} |u(t,x)-f(0,0) w(t,x)-l(x)|^2 \, dt \, dx \leq C_1 r^{2(\alpha + 2s)}$$
for all $r>0$ small. Moreover, 
$$C^{1/2}_1 + |\mathcal{B}| \leq C_0\big(1+\|u\|_{\Dom(H^s)} + |f(0,0)| + [f]_{L^{\alpha/2,\alpha}_+(0,0)} \big)$$
where $C_0>0$ depends on $A(x)$, $n$, $s$, $\alpha$ and ellipticity.
\end{enumerate}
\end{thm}

Similarly, Theorem \ref{thm:1.3_f_zero} is a direct consequence of the following result.

\begin{thm}\label{thm:bdd_holder_1_f_zero}
Let $u\in\Dom(H^s)$ be a solution to \eqref{eq:Hsuf} with Dirichlet boundary condition
and assume that $f\in L^{\alpha/2,\alpha}_+(0,0)$, for some $0< \alpha<1$, and that $f(t,0)=0$ for all $t\in [-1,1]$.
\begin{enumerate}[$(1)$]
    \item Assume that $ 0 < \alpha + 2s <1$. There exist $0< \delta < 1$, depending only on $n$, ellipticity, $\alpha$ and $s$, and a constant $C_1 >0$ such that if 
$$\sup_{0 < r \leq 1} \frac{1}{r^n} \int_{B^+_r} |A(x) - A(0)|^2\,dx < \delta ^2$$
then  
$$ \frac{1}{r^{n+2}}\int_{Q^+_r} |u(t,x)|^2 \, dt \, dx \leq C_1 r^{2(\alpha + 2s)} $$
for all $r>0$ small. Moreover, 
$$C^{1/2}_1 \leq C_0\big(\|u\|_{\Dom(H^s)}+ [f]_{L^{\alpha/2,\alpha}_+(0,0)} \big) $$
where $C_0>0$ depends on $A(x)$, $n$, $s$, $\alpha$ and ellipticity.
\item Assume that $1< \alpha + 2s<2$. There exists $0< \delta < 1,$ depending only on $n$, ellipticity, $\alpha$ and $s$, and a constant $C_1 >0$ such that if 
$$\sup_{0 < r \leq 1} \frac{1}{r^{n+2(\alpha + 2s-1)}} \int_{B^+_r} |A(x) - A(0)|^2 \,dx < \delta ^2$$
 then there exists a linear function $l(x) =  \mathcal{B} \cdot x$ such that 
$$ \frac{1}{r^{n+2}}\int_{Q^+_r} |u(t,x)-l(x)|^2 \, dt \, dx \leq C_1 r^{2(\alpha + 2s)}$$
for all $r>0$ small. Moreover, 
$$C^{1/2}_1 + |\mathcal{B}| \leq C_0\big(\|u\|_{\Dom(H^s)} +  [f]_{L^{\alpha/2,\alpha}_+(0,0)} \big)$$
where $C_0>0$ depends on $A(x)$, $n$, $s$, $\alpha$ and ellipticity.
\end{enumerate}
\end{thm}

Theorem \ref{thm:boundaryNeumannHolder} follows from the next statement.

\begin{thm}\label{thm:bdd_holder_neu}
Let $u$  be a solution to \eqref{eq:Hsuf} with Neumann boundary condition.
Assume that $f\in L^{\alpha/2,\alpha}_+(0,0)$ for some $0< \alpha<1$.
\begin{enumerate}[$(1)$]
    \item Assume that $ 0 < \alpha + 2s <1$. There exist $0< \delta < 1$, depending only on $n$, ellipticity, $\alpha$ and $s$, and a constant $C_1 >0$ such that if 
$$\sup_{0 < r \leq 1} \frac{1}{r^n} \int_{B^+_r} |A(x) - A(0)|^2 dx\, < \delta ^2 $$
then  there exists a constant $c$ such that 
$$ \frac{1}{r^{n+2}}\int_{Q^+_r} |u(t,x)-c|^2 \, dt \, dx \leq C_1 r^{2(\alpha + 2s)}$$
for all $r>0$ small. Moreover, 
$$C^{1/2}_1+|c|\leq C_0\big( \norm{u}_{\Dom(H^s)} + |f(0,0)| + [f]_{L^{\alpha/2,\alpha}_+(0,0)} \big)$$
where $C_0>0$ depends on $A(x)$, $n$, $s$, $\alpha$ and ellipticity.
\item Assume that $1< \alpha + 2s<2$. There exists $0< \delta < 1,$ depending only on $n$, ellipticity, $\alpha$ and $s$, and a constant $C_1 >0$ such that if 
$$\sup_{0 < r \leq 1} \frac{1}{r^{n+2(\alpha + 2s-1)}} \int_{B^+_r} |A(x) - A(0)|^2\, dx < \delta ^2 $$
then there exists a linear function $l(x) = \mathcal{A} +  \mathcal{B} \cdot x$ such that 
$$ \frac{1}{r^{n+2}}\int_{Q^+_r} |u(t,x)-l(x)|^2 \, dt \, dx \leq C_1 r^{2(\alpha + 2s)}$$
for all $r>0$ small. Moreover, 
$$C^{1/2}_1 + |\mathcal{A}| + |\mathcal{B}| \leq C_0\big( \norm{u}_{\Dom(H^s)} + |f(0,0)| + [f]_{L^{\alpha/2,\alpha}_+(0,0)} \big)$$
\end{enumerate}
where $C_0>0$ depends on $A(x)$, $n$, $s$, $\alpha$ and ellipticity.
\end{thm}

We say that a function $f\in L^2(Q_1^+)$ is in $L^{-s+\alpha/2, -2s+\alpha}_+(0,0)$, for $0 < \alpha < 1$, whenever 
$$[f]^2_{L^{-s+\alpha/2, -2s+\alpha}_+(0,0)} = \sup_{0 < r \leq 1} \frac{1}{r^{n + 2 + 2(-2s+\alpha)}} \int_{Q_r^+} |f(t,x)|^2 \,dt \, dx < \infty $$
and that is in $L^{-s+(1+\alpha)/2, -2s+\alpha+1}_+(0,0)$ whenever 
$$[f]^2_{L^{-s+(1+\alpha)/2, -2s+\alpha+1}_+(0,0)} = \sup_{0 < r \leq 1} \frac{1}{r^{n + 2 + 2(-2s+\alpha+1)}} \int_{Q_r^+} |f(t,x)|^2 \,dt \, dx < \infty .$$

By H\"older's inequality (see the remarks before Theorem \ref{thm:thm_2_holder}),
it is clear that Theorem \ref{thm:1.5} will follow from the next result.

\begin{thm}\label{thm:bdd_holder_2}
Let $u\in\Dom(H^s)$ be a solution to \eqref{eq:Hsuf} with either Dirichlet or Neumann boundary condition and let $0<\alpha<1$.
\begin{enumerate}[$(1)$]
 \item Assume that $f \in L^{ -s+\alpha/2, -2s+\alpha}_+(0,0)$. There exist  $0< \delta < 1$, depending only on $n,$ ellipticity, $\alpha$, $s$ and a constant $C_1>0$ such that if 
    $$\sup_{0 < r \leq 1} \frac{1}{r^n} \int_{B^+_r} |A(x) - A(0)|^2\, dx < \delta ^2$$
    then there exists a constant $c$ such that
    $$ \frac{1}{r^{n+2}}\int_{Q^+_r} |u(t,x)-c|^2 \, dt \, dx \leq C_1 r^{2 \alpha} $$
    for all $r>0$ small.  Moreover, 
$$C^{1/2}_1 \leq C_0\big( \|u\|_{\Dom(H^s)} + [f]_{L^{ -s+\alpha/2, -2s+\alpha}_+(0,0)} \big)$$
   where $C_0>0$ depends on $A(x)$, $n$, $s$, $\alpha$ and ellipticity.
    \item Assume that $f \in L^{ -s+(1+\alpha)/2, -2s+\alpha+1}_+(0,0)$.
    There exist  $0< \delta < 1$, depending only on $n,$ ellipticity, $\alpha$, $s$, and a constant $C_1>0$ such that if 
    $$\sup_{0 < r \leq 1} \frac{1}{r^{n+2 \alpha}} \int_{B^+_r} |A(x) - A(0)|^2\, dx < \delta ^2 $$
    then there exists a linear function $l(x) = \mathcal{A}+\mathcal{B} \cdot x$ such that
    $$ \frac{1}{r^{n+2}}\int_{Q^+_r} |u(t,x)-l(x)|^2 \, dt \, dx \leq C_1 r^{2(1+\alpha)} $$
    for all $r>0$ small. Moreover, 
$$C^{1/2}_1 +|\mathcal{A}| + |\mathcal{B}| \leq C_0\big( \|u\|_{\Dom(H^s)} + [f]_{L^{ -s+(1+\alpha)/2, -2s+\alpha+1}_+(0,0)} \big) $$
where $C_0>0$ depends on $A(x)$, $n$, $s$, $\alpha$ and ellipticity.    
\end{enumerate}
In particular, for the case of Dirichlet boundary condition, $c=0$ and $\mathcal{A}=0$ above.
\end{thm}

The rest of the Section is devoted to the proofs of Theorems \ref{thm:bdd_holder_1_f_non-zero}, \ref{thm:bdd_holder_1_f_zero},
\ref{thm:bdd_holder_neu} and \ref{thm:bdd_holder_2}.

%%%%%%%%%%%%%%%%%%%%%%%%%%%%%%%%%%%%%%%%%%%%%%%%%%%%%
\subsection{Proof of Theorem \ref{thm:bdd_holder_1_f_non-zero}(1)}
%%%%%%%%%%%%%%%%%%%%%%%%%%%%%%%%%%%%%%%%%%%%%%%%%%%%%

Let $U$ be the solution to the extension problem for $u$, so that $U$ is a weak solution to
$$\begin{cases}
y^a \partial_t U - \dvie(y^a B(x) \nabla U) =0&\hbox{in}~\StarQQPlus \\
-y^a U_y\big|_{y=0} = f &\hbox{on}~Q^+_1 \\
U=0&\hbox{on}~Q_1 \cap \{x_n = 0 \}.
\end{cases}$$
Without loss of generality, we can assume that $B(0) = I$. 
We need to compare $U$ with the solution $W$
to the extension problem for the half space solution $w$. Let $W$ solve \eqref{eq:diri_dbb_exten} with $\theta=f(0,0)$,
so that it is a weak solution to
$$\begin{cases}
y^a \partial_t W - \dvie(y^a\nabla W) =0&\hbox{in}~\StarQQPlus \\
-y^a W_y\big|_{y=0} = f(0,0) &\hbox{on}~Q^+_1 \\
W=0&\hbox{on}~Q_1 \cap \{x_n = 0 \}.
\end{cases}$$
Let $V = U-W$. Then $V$ is a weak solution to 
\begin{equation}\label{eq:bdd_exten_mod}
\begin{cases}
y^a \partial_t V - \dvie(y^a B(x) \nabla V) = - \dvie(y^a F)&\hbox{in}~\StarQQPlus \\
-y^a V_y\big|_{y=0} = h &\hbox{on}~Q^+_1 \\
V=0&\hbox{on}~Q_1 \cap \{x_n = 0 \}.
\end{cases}
\end{equation}
where
$$F = (I-B(x)) \nabla W,~F_{n+1}=0\qquad\hbox{and}\qquad h=f- f(0,0),~h(0,0) = 0.$$
We observe that $F$ satisfies a certain Morrey-type integrability condition. Indeed, when $s<1/2$, by Lemma \ref{lem:nablaW},
\begin{align*}
    [F]_{\alpha, s}^2&: = \sup_{0 < r \leq 1} \frac{1}{r^{n+3+a  + 2(\alpha + 2s-1)}} \int_{(Q^+_r)^*} y^a |F|^2 \, dt \, dX \\
    &=\sup_{0 < r \leq 1} \frac{1}{r^{n+3+a  + 2(\alpha + 2s-1)}} \int_{(Q^+_r)^*}y^a|(I-B(x)) \nabla W|^2 \, dt \, dX \\
    &\leq \sup_{0 < r \leq 1} \frac{1}{r^{n  + 3+ a + 2(\alpha + 2s-1)}} \int_{B^+_r} \int^{r^2}_{-r^2} \int^r_0 y^a|(I-A(x))|^2y^{4s-2} \, dy \, dt\, dx  \\
    &= \sup_{0 < r \leq 1} \frac{C_s}{r^{n  + 2 + 2\alpha}} \int_{B^+_r} \int^{r^2}_{-r^2} |(I-A(x))|^2 \, dt\, dx \\
     &= \sup_{0 < r \leq 1} \frac{C_s}{r^{n  + 2\alpha}} \int_{B^+_r} |(I-A(x))|^2 \, dx<C_s\delta^2
\end{align*}

We say that, given $\delta>0$, $V$ is a $\delta$-normalized solution to \eqref{eq:bdd_exten_mod} if the following conditions hold:
\begin{enumerate}
    \item $\displaystyle \sup_{0 \leq r \leq 1}\frac{1}{r^{n+2\alpha}} \int_{B_r^+} |A(x)-I|^2 dx < \delta^2 $;
    \item $[h]^2_{L^{\alpha/2,\alpha}_+(0,0)} := \displaystyle \sup_{0 < r \leq 1} \frac{1}{r^{n+2+2 \alpha}} \int_{Q^+_r} |h|^2\,dt\, dx < \delta^2$;
    \item $[F]_{\alpha, s} ^2=\displaystyle \sup_{0 < r \leq 1} \frac{1}{r^{n+3+a+2(
    \alpha+2s-1)}} \int_{(Q_r^+)^*} y^a |F|^2 \,dt\,dX < \delta^2$;
    \item $\displaystyle \int_{Q_1^+} V(t,x,0)^2 \,dt\,dx + \int_{(Q_1^+)^*} y^a V^2\,dt\,dX \leq 1$.
\end{enumerate}
By scaling and by considering 
$$V(t,x,y) \bigg[\bigg(\displaystyle \int_{Q_1^+} V(t,x,0)^2 \,dt\,dx + \int_{(Q_1^+)^*} y^a V^2\,dt\,dX \bigg)^{1/2} + \frac{1}{\delta}([F]_{\alpha,s} + [h]_{L^{\alpha/2,\alpha}_+(0,0)} ) \bigg]^{-1}$$
we can always assume that $V$ is a $\delta$-normalized solution.

Now we follow similar steps as in the proof of Lemma \ref{lem:induc_step_1} with necessary changes. Namely, we replace balls by half-balls
and use Corollaries \ref{cor:boundaryapprox} and \ref{cor:bdd_reg} and Lemma \ref{lem:trace_new}.
There is another change in the computation we need to consider because, unlike the proof of Theorem \ref{thm:thm_1_holder},
here we have $F \neq 0$. Indeed, we perform the following estimate:
\begin{align*}
&\lambda^2\int_{(Q_{\lambda}^+)^*} y^a |\nabla V|^2\, dt \, dX \\ 
&\leq C \lambda^2\Bigg( \int_{(Q^+_{2 \lambda})^*}  y^a \frac{1}{\lambda^2} |V-c|^2  \, dt \, dX + \norm{F}_{L^2((Q^+_{2 \lambda})^*)} + \int_{Q^+_{2\lambda}} |V(t,x,0)-c||h(t,x)|\, dt \,dx  \Bigg) \\ 
&\leq  C \int_{(Q^+_{2 \lambda})^*} y^a|V-c|^2 dt \, dX + C\delta^2 + C \left( \norm{V(\cdot,\cdot,0)}_{L^2(Q^+_{2\lambda})}  + |c| |Q^+_{2 \lambda}|^{1/2} \right)\norm{h}_{L^2(Q^+_{2 \lambda})} \\
&\leq2C \varepsilon^2 + Cc_{n,a} \lambda^{n+5+a}+ C(1+ |c|)\delta.
\end{align*}
Therefore, we obtain the existence of $0< \delta, \lambda <1$ such that if $V$ is a $\delta$-normalized solution then
$$\frac{1}{\lambda^{n+2}} \int_{Q^+_{\lambda}} |V(t,x,0)|^2 dt \, dx + \frac{1}{\lambda^{n+3+a}}  \int_{(Q^+_{\lambda})^*} |V|^2 dt \, dX < \lambda^{2(\alpha + 2s)}$$
Notice that here $c= V(0,0,0)= 0$. Using the above result, if we follow similar steps as in the proof of
Lemma \ref{lem:induction}, with similar necessary changes as above, and setting $c_k= 0$ for all induction step $k$, and we can prove that 
$$ \frac{1}{r^{n+2}} \int_{Q^+_r} |V(t,x,0)|^2 dt \, dx < C_1 r^{2(\alpha + 2s)}$$
for all $r >0$ sufficiently small. 
The constant $C_1$ satisfies the following bound
\begin{align*}
   C_1 \leq C^2_0 \bigg( &\int_{Q_1^+} U(t,x,0)^2 \,dt\,dx + \int_{(Q_1^+)^*} y^a U^2\,dt\,dX + \int_{Q_1^+} W(t,x,0)^2 \,dt\,dx\\
   &\quad  + \int_{(Q_1^+)^*} y^a W^2\,dt\,dX+ \frac{1}{\delta^2}[F]^2_{\alpha,s} + \frac{1}{\delta^2}[f]^2_{L^{\alpha/2,\alpha}_+(0,0)}\bigg).
\end{align*}
Notice that, from Lemma \ref{lem:W},
\begin{multline*}
    \int_{Q_1^+} W(t,x,0)^2 \,dt\,dx + \int_{(Q_1^+)^*} y^a W^2\,dt\,dX \\
    \leq C|f(0,0)|^2 \int_{Q_1^+} x^{4s}_n \,dt\,dx +  C|f(0,0)|^2  \int_{(Q_1^+)^*} y^a x^{4s}_n \, dt \, dX= C|f(0,0)|^2,
\end{multline*}
so we conclude that the estimate for $C_1$ in the statement holds.

%%%%%%%%%%%%%%%%%%%%%%%%%%%%%%%%%%%%%%%%%%%%%%%%%%%%%
\subsection{Proof of Theorem \ref{thm:bdd_holder_1_f_non-zero}(2) }
%%%%%%%%%%%%%%%%%%%%%%%%%%%%%%%%%%%%%%%%%%%%%%%%%%%%%

Let $U$, $V$, $F$ and $h$ be as in the proof of Theorem \ref{thm:bdd_holder_1_f_non-zero}(1). 
Observe that, by Lemma \ref{lem:nablaW}, $F$ now satisfies the following Campanato-type integrability condition:
\begin{align*}
    [F]^2_{\alpha, 1/2} &: = \sup_{0 < r \leq 1/2} \frac{1}{r^{n+3  + 2\alpha }} \int_{(Q^+_r)^*}|(I-A(x)) \nabla_x W|^2 \, dt \, dX \\
    &\leq \sup_{0 < r \leq 1/2} \frac{C}{r^{n+3  + 2\alpha }} \int_{(Q^+_r)^*} |(I-A(x))|^2|\log y|^2 \, dt\,dX \\
    &\leq \sup_{0 < r \leq 1/2} \frac{C}{r^{n+3  + 2\alpha }} \int_{(Q^+_r)^*} |(I-A(x))|^2y^{-2\varepsilon} \, dt\,dX \\
    &= \sup_{0 < r \leq 1/2} \frac{C}{r^{n+2(\alpha+\varepsilon)}} \int_{B^+_r} |(I-A(x))|^2\,dx<C\delta^2.
\end{align*}
By scaling and normalization, we can assume that $V$ is a $\delta$-normalized solution to \eqref{eq:bdd_exten_mod} in the sense that
\begin{enumerate}
    \item $\displaystyle \sup_{0 < r \leq 1/2}\frac{1}{r^{n+2(\alpha+\varepsilon)}} \int_{B_r^+} |A(x)-I|^2\, dx < \delta^2 $;
    \item $[h]^2_{L^{\alpha/2,\alpha}_+(0,0)} := \displaystyle \sup_{0 < r \leq 1/2} \frac{1}{r^{n+2+2 \alpha}} \int_{Q^+_r} |h|^2\,dt\, dx < \delta^2$;
    \item $\displaystyle [F]^2_{\alpha, 1/2} = \sup_{0 < r \leq 1/2} \frac{1}{r^{n+3+2
    \alpha}} \int_{(Q_r^+)^*} |F|^2 \,dt\,dX < \delta^2$;
    \item $\displaystyle\int_{Q_1^+} V(t,x,0)^2 \,dt\,dx + \int_{(Q_1^+)^*}  V^2\,dt\,dX \leq 1$.
\end{enumerate}
Then we follow the proof of Theorem  \ref{thm:thm_1_holder}(2).
We have a linear polynomial $\ell(x)$ such that $V- \ell$ is a weak solution to 
\begin{align*}
\begin{cases}
      \partial_tV - \dvie(B(x)\nabla V) = -\dvie(F+G) \quad \hbox{in}~ (Q_{1/2}^+)^* \\
     -(V-\ell)_y|_{y=0} = h \quad \hbox{on}~ (Q_{1/2}^+) 
\end{cases}
\end{align*}
where the vector field $G$ is given by 
$$G = ((I-A(x))\nabla_x \ell,0) \quad \hbox{and} \quad G(0)=0 $$
Then we can see that $G$ also satisfies the same Campanato-type condition as $F$. Indeed, as $|\nabla \ell| \leq C$, 
\begin{align*}
   [G]^2_{\alpha, 1/2}   &=\sup_{0 < r \leq 1/2} \frac{1}{r^{n+3  + 2\alpha}} \int_{(Q^+_r)^*}  |(I-A(x))\nabla_x \ell|^2 \, dt \, dX \\
   &\leq \sup_{0 < r \leq 1/2} \frac{C}{r^{n  + 2\alpha}} \int_{B^+_r}  |(I-A(x))|^2  \, dx\leq C \delta^2.
\end{align*}
With this we can continue as in the proof of Theorem \ref{thm:thm_1_holder}(2) and get $l_{\infty}(x) = \mathcal{B}_{\infty} \cdot x$ such that 
\begin{equation}\nonumber
  \frac{1}{r^{n+2}} \int_{Q^+_r} |V(t,x,0)-l_{\infty}(x)|^2 dt \, dx \leq C_1 r^{2(\alpha + 2s)}  
\end{equation} for $r>0$ sufficiently small. As in Theorem \ref{thm:bdd_holder_1_f_non-zero}(1),
$$C^{1/2}_1 + |\mathcal{B}_\infty|\leq C_0 \big(1 + \|u\|_{\Dom(H^s)} + |f(0,0)| + [f]_{L^{\alpha/2,\alpha}_+(0,0)} \big) $$
where $C_0$ depends on $\delta$, $n$, $s$, $\alpha$ and ellipticity.
In this particular case we observe that, the term $\mathcal{A}$ from Lemma \ref{lem:poly_step_1} will be $0$ because the our approximating function $W=0$ at the origin and hence $\mathcal{A}_\infty$ will be $0$. 

\subsection{Proof of Theorem \ref{thm:bdd_holder_1_f_non-zero}(3)}
Let $U$, $V$, $F$ and $h$ be as in the proof of Theorem \ref{thm:bdd_holder_1_f_non-zero}(1). 
Observe that, by Lemma \ref{lem:nablaW}, $F$ satisfies the following Campanato-type condition:
\begin{align*}
    [F]^2_{\alpha, s} & \leq \sup_{0 < r \leq 1} \frac{C}{r^{n+3+ a  + 2(\alpha+2s-1)}} \int_{(Q^+_r)^*} y^a|(I-A(x))|^2 \, dt \, dX \\
    &\leq \sup_{0 < r \leq 1} \frac{C}{r^{n + 2(\alpha+2s-1)}} \int_{B^+_r} |(I-A(x))|^2 \, dx \leq C \delta^2.
\end{align*}
Then again we can normalize $V$ and follow the proof of Theorem \ref{thm:thm_1_holder}(2). Details are left to the interested reader.

\subsection{Proof of Theorem \ref{thm:bdd_holder_1_f_zero}} The proof is very similar to the proof of Theorem \ref{thm:thm_1_holder} with minor changes. If we replace $Q_r$ by $Q_r^+$ and follow the other steps then we get our result. 

\subsection{Proof of Theorem \ref{thm:bdd_holder_neu}} We prove the regularity of the solution for the extension problem about the origin like we did in the case of
Dirichlet boundary condition. The extension problem is 
$$\begin{cases}
y^a \partial_t U - \dvie(y^a B(x) \nabla U)= 0&\hbox{in}~(Q^+_1)^* \\ 
-y^aU_y|_{y=0} = f&\hbox{on}~Q^+_1 \\ 
\partial_A U = 0&\hbox{on}~Q^*_1 \cap \{ x_n = 0 \} .
\end{cases}$$
Then the proof follows the similar steps as in the proof of Theorem \ref{thm:thm_1_holder} except we need to replace the $Q_r$ by $Q^+_r$.

\subsection{Proof of Theorem \ref{thm:bdd_holder_2}}The proof follows very similar lines to those for Theorem \ref{thm:thm_2_holder} with minor changes, by replacing $Q_r$ by $Q^+_r$. 

%%%%%%%%%%%%%%%%%%%%%%%%%%%%%%%%%%%%%%%%%%%%%%%%%%%%%%
\section{Proof of Theorem \ref{thm:Campanato}}\label{section:proofofCampanato}
%%%%%%%%%%%%%%%%%%%%%%%%%%%%%%%%%%%%%%%%%%%%%%%%%%%%%%

In this last Section, we prove the Campanato characterization of parabolic H\"older spaces, Theorem \ref{thm:Campanato} (see Theorem \ref{thm:final}).

%%%%%%%%%%%%%%%%%%%%%%%%%%%%%%%%%%%%%%%%%%%%%%%%%%%%%%
\subsection{Proof of Theorem \ref{thm:Campanato}$(1)$}\label{section:proofofCampanato_1}
%%%%%%%%%%%%%%%%%%%%%%%%%%%%%%%%%%%%%%%%%%%%%%%%%%%%%%

This is a classical result, see, for example, \cite{Lieberman,Schlag}. 

%%%%%%%%%%%%%%%%%%%%%%%%%%%%%%%%%%%%%%%%%%%%%%%%%%%%%
\subsection{Proof of Theorem \ref{thm:Campanato}$(2)$}
%%%%%%%%%%%%%%%%%%%%%%%%%%%%%%%%%%%%%%%%%%%%%%%%%%%%%

We have the following preliminary result.
\begin{lem}\label{lem:lemmaCampanato}
There exists a constant $c=c_{n,\Omega}>0$ such that for any $P(z)\in\mathcal{P}_1$,
$(t_0,x_0)\in\overline{I\times\Omega}$ and $0<r\leq r_0$,
$$|P(x_0)|^2\leq\frac{c}{r^{n+2}}\int_{Q_r(t_0,x_0)\cap(I\times\Omega)}|P(z)|^2\,dz\,d\tau.$$
and, for any $i=1,\ldots,n$,
$$|\partial_{z_i}P(x_0)|^2\leq\frac{c}{r^{n+2+2}}\int_{Q_r(t_0,x_0)\cap(I\times\Omega)}
|P(z)|^2\,dz\,d\tau.$$
\end{lem}

\begin{proof}
Observe that if $\beta$ is a multi-index with $|\beta|\leq1$ then
\begin{align*}
\frac{1}{r^{n+2+2|\beta|}}\int_{Q_r(t_0,x_0)\cap(I\times\Omega)}
|P(z)|^2\,dz\,d\tau 
&= \frac{|(t_0-r^2,t_0+r^2)\cap I|}{r^{n+2+2|\beta|}}\int_{B_r(x_0)\cap\Omega}|P(z)|^2\,dz \\
&\geq  \frac{1}{2r^{n+2|\beta|}}\int_{B_r(x_0)\cap\Omega}|P(z)|^2\,dz.
\end{align*}
Notice that there is a constant $A=A_\Omega>0$ such that $|E|=|B_r(x_0)\cap\Omega|\geq Ar^n$.
Then, by \cite[Lemma 2.I]{Campanato}, there is a constant $c>0$, depending only on $n$ and $A$ such that
$$\frac{c}{r^{n+2|\beta|}}\int_{B_r(x_0)\cap\Omega}|P(z)|^2\,dz\geq|D^\beta P(x_0)|^2.$$
\end{proof}

It is easy to see that the infimum for the integral quantity in \eqref{eq:condition2}
is achieved at a unique polynomial (see \cite{Campanato}). Therefore, \eqref{eq:condition2} is restated as follows:
\emph{for any $(t,x)\in\overline{I\times\Omega}$
and $0<r\leq r_0$ there is a unique polynomial $P(z,(t,x),r,u)\in\mathcal{P}_1$ such that}
$$\frac{1}{|Q_r(t,x)\cap(I\times\Omega)|}\int_{Q_r(t,x)\cap(I\times\Omega)}|u(\tau,z)-P(z,(t,x),r,u)|^2\,d\tau\,dz\leq Cr^{2(1+\beta)}.$$

A generic polynomial $P\in\mathcal{P}_1$ is written as
$$P(z)=a_0+\sum_{j=1}^na_j(z_j-x_j).$$
For the unique polynomial $P(z,(t,x),r)\equiv P(z,(t,x),r,u)$ above we have
$$a_0((t,x),r)=P(z,(t,x),r)\big|_{z=x}$$
and
$$a_i((t,x),r)=\big[\partial_{z_i}P(z,(t,x),r)\big]\Big|_{z=x}\qquad\hbox{for}~i=1,\ldots,n.$$

\begin{lem}\label{lem:succesive1}
Let $u$ satisfy \eqref{eq:condition2}. There exists $c=c(n,\beta)>0$ such that for any $(t_0,x_0)\in\overline{I\times\Omega}$, $0<r\leq r_0$
and $k\geq0$, we have
\begin{multline*}
\int_{Q_{r/2^{k+1}}(t_0,x_0)\cap(I\times\Omega)}|P(z,(t_0,x_0),r/2^k)-P(z,(t_0,x_0),r/2^{k+1})|^2\,d\tau\,dz\\
\leq C_{\ast\ast}c(r/2^k)^{n+2+2(1+\beta)}.
\end{multline*}
\end{lem}

\begin{proof}
We have
\begin{align*}
\int_{Q_{r/2^{k+1}}(t_0,x_0)\cap(I\times\Omega)}&|P(z,(t_0,x_0),r/2^k)-P(z,(t_0,x_0),r/2^{k+1})|^2\,d\tau\,dz \\
&\leq 2\int_{Q_{r/2^{k}}(t_0,x_0)\cap(I\times\Omega)}|P(z,(t_0,x_0),r/2^k)-u(\tau,z)|^2\,d\tau\,dz\\
&\quad+2\int_{Q_{r/2^{k+1}}(t_0,x_0)\cap(I\times\Omega)}|u(\tau,z)-P(z,(t_0,x_0),r/2^{k+1})|^2\,d\tau\,dz \\
&\leq C_{\ast\ast}c(r/2^k)^{n+2+2(1+\beta)}.
\end{align*}
\end{proof}

\begin{lem}\label{lem:diff_cen}
Let $u$ satisfy \eqref{eq:condition2}. There exists $c=c(n,\beta)>0$ such that for any $(t_0,x_0),(s_0,y_0)\in\overline{I\times\Omega}$,
if we denote by $d_0=\max(|t_0-s_0|^{1/2},|x_0-y_0|)\leq r_0$, then
$$|a_0((t_0,x_0),2d_0)-a_0((s_0,x_0),2d_0)|^2\leq cC_{\ast\ast}|t_0-s_0|^{1+\beta}$$
and, for $i=1,\ldots,n$,
$$|a_i((t_0,x_0),2d_0)-a_i((s_0,y_0),2d_0)|^2\leq cC_{\ast\ast}d_0^{2\beta}.$$
\end{lem}

\begin{proof}
Consider first the case $i=0$ and the polynomial
$$P(z)\equiv P(z,(t_0,x_0),2d_0)-P(z,(s_0,x_0),2d_0).$$
By Lemma \ref{lem:lemmaCampanato} with $r=d_0$,
\begin{align*}
|a_0((t_0,x_0),2d_0)&-a_0((s_0,x_0),2d_0)|^2 = |P(x_0,(t_0,x_0),2d_0)-P(x_0,(s_0,x_0),2d_0)|^2 \\
&\leq \frac{c}{d_0^{n+2}}\int_{Q_{d_0}(t_0,x_0)\cap(I\times\Omega)}|P(z,(t_0,x_0),2d_0)-P(z,(s_0,x_0),2d_0)|^2\,d\tau\,dz \\
&\leq \frac{2c}{d_0^{n+2}}\int_{Q_{2d_0}(t_0,x_0)\cap(I\times\Omega)}|P(z,(t_0,x_0),2d_0)-u(\tau,z)|^2\,d\tau\,dz \\
&\quad +\frac{2c}{d_0^{n+2}}\int_{Q_{2d_0}(s_0,x_0)\cap(I\times\Omega)}|u(\tau,z)-P(z,(s_0,x_0),2d_0)|^2\,d\tau\,dz \\
&\leq cC_{\ast\ast}d_0^{2(1+\beta)}=cC_{\ast\ast}|t_0-s_0|^{1+\beta}.
\end{align*}

For $i=1,\ldots,n$, the proof is similar using Lemma \ref{lem:lemmaCampanato}.
\end{proof}

\begin{lem}\label{lem:coeff_gap}
Let $u$ satisfy \eqref{eq:condition2}. There exists $c=c(n,\beta,\Omega)>0$ such that for any $(t_0,x_0)\in\overline{I\times\Omega}$,
$0<r\leq r_0$ and $k\geq0$,
$$|a_0((t_0,x_0),r)-a_0((t_0,x_0),r/2^k)|\leq c(C_{\ast\ast})^{1/2}\sum_{j=0}^{k-1}(r/2^j)^{1+\beta}$$
and, for $i=1,\ldots,n$,
$$|a_i((t_0,x_0),r)-a_i((t_0,x_0),r/2^k)|\leq c(C_{\ast\ast})^{1/2}\sum_{j=0}^{k-1}(r/2^j)^{\beta}.$$
\end{lem}

\begin{proof}
By applying Lemmas \ref{lem:lemmaCampanato} and \ref{lem:succesive1}, for $i=1,\ldots,n$,
\begin{align*}
&|a_i((t_0,x_0),r)-a_i((t_0,x_0),r/2^k)| \leq \sum_{j=0}^{k-1}|a_i((t_0,x_0),r/2^j)-a_i((t_0,x_0),r/2^{j+1})| \\
&= \sum_{j=0}^{k-1}|\partial_{z_i}P(x_0,(t_0,x_0),r/2^j)-\partial_{z_i}P(x_0,(t_0,x_0),r/2^{j+1})| \\
&\leq c\sum_{j=0}^{k-1}\bigg[\frac{c}{(r/2^{j+1})^{n+2+2}}\int_{Q_{r/2^{j+1}}(t_0,x_0)\cap(I\times\Omega)}|P(z,(t_0,x_0),r/2^j)-P(z,(t_0,x_0),r/2^{j+1})|^2\,dz\,d\tau\bigg]^{1/2} \\
&\leq c(C_{\ast\ast})^{1/2}\sum_{j=0}^{k-1}(r/2^j)^\beta.
\end{align*}
The case $i=0$ follows the same lines.
\end{proof}

\begin{lem}\label{lem:vi}
Let $u$ satisfy \eqref{eq:condition2}.
Then there exists a family of functions $\{v_i(t,x)\}_{i=0}^n$ defined in $\overline{I\times\Omega}$
such that for all $0 < r \leq r_0$,
$$|a_0((t_0,x_0),r) - v_0(t_0,x_0)| \leq C (C_{\ast \ast})^{1/2} r^{1+\beta}$$ 
and, for all $i=1, 2,\ldots, n$,
$$|a_i((t_0,x_0),r) - v_i(t_0,x_0)| \leq C (C_{\ast \ast})^{1/2} r^{\beta}.$$ 
Moreover, for all $i=0,1,\ldots,n$,
$$\lim_{r \to 0}a_i((t_0,x_0),r) = v_i(t_0,x_0)$$ 
uniformly with respect to $(t_0,x_0)$.
\end{lem}

\begin{proof}
Using Lemma \ref{lem:coeff_gap}, for $i=1,2,\ldots,n$, if $j<k$ then we find that 
$$|a_i((t_0,x_0),r/2^j)-a_i((t_0,x_0),r/2^k)| \leq c (C_{\ast \ast})^{1/2} \sum_{m=j}^{k-1} (r/2^m)^\beta.$$
If $j, k$ are large then the sum above can be made very small.
Hence the limit 
\begin{equation}\label{eq:LIMIT}
\lim_{k\to\infty}a_i((t_0,x_0),r/2^k)=v_i(t_0,x_0)
\end{equation}
exists. We claim that the limit does not depend on $r$. Indeed, let $0<r_1<r_2 <r_0$. Then we have,
\begin{align*}
  |a_i&((t_0,x_0),r_1/2^k)-a_i((t_0,x_0),r_2/2^k)|^2\\
  &= |\partial_{z_i}P(x_0,(t_0,x_0),r_1/2^k)-\partial_{z_i}P(x_0,(t_0,x_0),r_2/2^k))|^2 \\ 
&\leq \frac{c 2^{k(n+4)}}{r^{n+4}_1}\int_{Q_{r_1/2^k}(t_0,x_0)\cap(I\times\Omega)}|P(z,(t_0,x_0),r_1/2^k)-P(z,(t_0,x_0),r_2/2^k)|^2\,d\tau\,dz \\ 
&\leq \frac{c 2^{k(n+4)+2}}{r^{n+4}_1}\int_{Q_{r_1/2^k}(t_0,x_0)\cap(I\times\Omega)}|P(z,(t_0,x_0),r_1/2^k)-u(\tau,z)|^2\,d\tau\,dz \\
&\quad+ \frac{c 2^{k(n+4)+2}}{r^{n+4}_1}\int_{Q_{r_1/2^k}(t_0,x_0)\cap(I\times\Omega)}|u(\tau,z)-P(z,(t_0,x_0),r_2/2^k)|^2\,d\tau\,dz \\
& \leq \frac{c 2^{2k}}{r^2_1}\int_{Q_{r_1/2^k}(t_0,x_0)\cap(I\times\Omega)}|P(z,(t_0,x_0),r_1/2^k)-u(\tau,z)|^2\,d\tau\,dz \\
&+ \frac{c 2^{2k}r^{n+2}_2}{r^{n+4}_1}\int_{Q_{r_2/2^k}(t_0,x_0)\cap(I\times\Omega)}|u(\tau,z)-P(z,(t_0,x_0),r_2/2^k)|^2\,d\tau\,dz \\
&\leq cC_{\ast\ast} \left(\frac{r_1}{2^k}\right)^{2\beta} + cC_{\ast\ast} \left(\frac{1}{2^k}\right)^{2\beta} \frac{r_2^{n+2+2(1+\beta)}}{r_1^{n+4}}= \frac{ cC_{\ast\ast}}{2^{2k \beta}}\frac{r_1^{n+4+2\beta}+r_2^{n+4+2\beta}}{r_1^{n+4}}
\end{align*}
Hence,
$$\lim_{k\to\infty}|a_i((t,x),r_1/2^k)-a_i((t,x),r_2/2^k)|=0$$
and the limit \eqref{eq:LIMIT} does not depend on $r$.
Now, recall that we have
$$|a_i((t_0,x_0),r)-a_i((t_0,x_0),r/2^k)|\leq c(C_{\ast\ast})^{1/2}\sum_{j=0}^{k-1}(r/2^j)^{\beta}$$ 
Then taking the limit $k \to \infty$, $|a_i((t_0,x_0),r)-v_i(t_0,x_0)| \leq c(C_{\ast\ast})^{1/2} r^{\beta}$.
For $i=0$, the proof is the same.  
\end{proof}

\begin{thm}\label{thm:holder}
Let $u$ satisfy \eqref{eq:condition2} and define $v_i$ as in Lemma \ref{lem:vi} for $i=1,2,\ldots,n$. Then 
$v_i$ is in $C^{\beta/2,\beta}_{t,x}(\overline{I\times\Omega})$ and for every $(t,x),(s,y)\in\overline{I\times\Omega}$
we have
$$|v_i(t,x)-v_i(s,y)|\leq C(C_{\ast\ast})^{1/2}\max(|t-s|^{1/2},|x-y|)^\beta.$$
\end{thm}
\begin{proof}
Let $(t,x),(s,y)\in\overline{I\times\Omega}$ such that $d=\max(|t-s|^{1/2},|x-y|)<r_0/2$. Then,
by Lemmas \ref{lem:diff_cen} and \ref{lem:vi},
\begin{align*}
|v_i(t,x)-v_i(s,y)| &\leq |v_i(t,x)-a_i((t,x),2d)|\\
&\quad +|v_i(s,y)-a_i((s,y),2d)|+|a_i((t,x),2d)-a_i((s,y),2d)| \\
&\leq C(C_{\ast\ast})^{1/2}d^{\beta}= C(C_{\ast\ast})^{1/2}\max(|t-s|^{1/2},|x-y|)^\beta.
\end{align*}
In the case when $d=\max(|t-s|^{1/2},|x-y|)\geq r_0/2$, then we can construct a polygonal connecting $(t,x)$ and $(s,y)$, contained in $\overline{I\times\Omega}$, whose segments have length less than $r_0/2$. After that we can apply the inequality above to each pair of consecutive vertices. Again notice that the number of segments
needed for any pair of points $(t,x)$ and $(s,y)$ can be universally bounded in terms of the size of $I\times\Omega$, see \cite[p.~149]{Campanato_2}.
\end{proof}

\begin{thm}\label{thm:derivative}
Let $u$ satisfy \eqref{eq:condition2} and define $v_i$ as in Lemma \ref{lem:vi} for $i=0,1,\ldots,n$.
Then, for every $(t,x) \in I \times \Omega$ 
$$ \frac{\partial v_0(t,x)}{\partial x_i} = v_i(t,x) \quad\hbox{for}~i=1,\ldots,n.$$
\end{thm}

\begin{proof}
Let $(t,x) \in I \times \Omega$ be any point and $r>0$ sufficiently small such that $Q_r(t,x) \subset I \times \Omega$. Now we see that 
$$a_0((t,x+re_i),2r)= P(z,(t,x+re_i),2r)\big|_{z= x+ re_i}.$$
Using Taylor series expansion we can write,
\begin{align*}
P(z,(t,x+re_i),2r)\big|_{z=x} &= P(z,(t,x+re_i),2r)\big|_{z= x+ re_i} - \partial_{z_i}P(z,(t,x+re_i),2r)\big|_{z= x+ re_i}r \\
&= a_0((t,x+re_i),2r) - ra_i((t,x+re_i),2r).
\end{align*}
Then,
\begin{multline}\label{eq:PRELIMIT}
\frac{a_0((t,x+re_i),2r)-a_0((t,x),2r)}{r} \\ 
=\frac{P(z,(t,x+re_i),2r)\big|_{z=x} - P(z,(t,x),2r)\big|_{z=x}}{r} + a_i((t,x+re_i),2r)
\end{multline}
Now using Lemma \ref{lem:lemmaCampanato} we see that 
\begin{align*}
\bigg| &\frac{P(z,(t,x+re_i),2r)\big|_{z=x} - P(z,(t,x),2r)\big|_{z=x}}{r} \bigg|^2\\
&= \frac{1}{r^2} \left|  P(z,(t,x+re_i),2r)\big|_{z=x} - P(z,(t,x),2r)\big|_{z=x} \right|^2 \\ 
&\leq  \frac{c}{r^{2+n+2}} \int_{Q_{r}(t,x)\cap(I\times\Omega)}|P(z,(t,x+re_i),2r) - P(z,(t,x),2r)|^2\,d\tau\,dz \\ 
&\leq  \frac{c}{r^{2+n+2}} \int_{Q_{2r}(t,x+re_i)\cap(I\times\Omega)}|P(z,(t,x+re_i),2r) -u(\tau,z)|^2\,d\tau\,dz \\ 
&\quad+ \frac{c}{r^{2+n+2}} \int_{Q_{r}(t,x)\cap(I\times\Omega)}|u(\tau,z) - P(z,(t,x),2r)|^2\,d\tau\,dz \\ 
&\leq \frac{cC_{\ast\ast}}{r^{n+4}} r^{n+2 + 2(1+\beta)} = cC_{\ast\ast}r^{2 \beta}\to0
\end{align*}
as $r\to0$.
Next we see that, by Lemma \ref{lem:vi} and since $v_i$ are continuous functions (see Theorem \ref{thm:holder}),
\begin{align*}
    |a_i((t,x+re_i),2r)-v_i(t,x)| &\leq |a_i((t,x+re_i),2r)-v_i(t,x+re_i)| \\ 
    &\quad+ |v_i(t,x+re_i)-v_i(t,x)| \\
    &\leq c(C_{\ast\ast})^{1/2} r^{\beta} + |v_i(t,x+re_i)-v_i(t,x)|\to0
\end{align*}
as $r\to0$. Thus, it follows in \eqref{eq:PRELIMIT} that
$$\lim_{r\to0}\frac{a_0((t,x+re_i),2r)-a_0((t,x),2r)}{r}=v_i(t,x).$$
But now observe that
$$\lim_{r\to0}\frac{a_0((t,x+re_i),2r)-a_0((t,x),2r)}{r}=\lim_{r\to0}\frac{v_0(t,x+re_i)-v_0(t,x)}{r}=\partial_{x_i}v_0(t,x)$$
because, by Lemma \ref{lem:vi}, 
$$\left|\frac{v_0(t,x+re_i)-a_0((t,x+re_i),2r)}{r} \right| \leq c(C_{\ast\ast})^{1/2}r^{\beta}$$
and 
$$\left|\frac{v_0(t,x)-a_0((t,x),2r)}{r} \right| \leq c(C_{\ast\ast})^{1/2}r^{\beta}.$$
\end{proof}

The following result is a direct consequence of Theorems \ref{thm:holder} and \ref{thm:derivative}.

\begin{cor}\label{cor:redundant}
Let $u$ satisfy \eqref{eq:condition2}. If $v_0$ is as in Lemma \ref{lem:vi}
then $v_0\in C^{(1+\beta)/2,1+\beta}_{t,x}(\overline{I\times\Omega})$ with the estimate
$$[v_0]_{L^\infty_x(C^{(1+\beta)/2}_t)}+[\nabla v_0]_{C^{\beta/2,\beta}_{t,x}}
\leq c(C_{\ast\ast})^{1/2}.$$
\end{cor}
\begin{proof}
Let $(t,x),(s,x)\in\overline{I\times\Omega}$ such that $d=|t-s|^{1/2}<r_0/2$. Then,
by Lemmas \ref{lem:diff_cen} and \ref{lem:vi},
\begin{align*}
|v_0(t,x)-v_0(s,x)| &\leq |v_0(t,x)-a_0((t,x),2d)|\\
&\quad +|v_0(s,x)-a_0((s,t),2d)|+|a_0((t,x),2d)-a_0((s,x),2d)| \\
&\leq C(C_{\ast\ast})^{1/2}d^{1+ \beta}= C(C_{\ast\ast})^{1/2}|t-s|^{(1+\beta)/2}.
\end{align*}
In the case when $d>r_0/2$ we can apply a polygonal argument as in \cite[p.~149]{Campanato}.
Also we have already shown that, $v_i = \frac{\partial v_0}{\partial z_i}$ is in $C^{\beta/2,\beta}_{t,x}$ for each $i=1,2,\ldots,n$ and
$$|v_i(t,x)-v_i(s,y)|\leq C(C_{\ast\ast})^{1/2}\max(|t-s|^{1/2},|x-y|)^\beta.$$
See Theorems \ref{thm:derivative} and \ref{thm:holder}.
Hence by definition of $C^{(1+\beta)/2, 1+\beta}_{t,x}$ we have 
$$v_0 \in C^{(1+\beta)/2, 1+\beta}_{t,x}(\overline{I\times\Omega})$$
with the corresponding estimate.
\end{proof}

\begin{thm}\label{thm:final}
Let $u$ satisfy \eqref{eq:condition2}. Then $u\in C^{(1+\beta)/2,1+\beta}_{t,x}(\overline{I\times\Omega})$ with the estimates
$$[u]_{L^\infty_x(C^{(1+\beta)/2}_t)}+[\nabla u]_{C^{\beta/2,\beta}_{t,x}}\leq c(C_{\ast\ast})^{1/2}$$
and
$$\|u\|_{L^\infty(I\times\Omega)}+\|\nabla u\|_{L^\infty(I\times\Omega)}\leq c\big((C_{\ast\ast})^{1/2}+\|u\|_{L^2(I\times\Omega)}\big).$$
\end{thm}

\begin{proof}
For any $(t_0,x_0) \in I \times \Omega$, we have, by Lebesgue differentiation theorem,
$$\lim_{r \to 0}\frac{1}{|Q_r(t_0,x_0)\cap(I\times\Omega)|}\int_{Q_r(t_0,x_0)\cap(I\times\Omega)}|u(t,x)-u(t_0,x_0)|^2\,dt\,dx = 0$$
see \cite{Calderon}. Then, for any $0<r\leq r_0$,
\begin{align*}
    |a_0((t_0,x_0),r)-u(t_0,x_0)|^2 &\leq  \frac{C}{r^{n+2}}\int_{Q_r(t_0,x_0)\cap(I\times\Omega)}|P(x,(t_0,x_0),r)-a_0((t_0,x_0),r)|^2 \,dt\,dx  \\
                                     &\quad  + \frac{C}{r^{n+2}}\int_{Q_r(t_0,x_0)\cap(I\times\Omega)} |P(x,(t_0,x_0),r)-u(t,x)|^2 \,dt\,dx \\
                                     &\quad + \frac{C}{r^{n+2}}\int_{Q_r(t_0,x_0)\cap(I\times\Omega)} |u(t,x)-u(t_0,x_0)|^2 \,dt\,dx.
\end{align*}
Now, using \eqref{eq:condition2} and the following equation,
$$P(x,(t_0,x_0),r) = a_0((t_0,x_0),r) + \sum^n_{j=1} a_j((t_0,x_0),r)(x_j-(x_0)_j) $$
We get 
$$ \frac{1}{r^{n+2}}\int_{Q_r(t_0,x_0)\cap(I\times\Omega)}|P(x,(t_0,x_0),r)-a_0((t_0,x_0),r)|^2 \,dt\,dx \leq C\sum^n_{j=1} |a_j((t_0,x_0),r)|^2 r^2.$$
For a fixed $(t_0,x_0)$, $|a_j((t_0,x_0),r)|^2$ converges as $r\to0$, see Lemma \ref{lem:vi}.
Hence, as $r \to 0$, using all the previous results and estimates we see that
$$v_0(t_0,x_0)=\lim_{r \to 0} a_0((t_0,x_0),r)= u(t_0,x_0).$$
Therefore, $u$ can be modified on a set of measure zero so that $u=v_0$. In particular,
by Theorem \ref{thm:derivative}, $u$ is differentiable in $I\times\Omega$ and, by using Corollary \ref{cor:redundant}, seminorm estimates follow.
For the boundedness of $u$ and $\nabla u$, we use Lemmas \ref{lem:lemmaCampanato}
and \ref{lem:vi} to bound in the following way. On one hand, 
\begin{align*}
|u(t,x)|^2 &\leq C|u(t,x)-a_0((t,x),r_0)|^2+C|a_0((t,x),r_0)|^2 \\
&= C|v_0(t,x)-a_0((t,x),r_0)|^2+C|P(x,(t,x),r_0)|^2 \\
&\leq cC_{\ast\ast}r_0^{2(1+\beta)}+\frac{C}{r_0^{n+2}}\int_{Q((t,x),r_0)\cap(I\times\Omega)}|P(z,(t,x),r_0)|^2\,dz\,d\tau \\
&\leq cC_{\ast\ast}r_0^{2(1+\beta)}+\frac{C}{r_0^{n+2}}\int_{Q((t,x),r_0)\cap(I\times\Omega)}|P(z,(t,x),r_0)-u(\tau,z)|^2\,dz\,d\tau \\
&\quad +\frac{C}{r_0^{n+2}}\int_{I\times\Omega}|u(\tau,z)|^2\,dz\,d\tau \\
&\leq cC_{\ast\ast}r_0^{2(1+\beta)}+\frac{C}{r_0^{n+2}}\|u\|_{L^2(I\times\Omega)}^2.
\end{align*}
Similarly, 
\begin{align*}
|u_{x_i}(t,x)|^2 &\leq C|u_{x_i}(t,x)-a_i((t,x),r_0)|^2+C|a_i((t,x),r_0)|^2 \\
&\leq cC_{\ast\ast}r_0^{2\beta}+\frac{C}{r_0^{n+4}}\|u\|_{L^2(I\times\Omega)}^2.
\end{align*}
\end{proof}

\medskip

\noindent\textbf{Acknowledgments.} The authors are grateful to the referee for providing useful comments that
helped improve the presentation of the paper.

%%%%%%%%%%%%%%%%%%%%%%%%%%%%%%%%%%%%%%%%%%%%%%%%%%%%%%

%%%%%%%%%%%%%%%%%%%%%%%%%%%%%%%%%%%%%%%%%%%%%%%%%%%%%%

%%%%%%%%%%%%%%%%%%%%%%%%%%%%%%%%%%%%%%%%%%%%%%%%%%%%%%
\end{document}